\documentclass[11pt]{article}
\usepackage{amsmath,amssymb,amsthm,bm,graphicx,color,epsfig,enumerate,caption}

\usepackage{algorithm}
\usepackage{algorithmic}

\newtheorem{theorem}{Theorem}[section]
\newtheorem{lemma}[theorem]{Lemma}

\newtheorem{definition}[theorem]{Definition}

\newtheorem{algo}[theorem]{Algorithm}

\newcommand{\R}{\mathbb{R}}

\newcommand{\Z}{\mathbb{Z}}
\newcommand{\C}{\mathbb{C}}
\newcommand{\F}{\mathcal{F}}

\newcommand{\eps}{\varepsilon}

\newcommand{\SNR}{\text{SNR}}

\begin{document}

\title{Synchrosqueezed wave packet transforms and diffeomorphism based spectral analysis for 1D general mode decompositions}

\author{Haizhao Yang\\
  \vspace{0.1in}\\
  Department of Mathematics, Stanford University
}

\date{October 2013; revised April 2014}
\maketitle

\begin{abstract}
  This paper develops new theory and algorithms for 1D general mode decompositions. First, we introduce the 1D synchrosqueezed wave packet transform and prove that it is able to estimate instantaneous information of well-separated modes from their superposition accurately. The synchrosqueezed wave packet transform has a better resolution than the synchrosqueezed wavelet transform in the time-frequency domain for separating high frequency modes. Second, we present a new approach based on diffeomorphisms for the spectral analysis of general shape functions. These two methods lead to a framework for general mode decompositions under a weak well-separation condition and a well-different condition. Numerical examples of synthetic and real data are provided to demonstrate the fruitful applications of these methods.

\end{abstract}



\section{Introduction}
\label{sec:intro}
\subsection{Problem statement}
In signal processing, analyzing instantaneous properties (e.g., instantaneous frequencies, instantaneous amplitudes and instantaneous phases \cite{Boashash92,Picinbono97}) of signals has been an important topic for over two decades. In many applications \cite{Daubechies2011,Eng2,Hau-Tieng2013,Eng1,SSCT,SSWPT}, a signal would be a superposition of several components, for example, a complex signal
\begin{equation}
\label{P1}
f(t)=\sum_{k=1}^K \alpha_k(t) e^{2\pi i N_k \phi_k(t)},
\end{equation}
where $\alpha_k(t)$ is the instantaneous amplitude, $2\pi N_k \phi_k(t)$ is the instantaneous phase and $N_k\phi_k'(t)$ is the instantaneous frequency. One wishes to decompose the signal $f(t)$ to obtain each component $\alpha_k(t) e^{2\pi iN_k \phi_k(t)}$ and its corresponding instantaneous properties. This is referred to as the mode decomposition problem. 

\begin{figure}[ht!]
  \begin{center}
    \begin{tabular}{c}
      \includegraphics[height=2.4in]{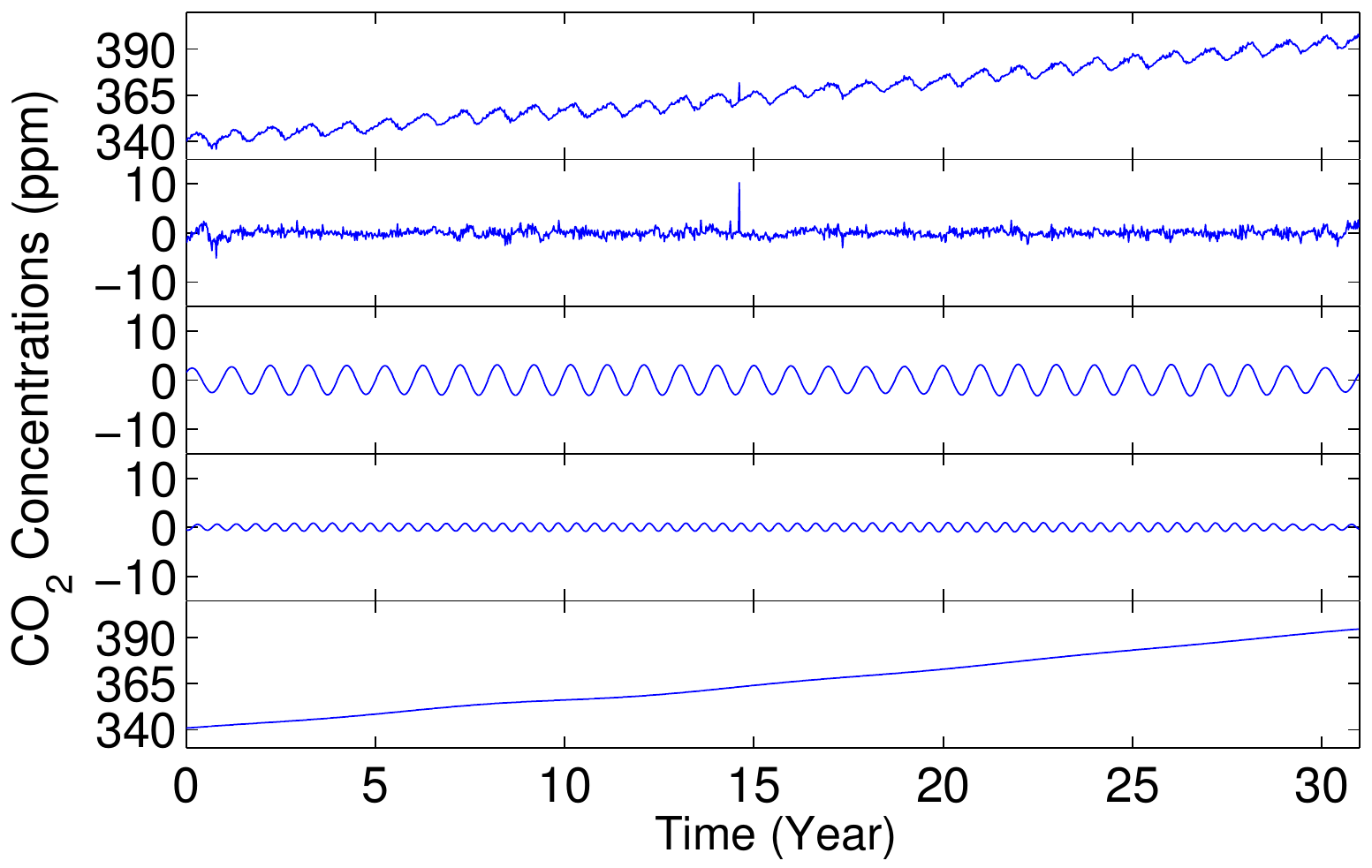} 
    \end{tabular}
  \end{center}
  \caption{The top signal is the observed CO$_2$ concentration of recent $31$ years (1981-2011) at MLO. Below the original signal are the components provided by the wavelet transform. Only relevant components are separated and presented.}
  \label{fig:CO2wavelet}
\end{figure}

In spite of considerable successes of analyzing signals by decomposing them in the form \eqref{P1}, a superposition of a few wave-like components belongs to a very limited class of oscillatory patterns. Most of all, decompositions in the form \eqref{P1} lose important physical information in some cases as detailed in \cite{Hau-Tieng2013,PhysicalAnal}. To be more concrete, we take the same daily atmospheric CO$_2$ concentration data in \cite{PhysicalAnal} as an example. It is observed by National Oceanic and Atmospheric Administration at Mauna Loa (MLO). The method based on wavelet transforms is capable of decomposing data in the form \eqref{P1}, providing one annual cycle, one semiannual cycle and a growing trend (see Figure \ref{fig:CO2wavelet}). However, each component alone cannot reflect the true nonlinear evolution pattern: the CO$_2$ concentration slowly increased in a longer period and quickly decreased in a shorter period. This special pattern is a result of seasonal photosynthetic drawdown and respiratory release of CO$_2$ by terrestrial ecosystems \cite{PhysicalAnal}. Fortunately, such a nonlinear evolution pattern can be recovered by summing up the annual cycle and the semiannual cycle as shown in Figure \ref{fig:CO2WaveShape}. This motivates the study of a more general decomposition of the form
\begin{equation}
\label{P2}
f(t)=\sum_{k=1}^K f_k(t) = \sum_{k=1}^K \alpha_k(t) s_k(2 \pi N_k \phi_k(t)),
\end{equation}
where $\{s_k(t)\}_{1\leq k\leq K}$ are $2\pi$-periodic general shape functions. By applying the Fourier expansion of general shape functions, the form \eqref{P2} is informally similar to the form \eqref{P1} with infinite terms, i.e., 
\begin{equation}
f(t)= \sum_{k=1}^K \alpha_k(t) s_k(2 \pi N_k \phi_k(t))=\sum_{k=1}^K \sum_{n=-\infty}^{\infty} \widehat{s_k}(n)\alpha_k(t) e^{2\pi i n N_k \phi_k(t)}.
\end{equation}
One could combine terms with similar oscillatory patterns in the form \eqref{P1} to obtain a more efficient and more meaningful decomposition in the form \eqref{P2}. This is the general mode decomposition problem discussed in this paper.

\begin{figure}
  \begin{center}
    \begin{tabular}{ccc}
      \includegraphics[height=1.6in]{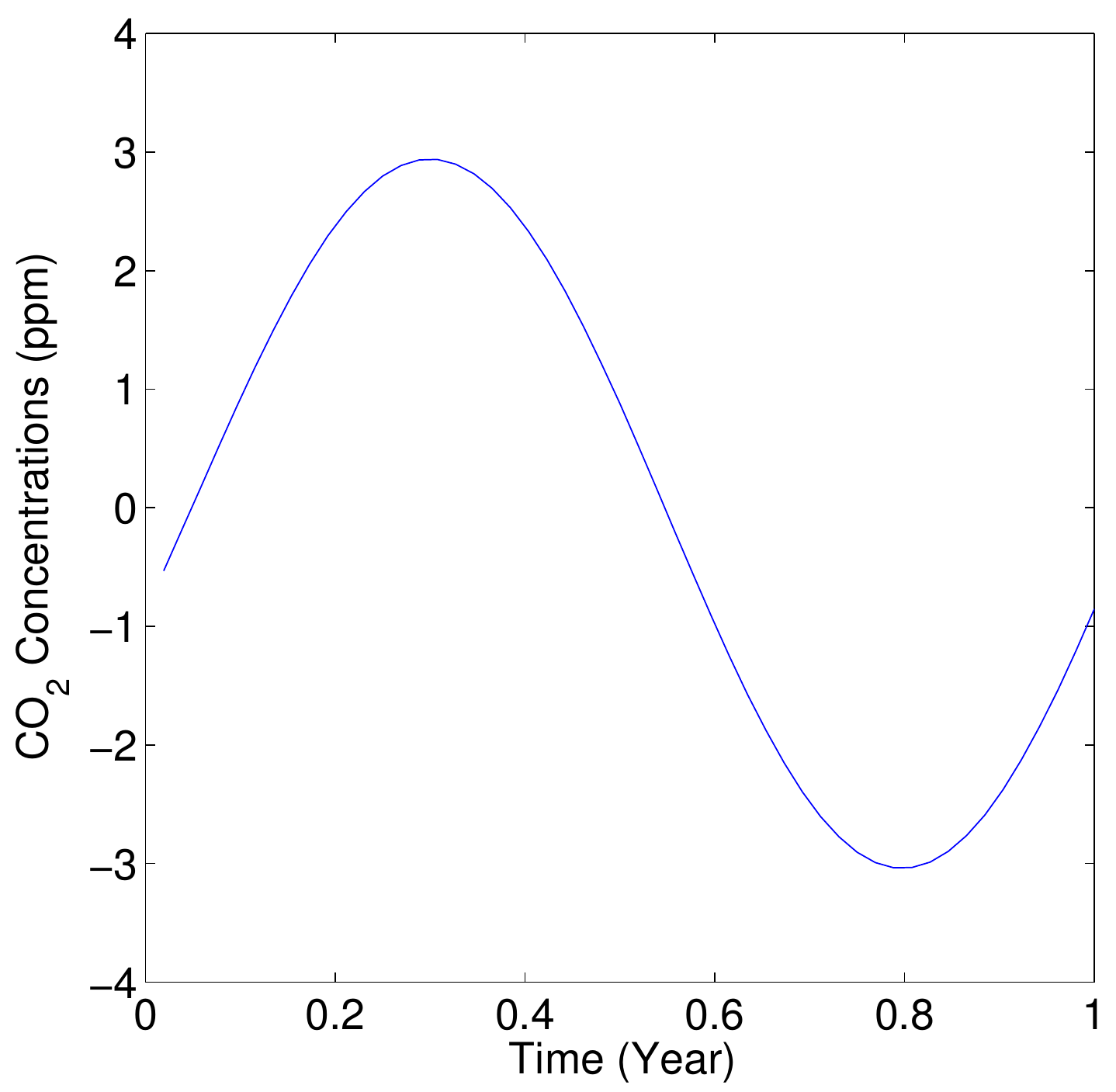} & \includegraphics[height=1.6in]{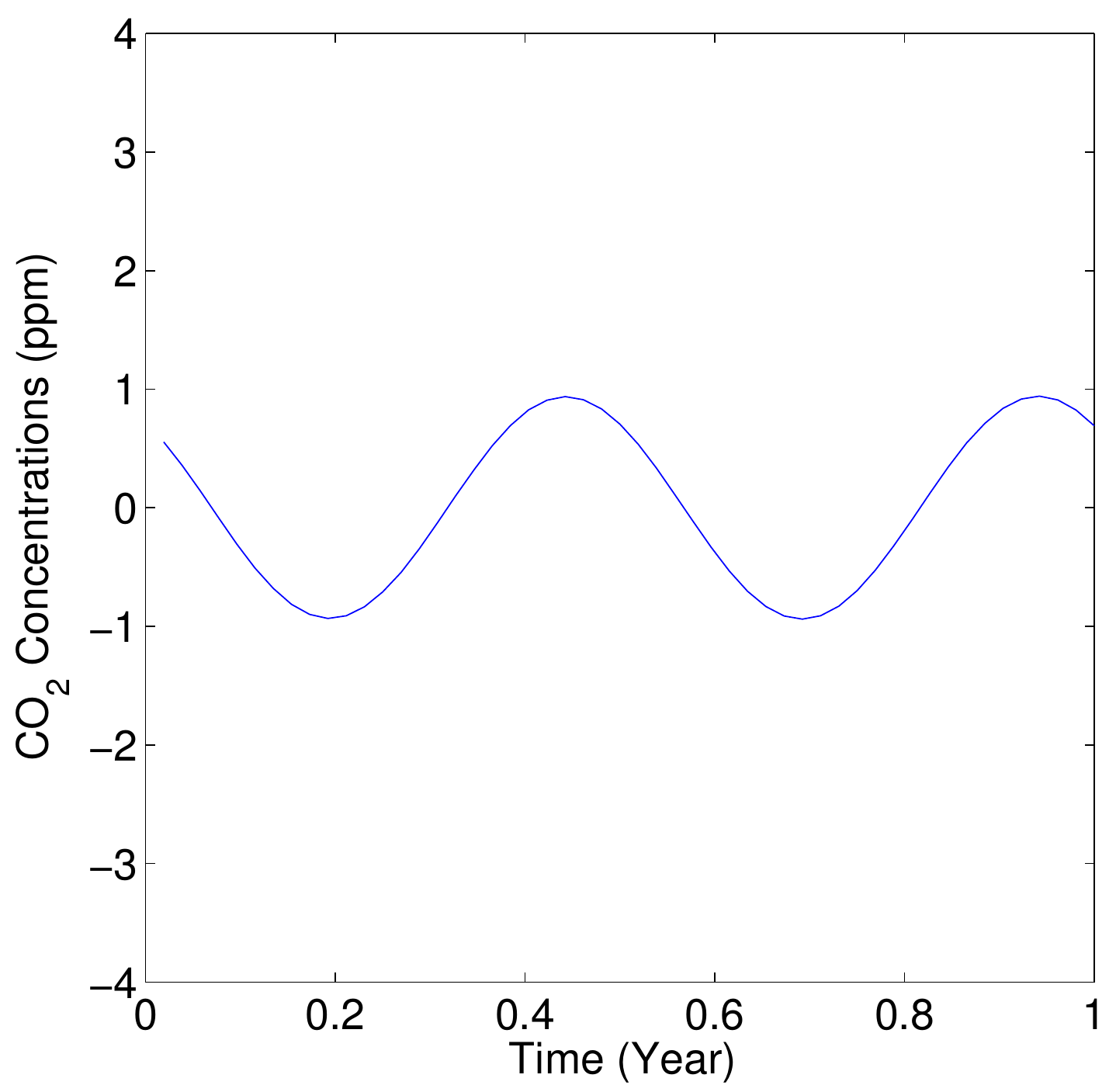} & \includegraphics[height=1.6in]{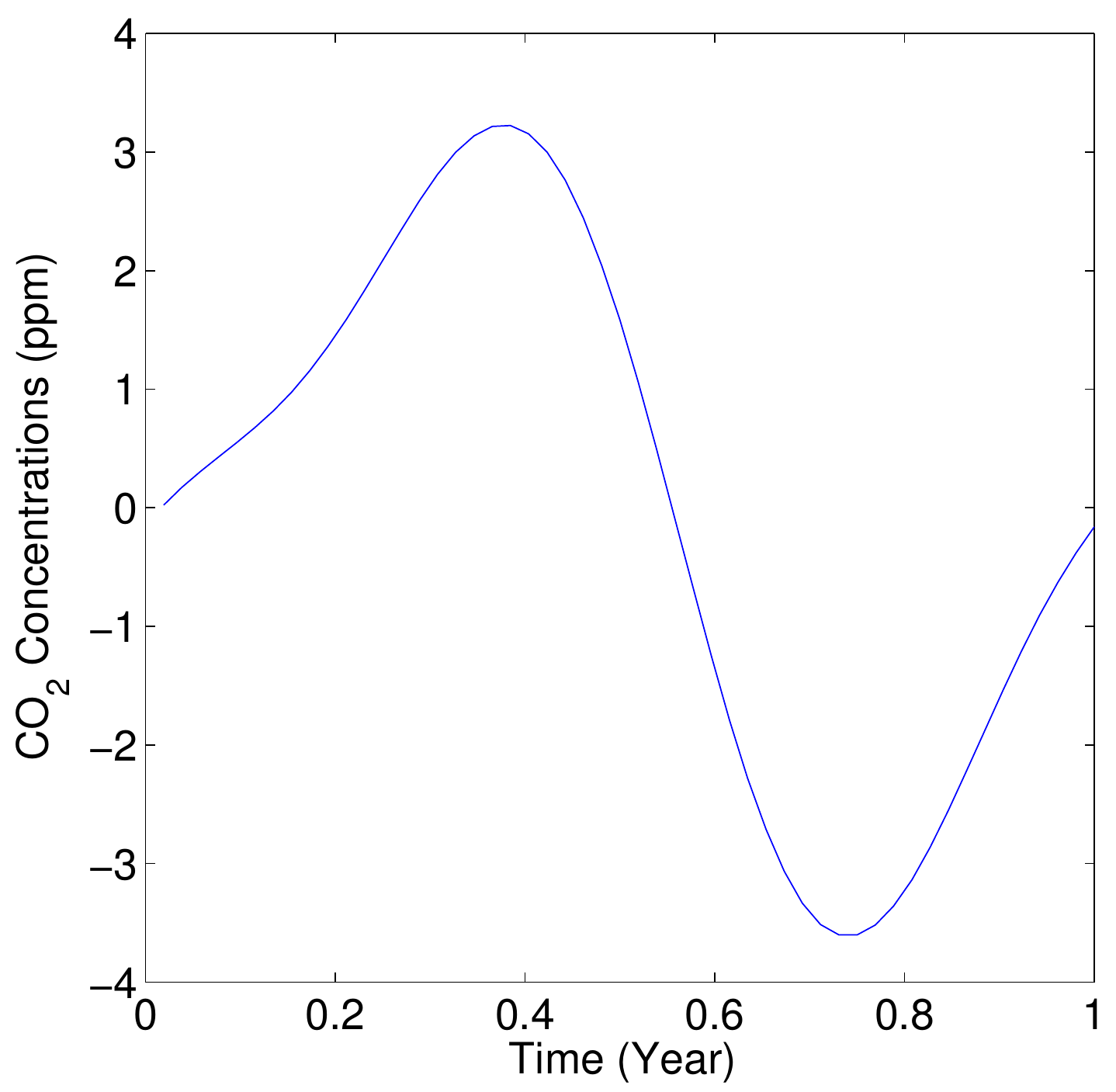}
    \end{tabular}
  \end{center}
  \caption{Wave shapes of relevant components provided by wavelet transform. Left: Annual wave shape. Middle: Semiannual wave shape. Right: Summation of the annual wave shape and semiannual wave shape.}
  \label{fig:CO2WaveShape}
\end{figure}

\subsection{Synchrosqueezed time-frequency analysis}

A powerful tool for mode decomposition problem is the synchrosqueezed time-frequency analysis consisting of a linear time-frequency analysis tool and a synchrosqueezing technique \cite{Daubechies2011,SSSTFT,SSGeneral,SSRobust}. 
Synchrosqueezed wavelet transforms (SSWT), first proposed in \cite{Daubechies1996} by Daubechies et al., can accurately decompose a class of superpositions of wave-like components and estimate their instantaneous frequencies, as proved rigorously in \cite{Daubechies2011}. Following this research line, a synchrosqueezed short-time Fourier transform (SSSTFT) and a generalized synchrosqueezing transform have been proposed in \cite{SSSTFT} and \cite{SSGeneral}, respectively. Stability properties of these synchrosqueezing approaches have been studied in \cite{SSRobust} recently. With these newly developed theories, these synchrosqueezing transforms have been applied to analyze signals in the form \eqref{P1} in many applications successfully  \cite{HauBio3,Geo,SSRobust,HauBio2,HauBio1}.

In the analysis of existing synchrosqueezed transforms \cite{Daubechies2011,SSSTFT}, a key requirement to guarantee an accurate estimation of instantaneous properties and decompositions is the well-separation condition for a class of superpositions of intrinsic mode type functions. Let us take the SSWT as an example. Since it is of significance to study the relation between the magnitudes of instantaneous frequencies and the accuracy of instantaneous frequency estimates, $N$ and $N_k$ are introduced in the following definitions.

\begin{definition}
\label{def:IMFSSWT}
(Intrinsic mode type function for the SSWT). A continuous function $f:\R\rightarrow\C$, $f\in L^\infty(\R)$ is said to be intrinsic-mode-type (IMT) with accuracy $\epsilon>0$ if $f(t)=a(t)e^{2\pi iN\phi(t)}$ with $a(t)$ and $\phi(t)$ having the following properties:
\begin{eqnarray*}
&a\in C^1(\R)\cap L^\infty(\R),\quad \phi\in C^2(\R)\\
&\underset{t\in\R}{\inf}\phi'(t)>0,\quad \underset{t\in\R}{\sup}\phi'(t)<\infty,\quad \underset{t\in\R}{\sup}|\phi''(t)|<\infty,\\
&|a'(t)|\leq \epsilon|N\phi'(t)|,|\phi''(t)|\leq\epsilon|\phi'(t)|,\quad \forall t\in\R.
\end{eqnarray*}
\end{definition} 

\begin{definition}
\label{def:WSSSWT}
(Superposition of well-separated intrinsic mode functions for the SSWT). A function $f:\R\rightarrow \C$ is said to be a superposition of well-separated intrinsic mode functions, up to accuracy $\epsilon$, and with separation $\triangle$, if there exists a finite $K$, such that
\[ f(t)=\sum_{k=1}^K f_k(t)=\sum_{k=1}^K a_k(t)e^{2\pi iN_k\phi_k(t)},\]
where all the $f_k$ are IMT, and where moreover their respective phase functions $\phi_k$ satisfy
\[ N_k\phi'_k(t)>N_{k-1}\phi'_{k-1}(t)\quad \text{and}\quad |N_k\phi'_k(t)-N_{k-1}\phi'_{k-1}(t)|\geq \triangle\left[N_k \phi'_k(t)+N_{k-1}\phi'_{k-1}(t)\right],\quad \forall t\in \R.\] 
\end{definition}

In \cite{Daubechies2011}, it is proved that the SSWT can estimate instantaneous frequencies of well-separated intrinsic mode functions from their superposition, using a mother wavelet supported in $[1-d,1+d]$, with $d<\triangle/(1+\triangle)$. The well-separation condition can be essentially referred to as the condition that the instantaneous frequencies $N_k\phi'_k(t)$ are not crossing over the support of the same wavelet in the time-frequency domain. 

For the general mode decomposition problem, a straightforward question would be whether the synchrosqueezed time-frequency representation can extract general modes $\{\alpha_k(t) s_k(2 \pi N_k \phi_k(t))\}_{1\leq k\leq K}$, identify general shape functions $\{s_k(t)\}_{1\leq k\leq K}$ and estimate instantaneous properties. Recently, \cite{Hau-Tieng2013} shows that the SSWT can be used to solve the general mode decomposition problem for a superposition of well-separated general modes with analytic wave shape functions $s_k(t)$ sufficiently close to the exponential function $e^{it}$, i.e., a few terms of the Fourier expansion of $s_k(t)$ are sufficient to approximate $s_k(t)$ well. 
However, this class of band-limited wave shape functions in \cite{Hau-Tieng2013} is still restrictive in some situations, e.g., spike signals in neurons have shape functions with a wide Fourier band as shown in Figure \ref{fig:ECGshape}. The SSWT would not be suitable to address the general mode decomposition problem in these circumstances because the well-separation condition for the SSWT is impractical for the following two reasons.
\begin{enumerate}
\item The superposition of two nearby Fourier expansion terms $\widehat{s_k}(n)\alpha_k(t) e^{2\pi i n N_k \phi_k(t)}$ and $\widehat{s_k}(n+1)\alpha_k(t) e^{2\pi i (n+1) N_k \phi_k(t)}$ are not well-separated when $n$ is large (see Figure \ref{fig:ScaleResolution} for an example), due to the low resolution of wavelet transforms in the high frequency part of the time-frequency domain. 
\item For two different instantaneous frequencies $N_k \phi_k'(t)$ and $N_j \phi_j'(t)$, their multiples may have crossover frequencies with high probability.
\end{enumerate}

\begin{figure}
  \begin{center}
    \begin{tabular}{cc}
      \includegraphics[height=2.4in]{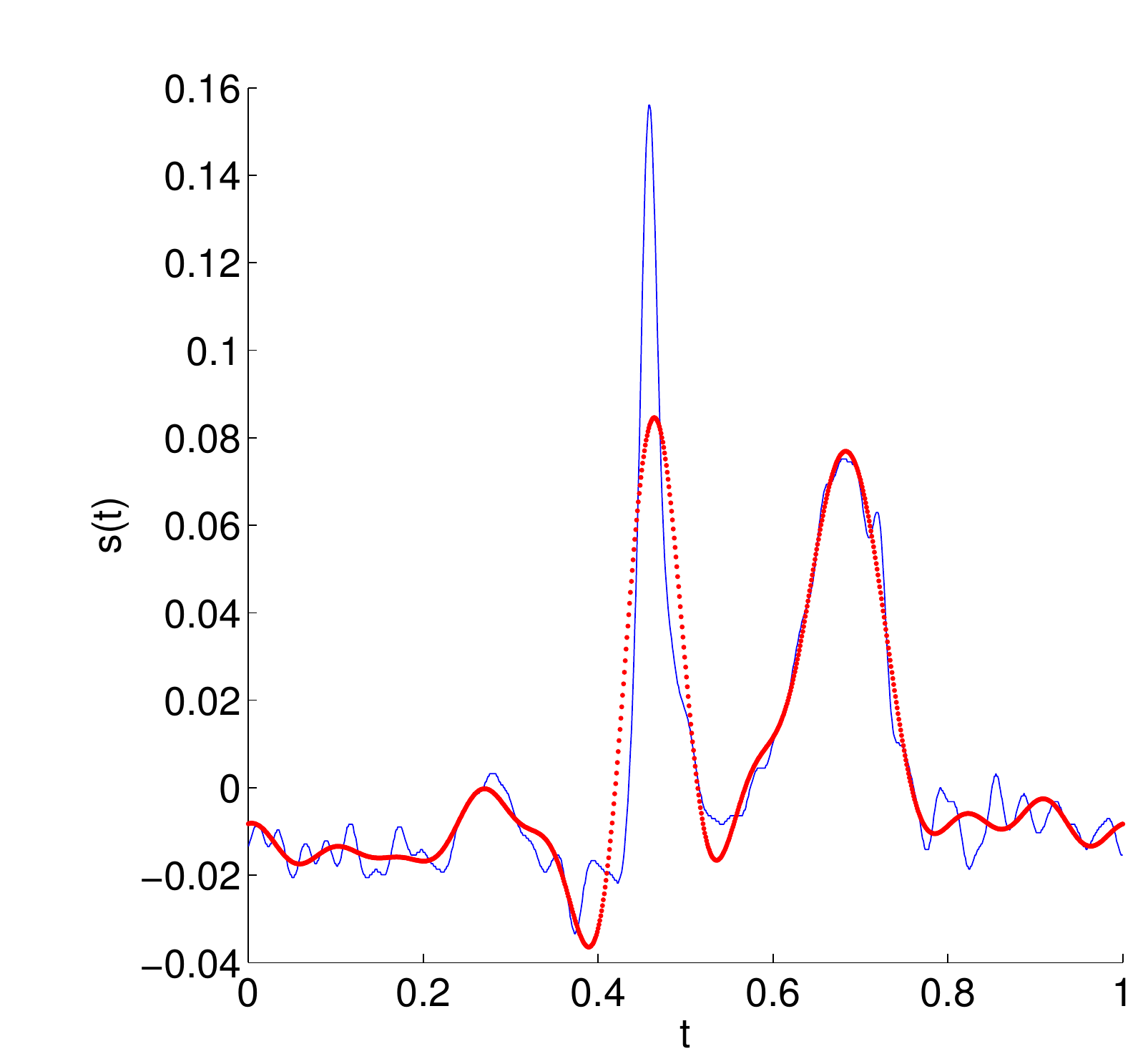} &
      \includegraphics[height=2.4in]{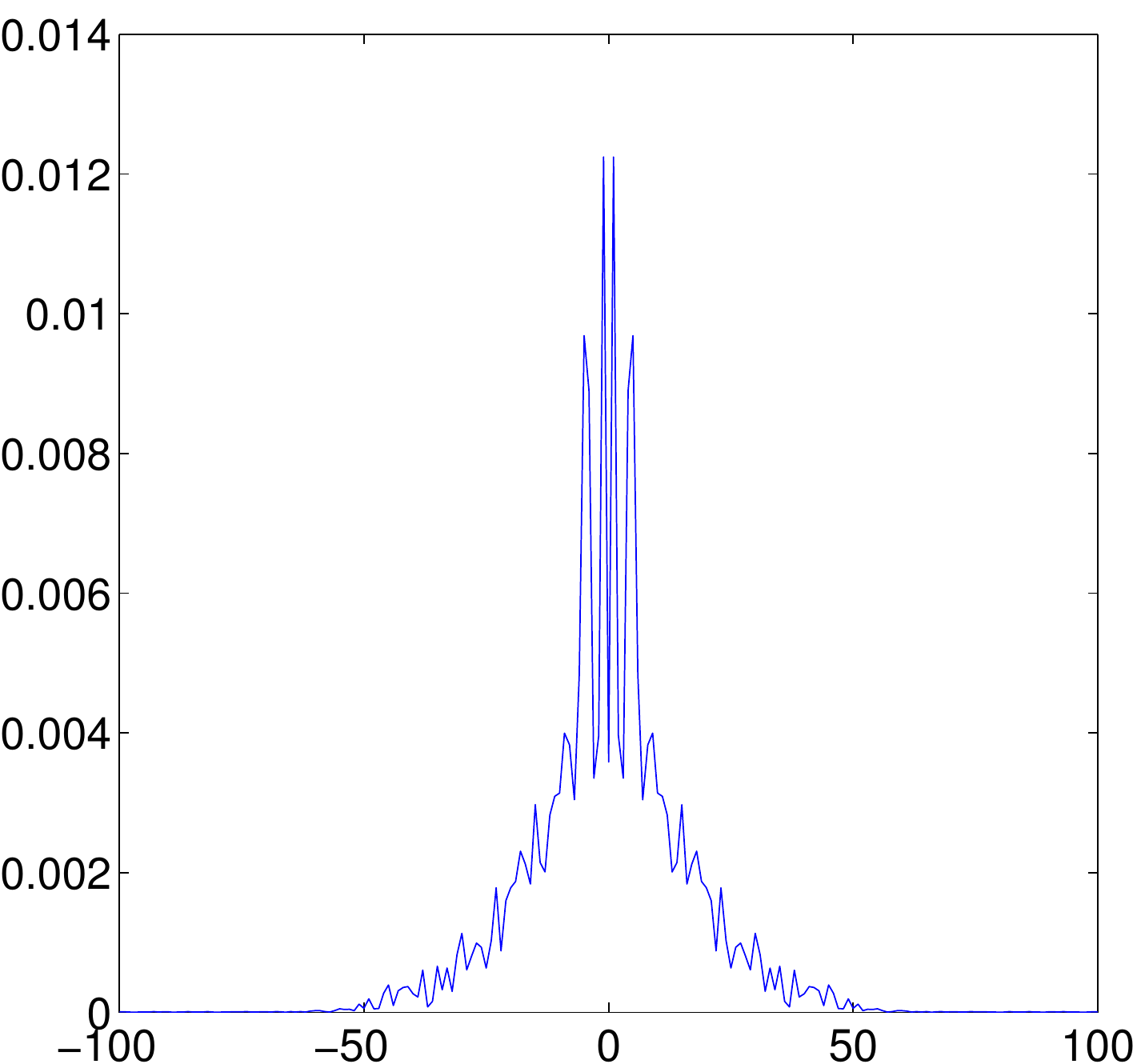}
    \end{tabular}
  \end{center}
  \caption{Left: The solid line plots a spike shape function $s(t)$ of a real ECG signal and the dotted line is its band-limited approximation $\sum_{|n|\leq 10} \widehat{s}(n) e^{2\pi i n t}$. The sum of a few Fourier expansion terms cannot approximate the shape function accurately. Most importantly, the highest peak, the key quantity called R-peak in \cite{goldberger06}, is smoothed and is hardly distinguished. Right: The Fourier power spectrum $|\widehat{s}(\xi)|$ of $s(t)$ is plotted. The energy in the Fourier domain is spreading widely.}
  \label{fig:ECGshape}
\end{figure}

\begin{figure}
  \begin{center}
    \begin{tabular}{c}
      \includegraphics[height=1in]{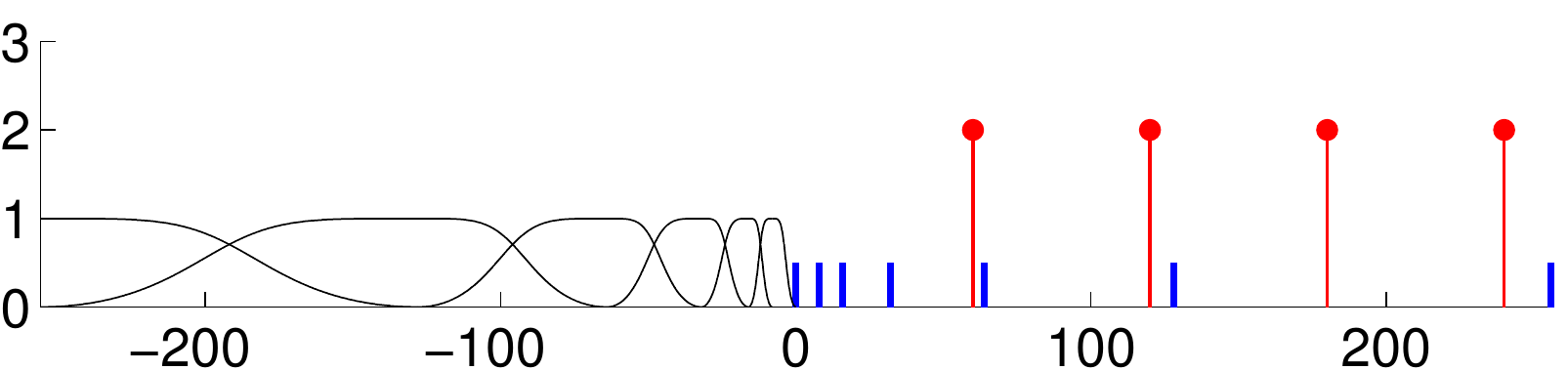}\\
      \includegraphics[height=1in]{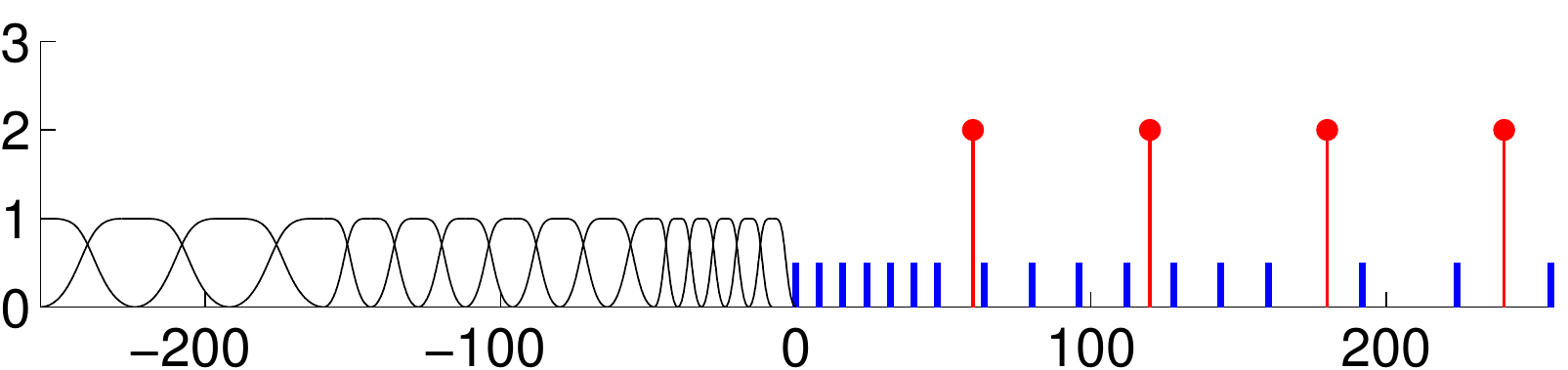}
    \end{tabular}
  \end{center}
  \caption{Top: Wavelet tiling and bump functions in the high frequency Fourier domain. The mother wavelet has a support of size $d=1$. Each short budding in the positive part indicates the center of one bump function, while the bump functions in the negative part are plotted. The dots denote the support of the Fourier transform of $\sum_{n=1}^4e^{2\pi inNt}$, where $N=60$. The well-separation condition for the SSWT is not satisfied, because the wavelet transform of these modes are overlapping in the time-frequency domain. Bottom: Wave packet tiling and bump functions with $d=1$. The well-separation condition holds.}  
  \label{fig:ScaleResolution}
\end{figure}

One possible idea to address the first problem might be to apply the SSSTFT in \cite{SSSTFT}. The SSSTFT has a much weaker requirement for the well-separation condition and it seems to have a much better resolution to distinguish intrinsic mode type functions with high frequencies. However, the SSSTFT may not be suitable to estimate instantaneous frequencies $nN_k \phi_k'(t)$ accurately when $nN_k$ is large. First of all, the resolution parameter $\alpha\geq O(nN_k)$ in \cite{SSSTFT} is large, resulting in a large error bound of instantaneous frequency estimates. Second, when $nN_k$ is large, $nN_k\phi''_k(t)$ is large, which may lead to an almost zero synchrosqueezed STFT of the component $\widehat{s_k}(n)\alpha_k(t) e^{2\pi i n N_k \phi_k(t)}$. In this case, it is difficult to estimate $nN_k\phi'_k(t)$ and to recover $\widehat{s_k}(n)\alpha_k(t) e^{2\pi i n N_k \phi_k(t)}$.

The above problems motivate the design of 1D synchrosqueezed wave packet transforms (SSWPT) and a diffeomorphism based spectral analysis method (DSA). First, the SSWPT has a better resolution to distinguish high frequency harmonic modes than the SSWT and provides more accurate instantaneous frequency estimate than the SSSTFT. Hence, the SSWPT is a good alternative to solve the general mode decomposition problem of the form \eqref{P2}. Second, under the weak well-separation condition that all the general components have at least one term in their Fourier expansions well-separated in the time-frequency domain, the SSWPT can estimate the instantaneous frequency information of each general component. Third, the DSA method can overcome the resolution problem and the crossover frequency problem in existing synchrosqueezing methods once sufficient information is provided by the SSWPT. The DSA method consists of diffeomorphisms and a short-time Fourier transform (in practice, the Fourier transform is applied if $f(t)$ is defined only in a bounded interval). It is capable of decomposing a wide class of general superpositions accurately.

\subsection{Related work}
There are three other research lines to address the mode decomposition problems of the form \eqref{P1}. The first one is the empirical mode decomposition (EMD) method initialized by Huang et al. in \cite{Huang1998} and refined in \cite{Huang2009}. To improve the noise resistance of the EMD methods, some variants have been proposed in \cite{Hou2009,Wu2009EEMD}. It has been shown that the EMD methods can decompose some signals into more general components of the form \eqref{P2} instead of the form \eqref{P1} in some cases (see Figure \ref{fig:EEMD} left) in \cite{PhysicalAnal}. In this sense, the EMD methods are able to reflect the nonlinear evolution of the physically meaningful oscillations using general shape functions. However, this advantage is not stable and consistent as illustrated in Figure \ref{fig:EEMD} right. It is worth more effort to understand the EMD methods on general mode decompositions.

\begin{figure}[ht!]
  \begin{center}
    \begin{tabular}{cc}
      \includegraphics[height=2.8in]{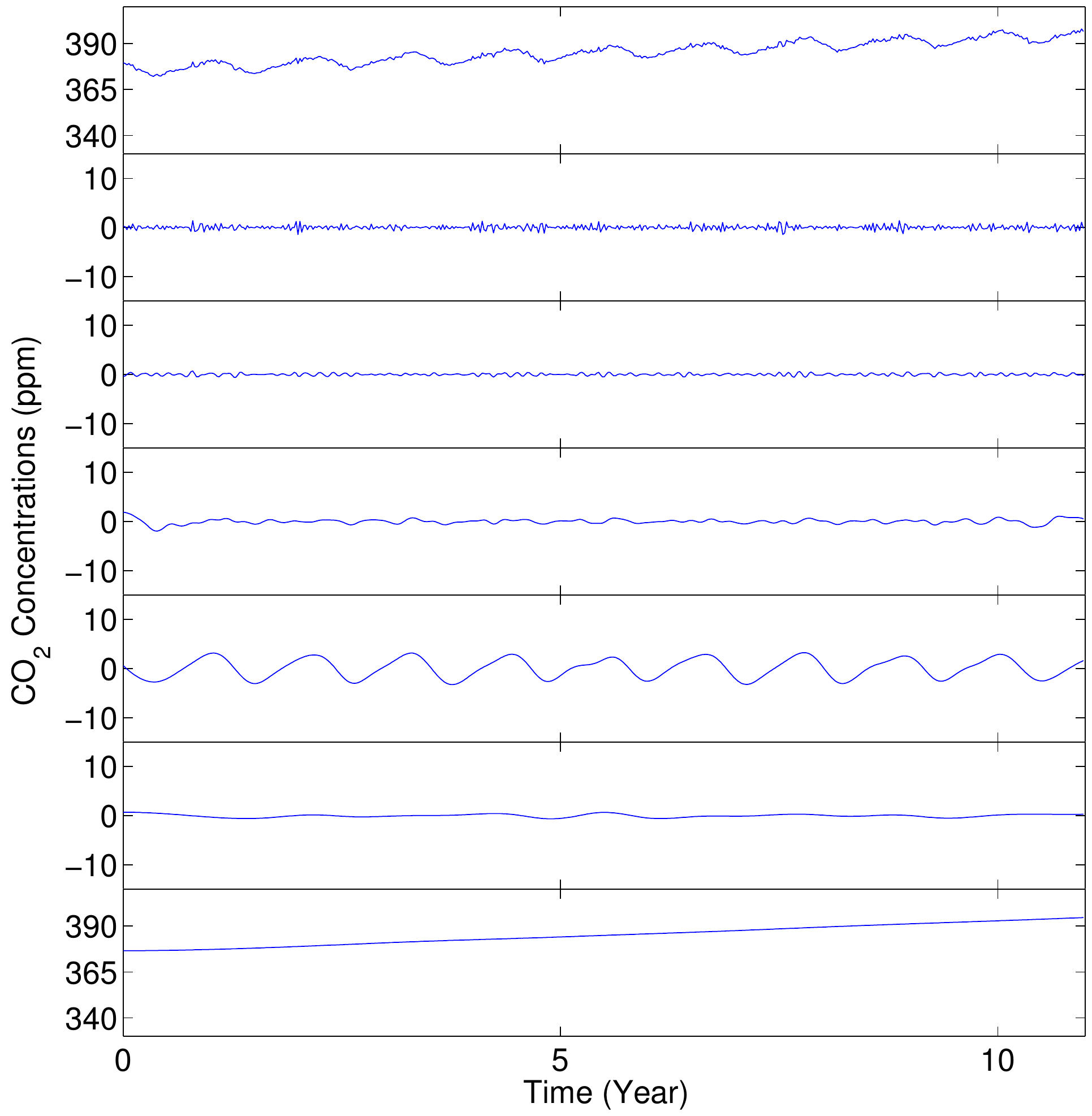} &
\includegraphics[height=2.8in]{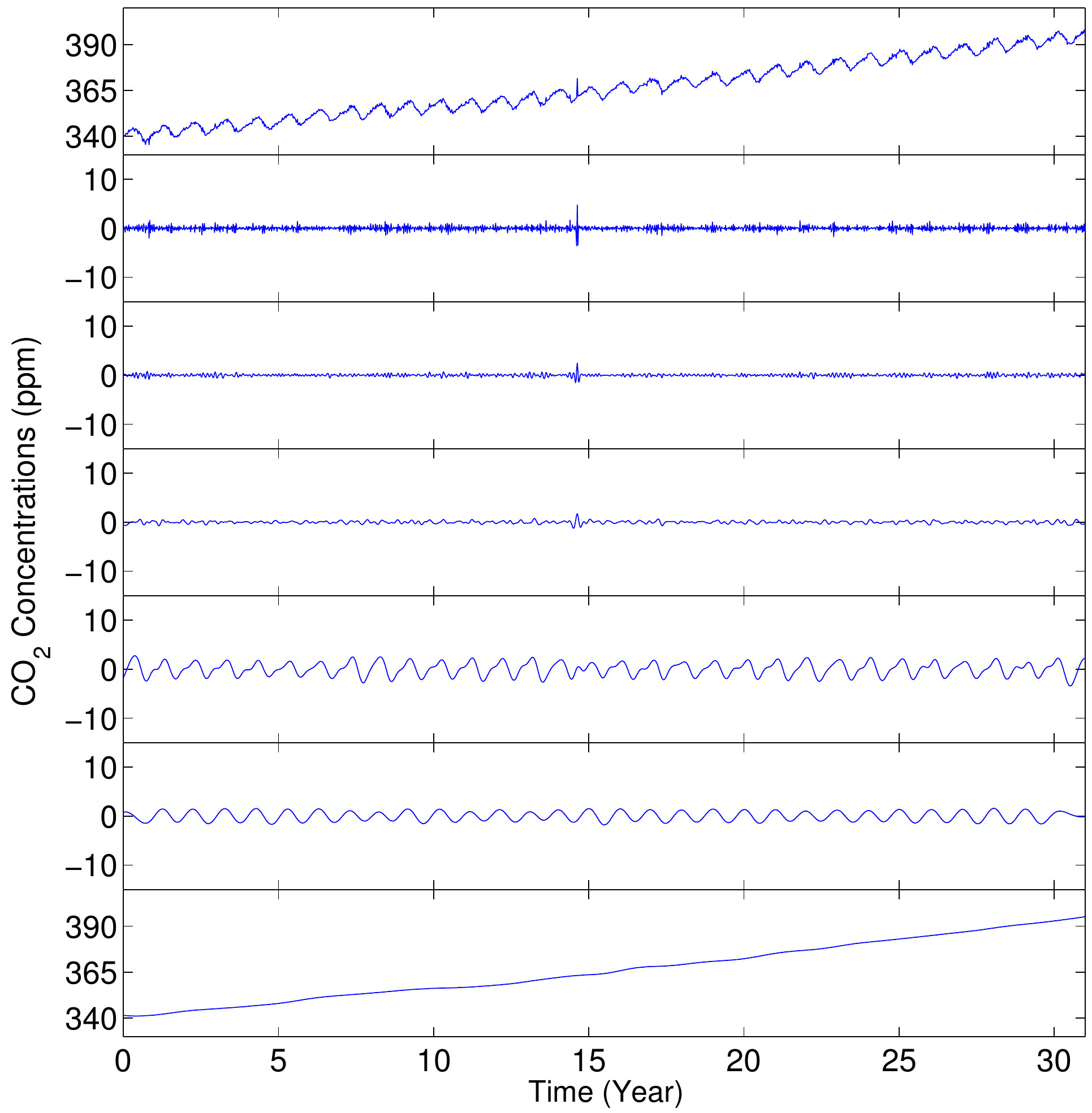} 
    \end{tabular}
  \end{center}
  \caption{Left: Decomposition results of recent $11$ years (2001-2011) CO$_2$ concentration at MLO by EEMD. Right: The results of recent $31$ years (1981-2011). The first row of each table shows the original data and the components provided by the EEMD method in \cite{Wu2009EEMD} are plotted in an order of decreasing frequencies below the original data. Left: The EEMD method provides a single general component. Right: The EEMD method provides two annual cycles, contrary to the result of $11$ years data.}
  \label{fig:EEMD}
\end{figure}

By extracting the components one-by-one from the most oscillatory one, Hou and Shi proposed a nonlinear optimization scheme to decompose signals. The first model in \cite{Hou2011} is based on nonlinear $TV^3$ minimization, which is computationally costly. To deal with this problem, the second paper \cite{Hou2012} proposed a nonlinear matching pursuit model based on sparse representations of signals in a data-driven time-frequency dictionary, which has a fast algorithm for periodic data. Under some sparsity assumptions, the analysis of convergence for the latter scheme has been recently studied in \cite{Hou2013}.

The third method is the empirical wavelet transform recently proposed in \cite{EWT,2DEWT} by Gilles, Tran and Osher, which empirically builds a wavelet filter bank according to the energy distribution of a given signal in the Fourier domain so as to obtain an adaptive time-frequency representation.

The rest of this paper is organized as follows. In Section \ref{sec:implementation}, the 1D SSWPT and the DSA method are briefly introduced by providing a simple example. In Section \ref{sec:SSWPT},  main theorems for 1D SSWPT to solve the general mode decomposition problem are presented. In Section \ref{sec:diff}, the DSA method is theoretically analyzed. In Section \ref{sec:results}, some synthetic and real examples are provided to demonstrate the efficiency of the above two methods. Finally, we conclude with some discussions of future work in Section \ref{sec:conclusion}.

\section{Implementation of proposed methods}
\label{sec:implementation}
\subsection{1D synchrosqueezed wave packet transforms (SSWPT)}
\label{sub:1DSSWPT}
In what follows, we briefly introduce the 1D SSWPT based on the 2D SSWPT in \cite{SSWPT}. Let $w(t)$ be a mother wave packet in the Schwartz class and the Fourier transform
$\widehat{w}(\xi)$ is a real-valued, non-negative, smooth function
with a support equal to $(-d,d)$ determined by a parameter $d\leq 1$. We can use $w(t)$ to define a family of wave packets through scaling, modulation, and translation, controlled by a geometric parameter $s$.

\begin{definition}
  \label{def:WA2d}
  Given the mother wave packet $w(t)$ and the parameter $s\in(1/2,1)$,
  the family of wave packets $\{w_{a b}(t): |a|\geq 1,b\in \R\}$
  is defined as
  \[
  w_{a b}(t)=|a|^{s/2} w(|a|^s(t-b)) e^{2\pi i (t-b)a},
  \]
  or equivalently, in the Fourier domain as
  \[
  \widehat{w_{ab}}(\xi) = |a|^{-s/2} e^{-2\pi i b\xi}
  \widehat{w}(|a|^{-s}(\xi-a)).
  \]
\end{definition}
Notice that if $s$ were equal to $1$, these
functions would be qualitatively similar to the standard wavelets. On the other hand, if $s$ were equal to $1/2$, we would obtain the wave atoms defined in \cite{Demanet2007}. But $s\in(1/2,1)$ is essential as we shall see in the main theorems.

The instantaneous frequency of the low frequency part is not well defined as discussed in \cite{Picinbono97}. For this reason, it is enough to consider the wave packets with $|a|\geq 1$. The high frequency modes can be identified and extracted independently of the low frequency part so that the low frequency part can be recovered by removing high frequency modes.
%
\begin{definition}
  \label{def:WAT}
  The 1D wave packet transform of a function $f(t)$ is a function
  \begin{align}
    W_f(a,b) 
    &= \langle w_{a b},f\rangle =  \int \overline{w_{a b}(t)}f(t)dt \label{eq:WAT} \\ 
    &= \langle \widehat{w_{ab}},\widehat{f}\rangle = \int \overline{\widehat{w_{ab}}(\xi)} \widehat{f}(\xi)d\xi \nonumber
  \end{align}
  for $|a|\geq 1,b\in \R$.
\end{definition}
For $f\in L^2(\R)$, if the Fourier transform $\widehat{f}(\xi)$ vanishes for $|\xi|< 1$, it is
easy to check that the $L^2$ norms of $W_f(a,b)$ and $f(t)$ are
equivalent, i.e., $\exists c_1$ and $c_2$ such that $0<c_1<c_2<\infty$ and
\begin{equation}
c_1\int |f(t)|^2 dt\leq  \int |W_f(a,b)|^2 da db \leq c_2\int |f(t)|^2 dt.  \label{eq:ENEEQ}
\end{equation}

\begin{definition} 
  \label{def:IF}
  Instantaneous frequency information function:

  Let $f\in L^\infty(\R)$. The instantaneous frequency information function of $f$ is defined by
  \begin{equation}
    v_f(a,b)=
    \begin{cases}
      \frac{ \partial_b W_f(a,b) }{ 2\pi i W_f(a,b)},
      & \text{for }|W_f(a,b)|>0;\\
      \infty, 
      & otherwise.
    \end{cases}
    \label{E1single}
  \end{equation}
\end{definition}

It will be proved that, for a class of wave-like functions $f(t) = \alpha(t) e^{2\pi i N\phi(t)}$, $v_f(a,b)$ precisely approximates $N\phi'(b)$ independently of $a$ as long as $W_f(p,b)\neq 0$. Hence, if we squeeze the coefficients $W_f(a,b)$ together based upon the same instantaneous frequency information function $v_f(a,b)$, then we would obtain a sharpened time-frequency representation of $f(t)$. This motivates the definition of the synchrosqueezed energy distribution as follows.

\begin{definition}
  Given $f(t)$, $W_f(a,b)$, and $v_f(a,b)$, the synchrosqueezed energy
  distribution $T_f(v, b)$ is defined by
  \begin{equation}
    T_f(v,b) = \int_{\R} |W_f(a,b)|^2 \delta(\Re{v_f(a,b)}-v) da \label{eq:SED}
  \end{equation}
  for $v,b\in \R$.
\end{definition} 

\begin{figure}[ht!]
  \begin{center}
    \begin{tabular}{cccc}
     \includegraphics[height=1.3in]{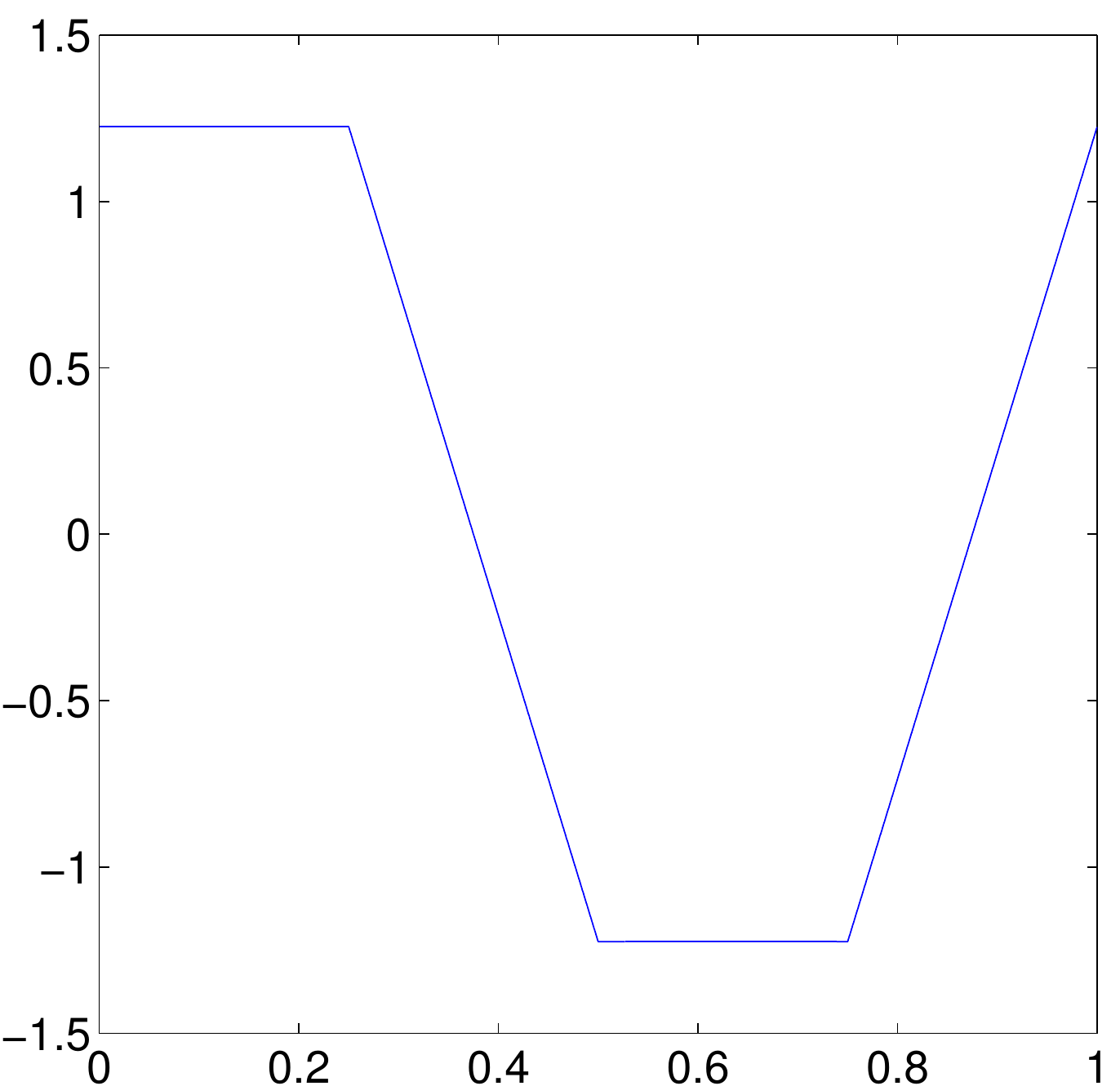} & \includegraphics[height=1.3in]{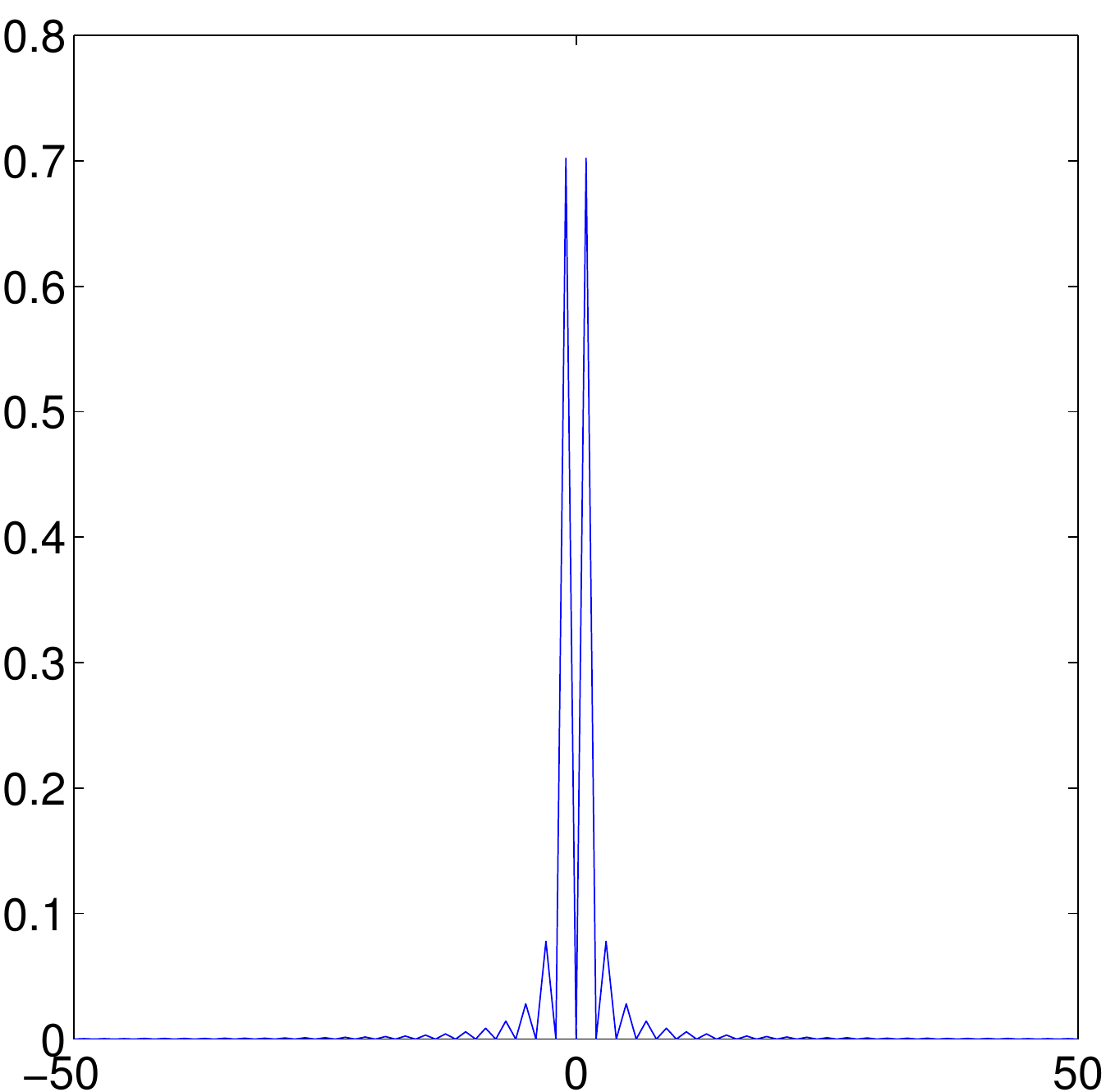}  & \includegraphics[height=1.3in]{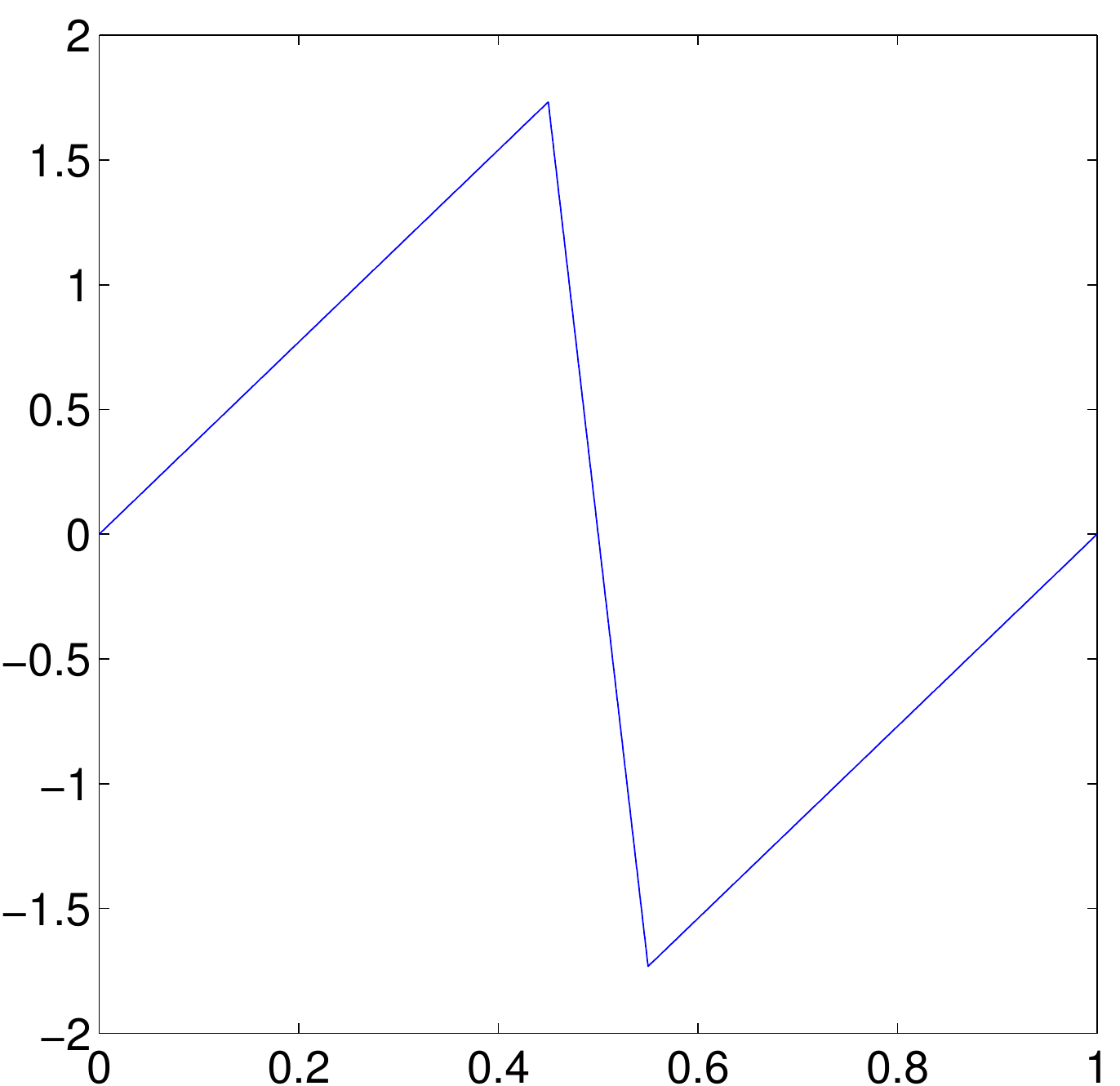}&\includegraphics[height=1.3in]{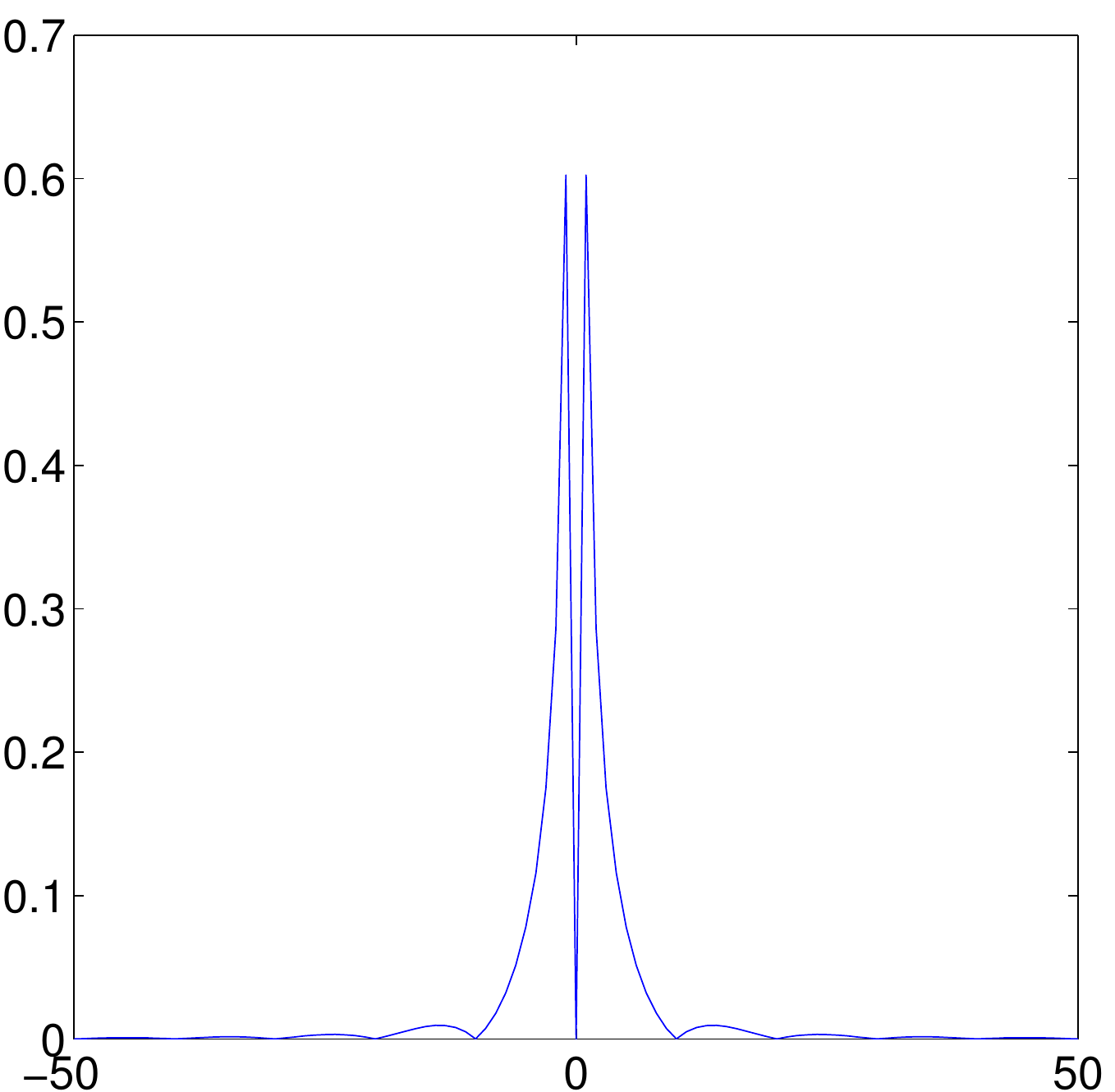}
    \end{tabular}
 \begin{tabular}{c}
      \includegraphics[height=1in]{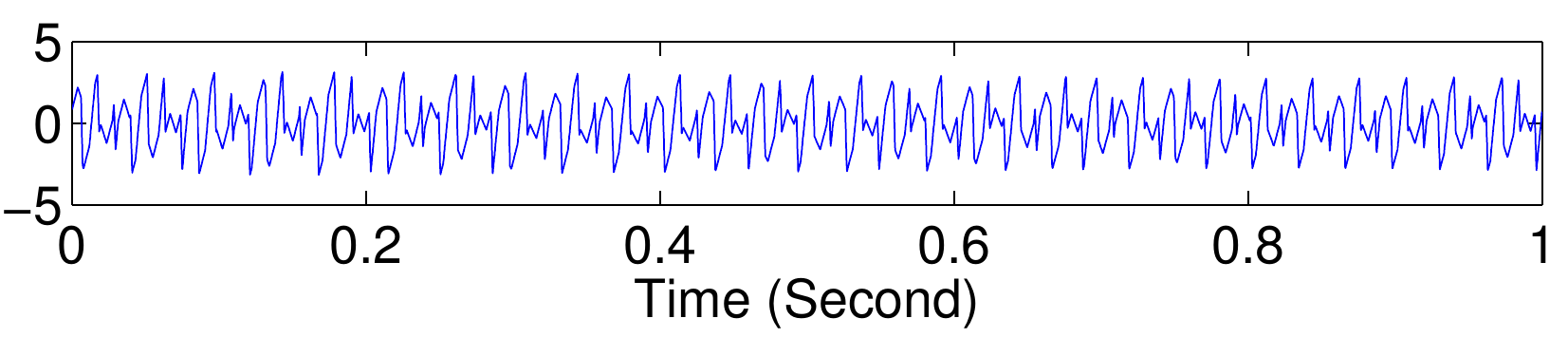}  
    \end{tabular}
  \end{center}
  \caption{Top left: The general shape function $s_1(t)$ and its spectral energy $|\widehat{s_1}(\xi)|$. Top right: The general shape function $s_2(t)$ and its spectral energy $|\widehat{s_2}(\xi)|$. Bottom: A superposition of general modes generated by using $s_1(t)$ and $s_2(t)$.}  
  \label{fig:s1s2}  
\end{figure}

For a multi-component signal $f(t)$, the synchrosqueezed energy of each component will concentrate around its corresponding instantaneous frequency. Hence, the SSWPT can provide information about their instantaneous frequencies. 

With the definition of the SSWPT above, it is ready to explain how it solves the general mode decomposition problem of the form \eqref{P2} under well-separation condition. We denote this method as the GMDWP method for short. Let us consider the \textbf{Example $1$}:
\[
f_1(t) = \alpha_1(t)s_1(2\pi N_1\phi_1(t))= (1+0.05\sin(4\pi x))s_1\left(120\pi(x+0.01\sin(2\pi x))\right)
\] 
and
\[
f_2(t) = \alpha_2(t)s_2(2\pi N_2\phi_2(t))= (1+0.1\sin(2\pi x))s_1\left(180\pi(x+0.01\cos(2\pi x))\right),
\] 
where $s_1(t)$ and $s_2(t)$ are periodic general shape functions defined in $[0,1]$ as shown in Figure \ref{fig:s1s2}. Let $f(t)=f_1(t)+f_2(t)$ (see Figure \ref{fig:s1s2} bottom) and we try to recover $f_1(t)$ and $f_2(t)$ from $f(t)$.

\begin{figure}[ht!]
  \begin{center}
    \begin{tabular}{ccc}
     \includegraphics[height=1.6in]{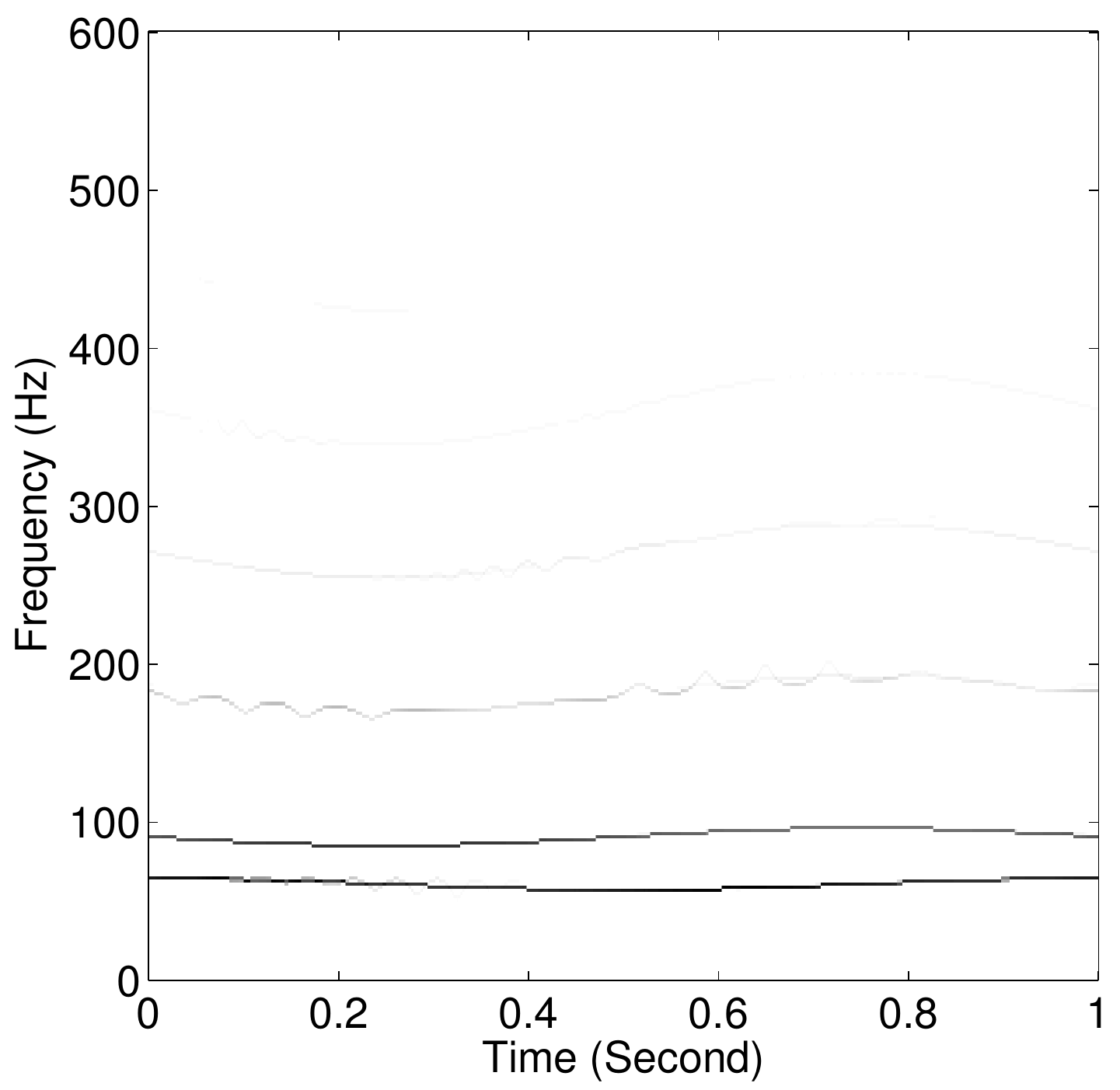}& \includegraphics[height=1.6in]{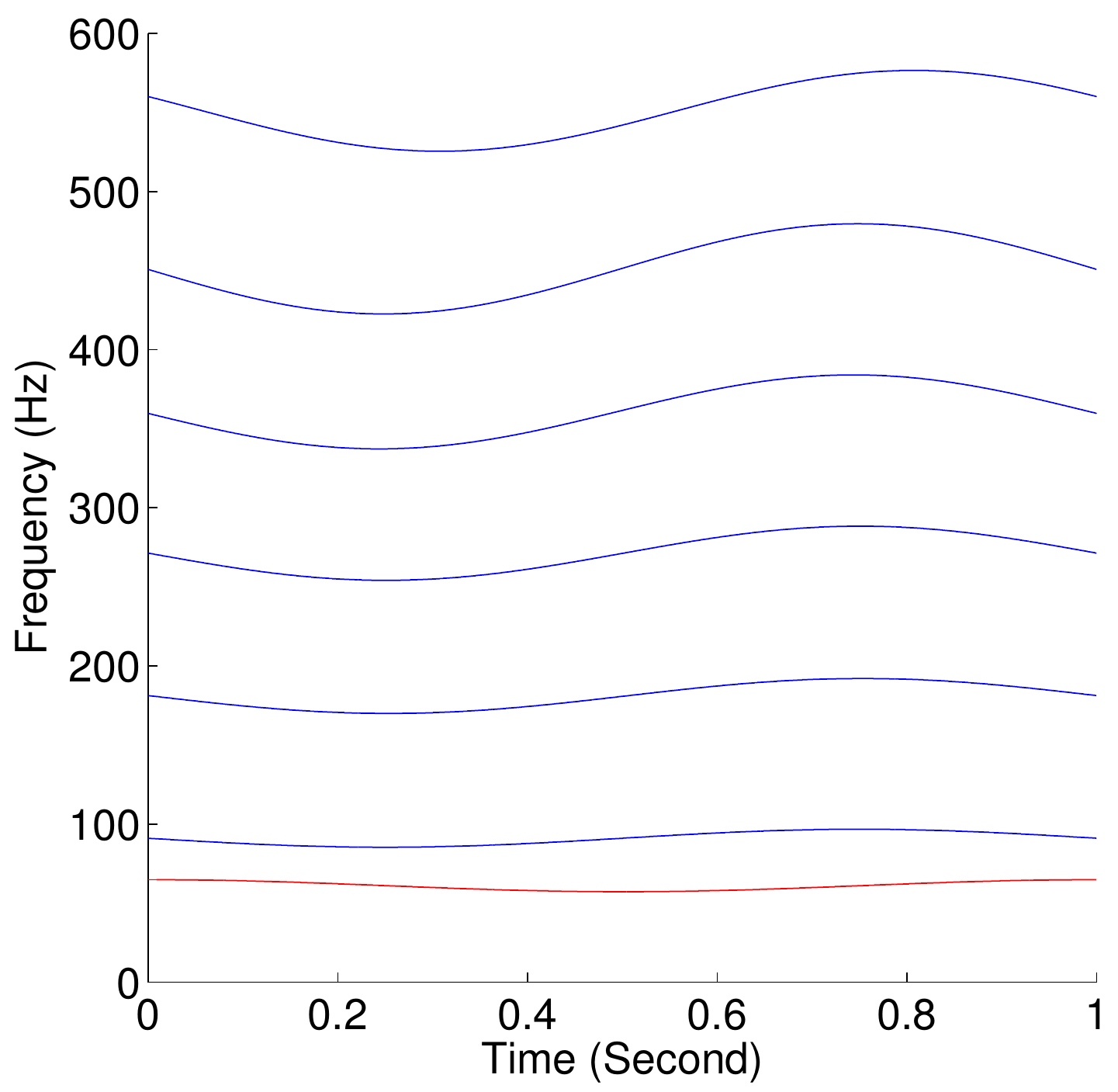} & \includegraphics[height=1.6in]{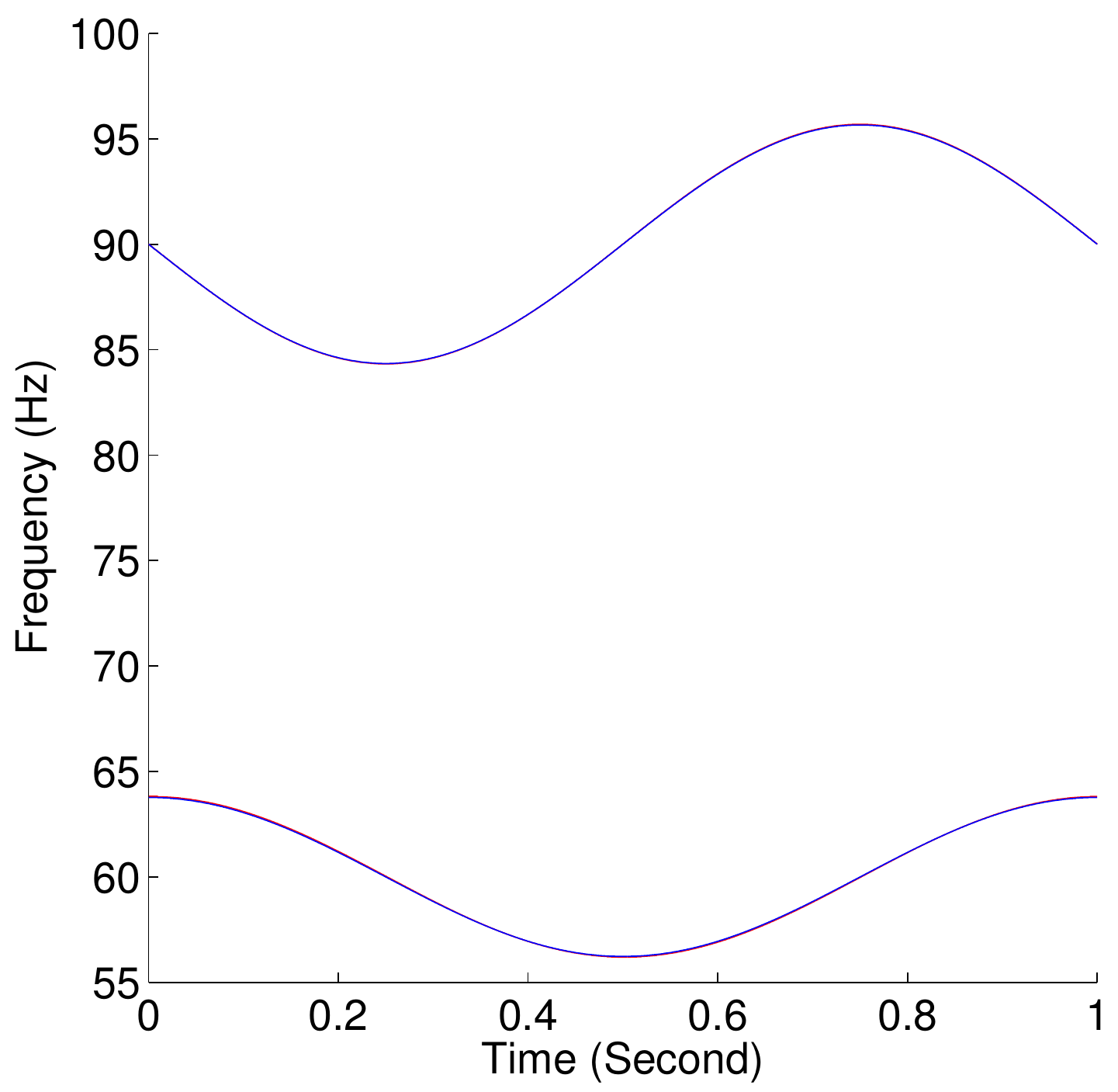}
    \end{tabular}
  \end{center}
  \caption{Left: The synchrosqueezed energy distribution of $f(t)$. The Fourier expansion terms with strong energy are well-separated. Middle: The instantaneous frequency estimates $\psi_{kn}(b)$ of $nN_k\phi_k'(b)$ and the result of curve classification as indicated by different colors. Right: The red curves are the instantaneous frequency estimates $\tilde{\psi}_k$ and the blue curves are the real instantaneous frequencies $N_k\phi_k'(b)$.}
\label{fig:s1s2freq}
\end{figure}

\textbf{Step 1:} Apply the SSWPT on $f(t)$ to compute the synchrosqueezed energy distribution $T_f(v,b)$. The essential support of $T_f(v,b)$ (the support above a certain level) is separated into essentially disjoint sets $\{S_{kn}\}$ as shown in Figure \ref{fig:s1s2freq} left. Each set $S_{kn}$ corresponds to one term in the Fourier expansion of one general mode, say, $\widehat{s_k}(n)\alpha_k(t) e^{2\pi i n N_k \phi_k(t)}$. 

\textbf{Step 2:} As  we shall see in Theorem \ref{thm:main}, those points in $S_{kn}$ are concentrating around $n N_k \phi_k'(t)$. This is also illustrated by Figure \ref{fig:s1s2freq} left. Since each set $S_{kn}$ is well separated from other sets,  by applying the clustering method in \cite{SSWPT}, one can identify each set $S_{kn}$.

\textbf{Step 3:} For each point $(v,b)\in S_{kn}$, $v\approx nN_k\phi_k'(b)$.  This motivates the definition of a weighted mean of the positions of the points in $S_{kn}$ as 
\[
\psi_{kn}(b)=\frac{\sum_{(v,b)\in S_{kn}}|T_f(v,b)|v}{\sum_{(v,b)\in S_{kn}}|T_f(v,b)|}.
\]
$\psi_{kn}(b)$ provides an accurate estimate of the instantaneous frequency $nN_k\phi_k'(b)$ as shown in Figure \ref{fig:s1s2freq} middle. In the presence of noise, $\psi_{kn}(b)$ will be disturbed by noise. Hence, some low-pass filter or smoothing process could be applied to $\psi_{kn}(b)$.

\textbf{Step 4:} Apply the curve classification Algorithm \ref{alg:curve} to identify $\{\psi_{kn}\}_{|n|\geq 1}$ for each $k$. In the case of Example $1$, there are two general modes. Hence, this step gives $\psi_{11}$ and $\{\psi_{2n}\}_{n=1}^6$. $\psi_{11}$ is plotted in red and $\{\psi_{2n}\}_{n=1}^6$ are plotted in blue in Figure \ref{fig:s1s2freq} middle.

\textbf{Step 5:} Since $\psi_{kn}\approx nN_k\phi_k'$, one arrives at a function $\tilde{\psi}_k$ as an estimation of the instantaneous frequency $N_k\phi_k'$ of $\alpha_k(t)s_k(2\pi N_k \phi_k(t))$ by applying the method in Theorem \ref{thm:inst2}. For Example $1$, this step gives $\tilde{\psi}_1$ and $\tilde{\psi}_2$ and they are shown in Figure \ref{fig:s1s2freq} right in red.

\textbf{Step 6:} Suppose $U_k = \bigcup_{n=-\infty}^{\infty} S_{kn}$, then each general mode can be recovered by 
\begin{align}
 f_k(t)= \alpha_k(t)s_k(2\pi N_k \phi_k(t)) = \int_{v_f(a,b) \in U_k} \tilde{w}_{ab}(t) W_f(a,b) dadb, \label{eq:REC}\nonumber
\end{align}
where the set of functions $\{ \tilde{w}_{ab}(t), |a|\in [1,\infty),b\in \R\}$ is
the dual frame of $\{w_{ab}(x),|a|\in [1,\infty),b\in\R\}$. This step can reconstruct two components $f_1(t)$ and $f_2(t)$ for Example $1$ using essential supports of $T_f(v,b)$. These two reconstructed components are shown in Figure \ref{fig:s1s2mode}.

\textbf{Step 7:} Since each Fourier expansion term can be recovered by 
\begin{align}
  \widehat{s_k}(n)\alpha_k(t) e^{2\pi i n N_k \phi_k(t)} = \int_{v_f(a,b) \in S_{kn}} \tilde{w}_{ab}(t) W_f(a,b) dadb, \nonumber
\end{align}
one can recover the instantaneous amplitude up to a constant factor by
\[
\alpha_k(t) \eqsim \tilde{\alpha}_k(t) =\sqrt{\sum_n| \widehat{s_k}(n)\alpha_k(t) e^{2\pi i n N_k \phi_k(t)}|^2}.
\]
This step gives two instantaneous amplitude estimates $\tilde{\alpha}_1(t)$ and $\tilde{\alpha}_2(t)$ for Example $1$. As shown in Figure \ref{fig:s1s2mode} left, the normalized $\tilde{\alpha}_k(t)$ can estimate the normalized $\alpha_k(t)$ accurately. Hence, the general shape function $s_k(t)$ can be recovered by
\[
s_k(t) \eqsim \frac{f_k((N_k\phi_k)^{-1}(\frac{t}{2\pi}))}{\tilde{\alpha}_k((N_k\phi_k)^{-1}(\frac{t}{2\pi}))},
\]
up to a constant factor.

\begin{figure}[ht!]
  \begin{center}
    \begin{tabular}{ccc}
\includegraphics[height=1.6in]{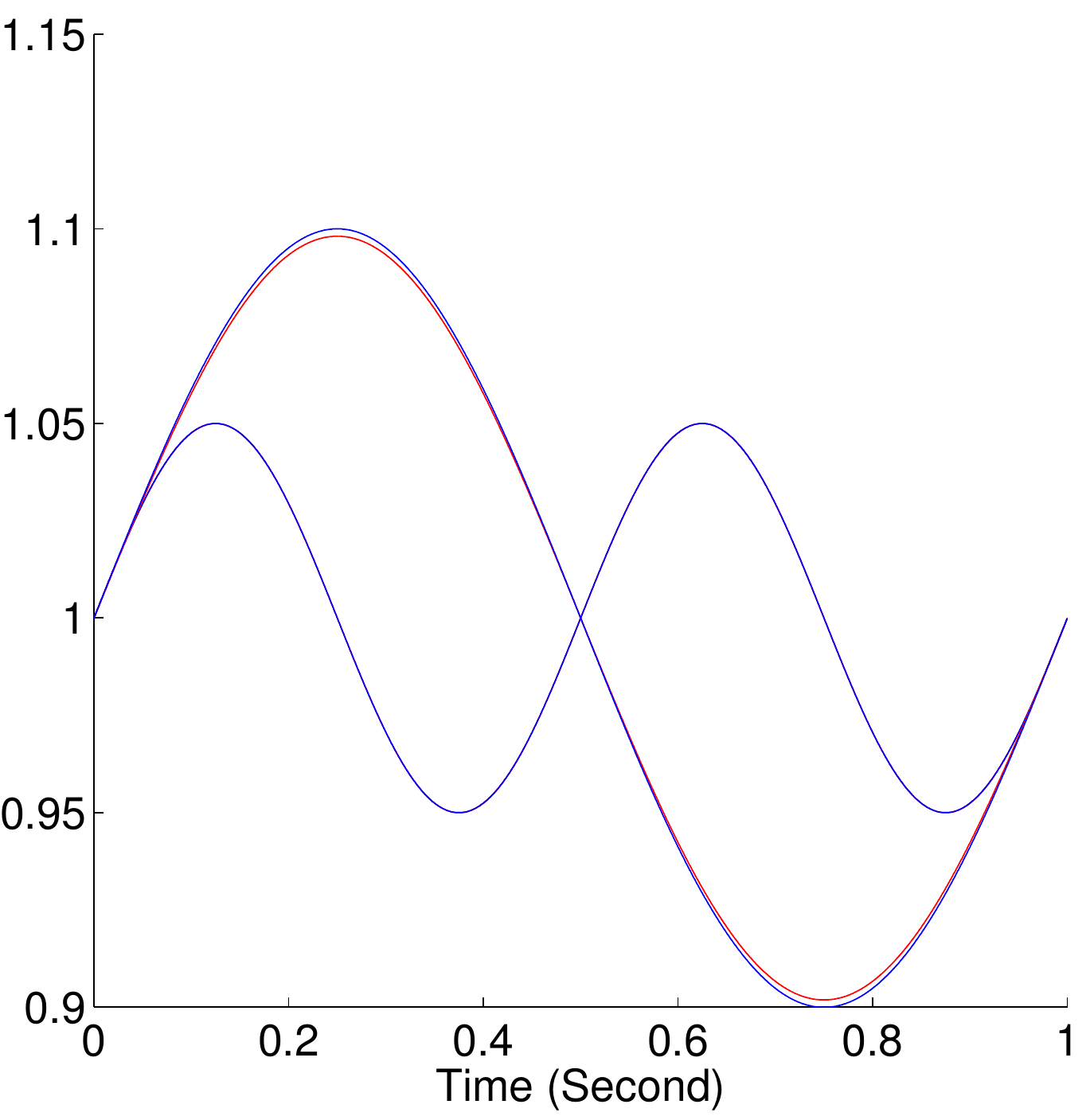} &\includegraphics[height=1.6in]{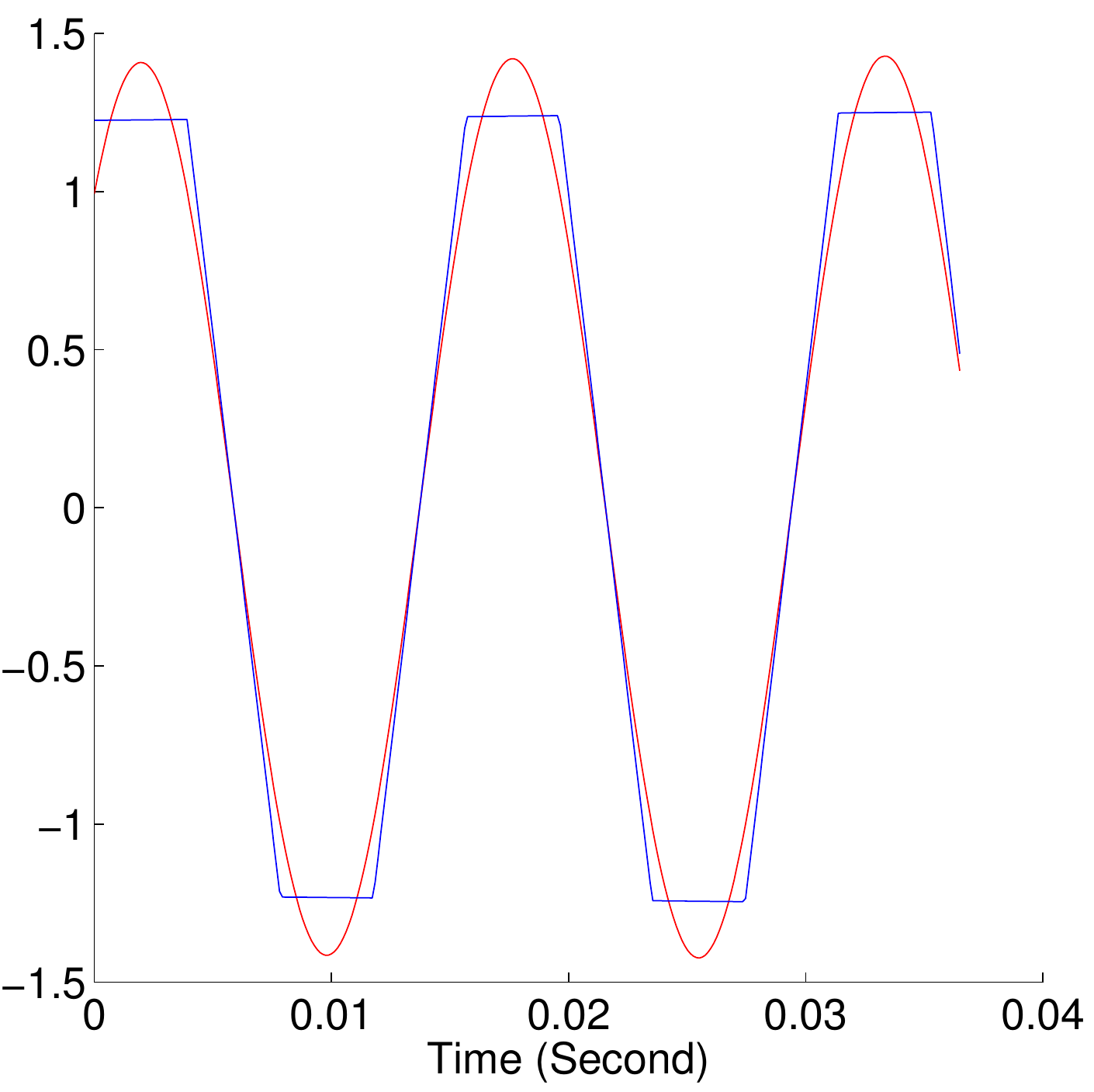} & \includegraphics[height=1.6in]{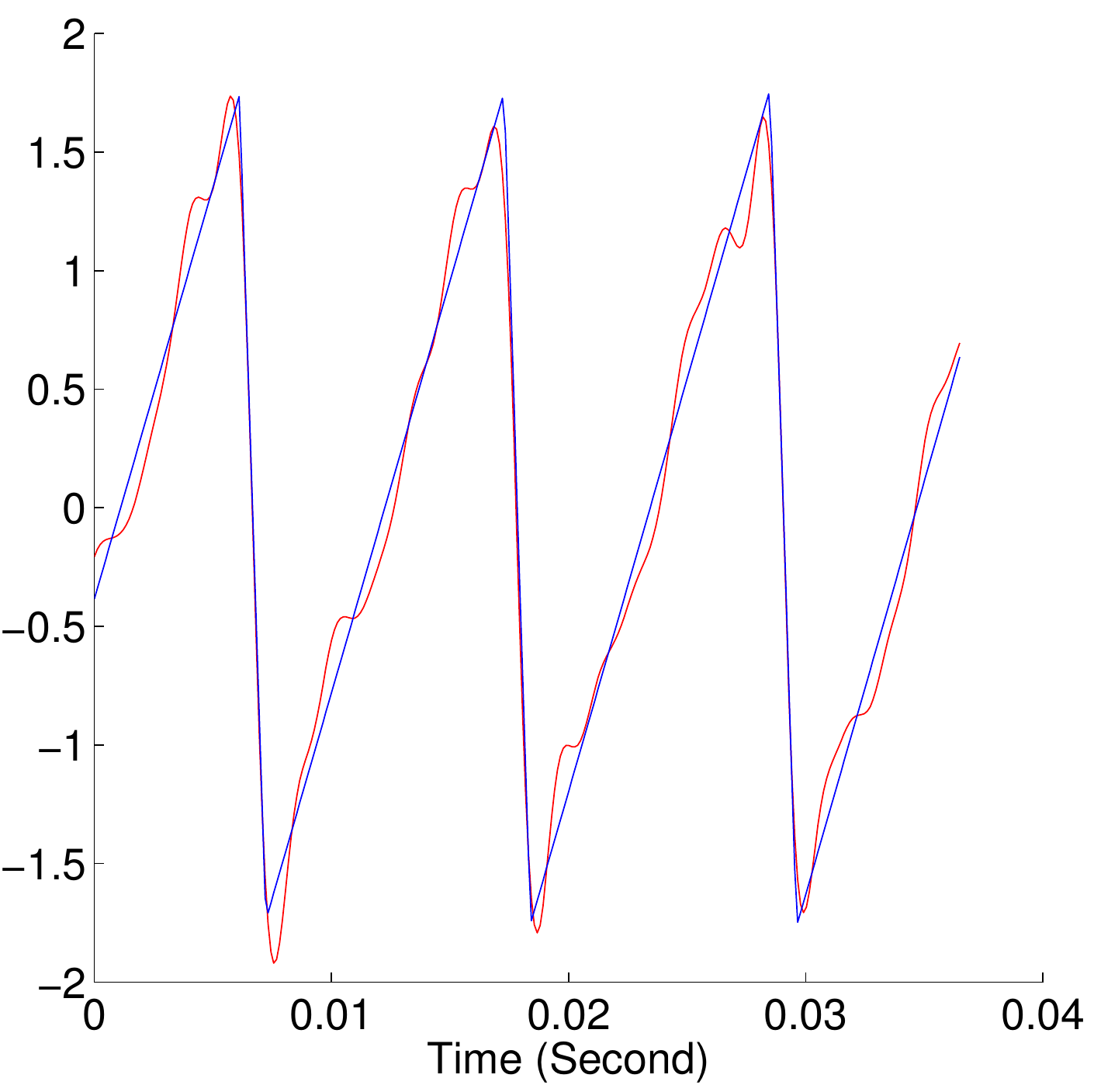}
    \end{tabular}
  \end{center}
  \caption{Blue: Real signals. Red: Reconstructed results. Left: Normalized instantaneous amplitude estimates $\tilde{\alpha}_k(t)$ in red and normalized instantaneous amplitudes $\alpha(t)$ in blue. Middle and right: The real general modes and the recovered general modes provided by the GMDWP method.}
\label{fig:s1s2mode}
\end{figure}

As we can see in the above example, the GMDWP method can provide accurate estimates of instantaneous frequencies and instantaneous amplitudes from the well-separated essential supports of $T_f(v,b)$. However, the reconstructed general modes are not satisfactory (see Figure \ref{fig:s1s2mode} middle and right). As Figure \ref{fig:s1s2log} shows, considering only the essential supports would ignore Fourier expansion terms with weak energy, the information of which is indispensable to reconstruct exact general modes. This desires the diffeomorphism based spectral analysis for exact reconstructions of general modes.

\begin{figure}[ht!]
  \begin{center}
    \begin{tabular}{cc}
\includegraphics[height=1.6in]{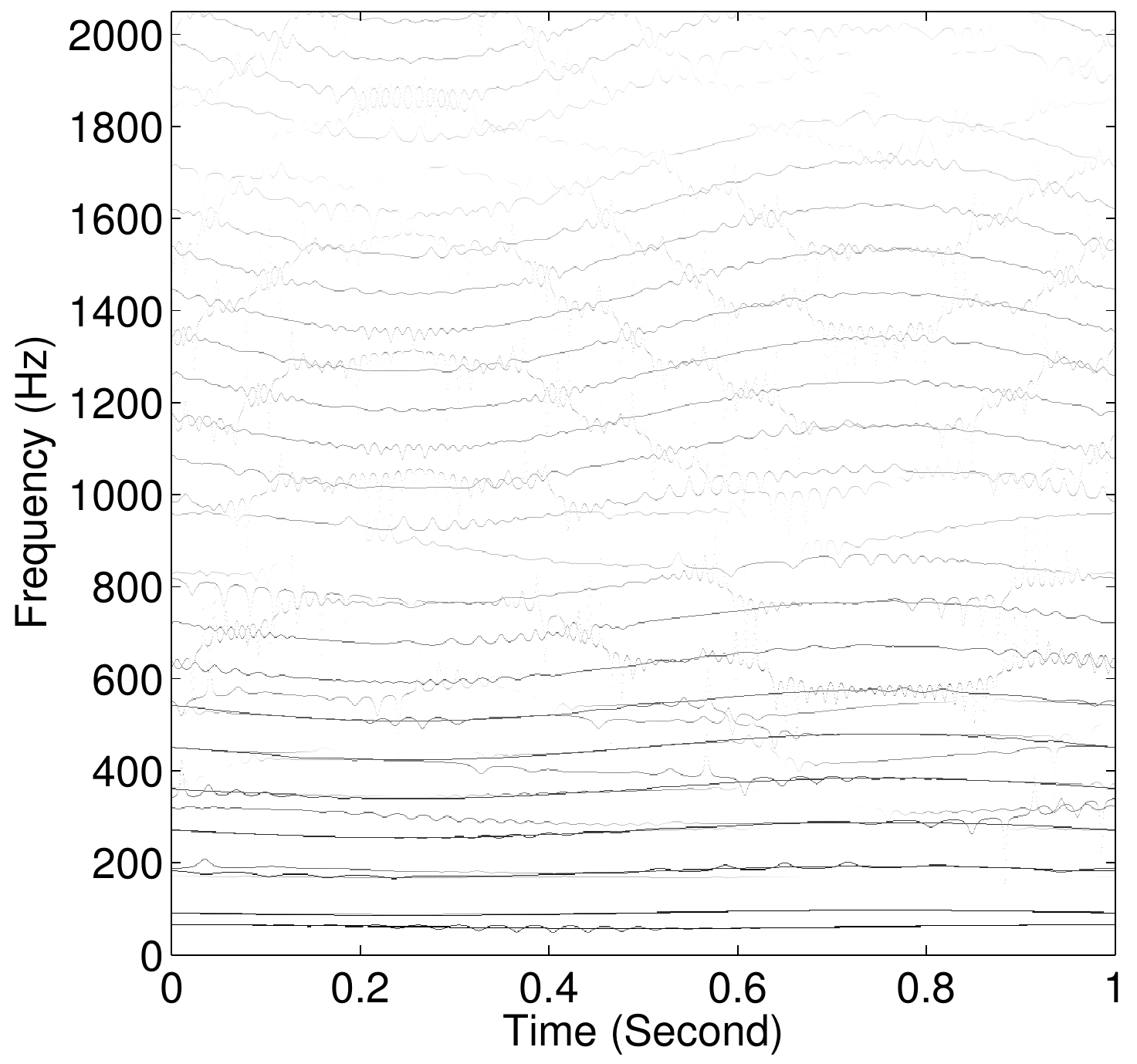} &\includegraphics[height=1.6in]{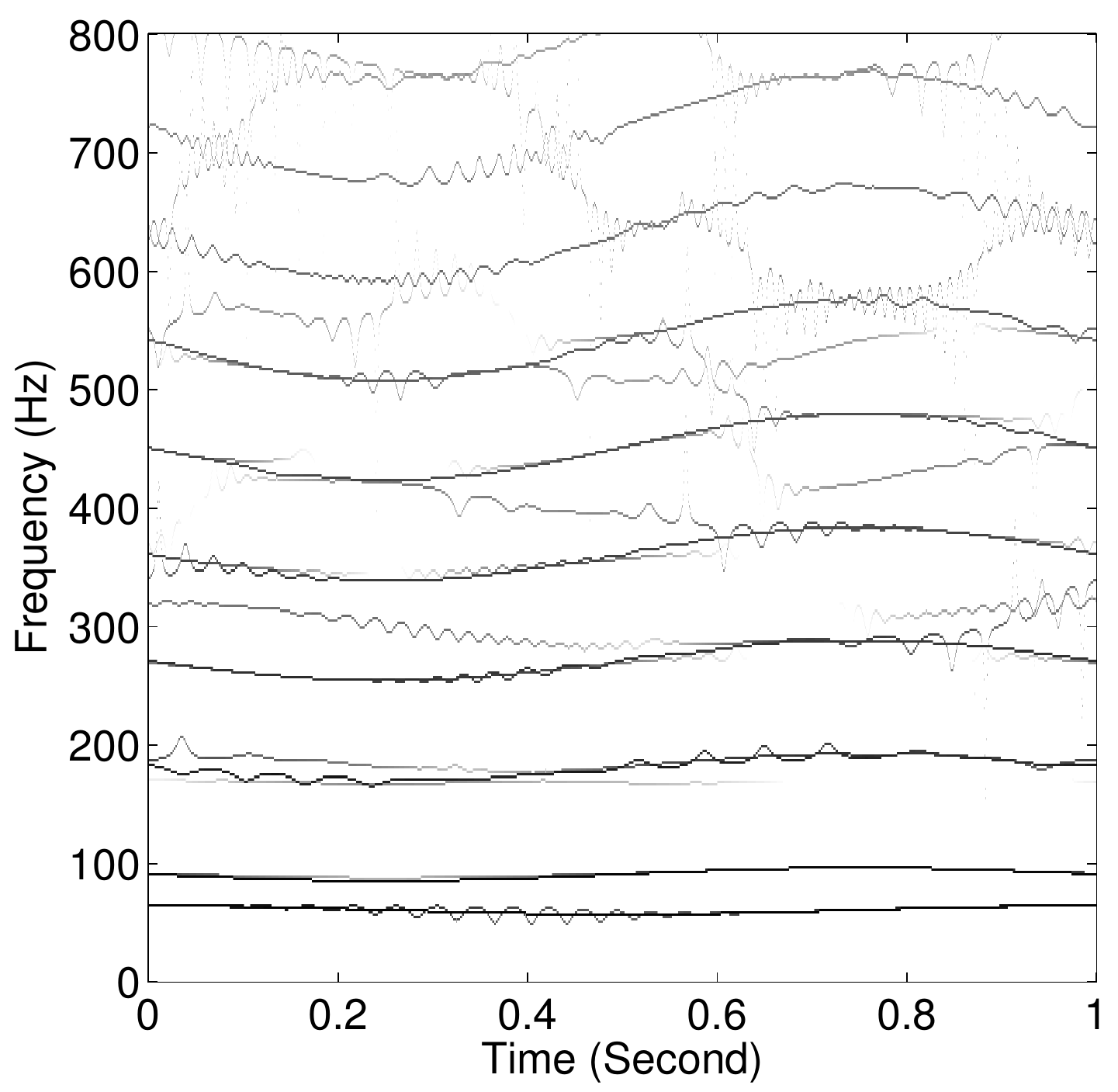} 
    \end{tabular}
  \end{center}
  \caption{Left: $\log_{10}(T_f(v,b))$ in the visible time-frequency domain. Right: $\log_{10}(T_f(v,b))$ in the low frequency part of the time-frequency domain. Some Fourier expansion terms with weak energy are interfering other terms.}
\label{fig:s1s2log}
\end{figure}

\subsection{Diffeomorphism based spectral analysis (DSA)}
\label{sub:diff}
As discussed previously, the well-separation condition, which assumes that the instantaneous frequencies $n N_k \phi_k'(t)$ of nonzero Fourier expansion terms (i.e. $\widehat{s_k}(n)\neq 0$) are well-separated from each other, is not practical. However, it is reasonable to assume that each general mode has at least one Fourier expansion term $\widehat{s_k}(n_k)\alpha_k(t) e^{2\pi i n_k N_k \phi_k(t)}$ with $n_k N_k\phi_k'(t)$ well-separated from other instantaneous frequencies, so that the SSWPT can estimate $n_k N_k\phi_k'(t)$ accurately. This is referred to as the weak well-separation condition. Indeed, we only need the well-separation of the Fourier expansion terms with strong energy as shown by the example in Figure \ref{fig:s1s2freq} left. In what follows, a diffeomorphism based spectral analysis method is introduced to identify all the nonzero Fourier expansion terms using the pre-estimated instantaneous frequencies $\{n_k N_k\phi_k'(t)\}_{k=1}^K$ and the instantaneous amplitudes $\{|\widehat{s_k}(n_k)|\alpha_k(t)\}_{k=1}^K$ provided by the GMDWP method.

Without loss of generality, let us assume the signal of interest is defined in $[0,1]$. Notice that the smooth function $\phi_k(t)$ has the interpretations of a warping in each general mode via a diffeomorphism $\phi_k:\R\to\R$. With the instantaneous frequencies $\{n_k N_k\phi_k'(t)\}_{k=1}^K$ available, we can therefore define the instantaneous phase profiles by
\[
  p_k(t) = \frac{1}{m_k}\int_0^t n_k N_k\phi_k'(x)dx,
\]
where $m_k=\frac{1}{2} \left(\max \limits_{t} {n_k N_k\phi_k'(t)}+\min \limits_{t}{n_k N_k\phi_k'(t)}\right)$. Because $p_k(t)$ is a smooth monotonous function, we can define the inverse-warping profiles in $[0,1]$ by
\begin{eqnarray*}
h_k(t)& =& \frac{f\circ p_k^{-1}(t)}{|\widehat{s_k}(n_k)|\alpha_k\circ p_k^{-1}(t)}\\
&=&\sum_{n=-\infty}^{\infty}\frac{\widehat{s_k}(n)}{|\widehat{s_k}{(n_k)}|} e^{2\pi i ( \frac{n m_k}{n_k}t +n N_k \phi_k(0) )}\\& & + \sum_{j\neq k} \sum_{n=-\infty}^{\infty} \frac{\widehat{s_j}(n)}{|\widehat{s_k}{(n_k)}|} \frac{\alpha_j\circ p_k^{-1}(t)}{\alpha_k\circ p_k^{-1}(t)} e^{2\pi i n N_j \phi_j\circ p_k^{-1}(t)}.
\end{eqnarray*}

If the diffeomorphisms $\phi_k:\R\to\R$ are well different and the phases $2\pi N_k \phi_k(t)$ are sufficiently steep in $[0,1]$, which will be clarified later, the Fourier transform of each inverse-warping profile $\widehat{h_k}(\xi)$ will have sheer peaks at $\xi = \frac{n m_k}{n_k}$ and will be relative small elsewhere. This motivates the design of the DSA method as follows.

\textbf{Step 1:} Input: A signal $f(t)$, its instantaneous phase profiles $\{p_k(t)\}_{k=1}^K$ and instantaneous amplitudes $\{|\widehat{s_k}(n_k)|\alpha_k(t)\}_{k=1}^K$.

\textbf{Step 2:} Initialize: Set up the initial residual $r(t)=f(t)$ and the tolerance $\epsilon$. Let $f_k(t)=0$ be the initial guess of the $k$th general mode and denote $S_k=\emptyset$ as the initial guess of the spectrum information of the $k$th general shape function $s_k$ for $k=1$, $\dots$, $K$.

\textbf{Step 3:} For $k=1$, $\dots$, $K$,  compute the inverse-warping profiles in $[0,1]$ by
\begin{eqnarray*}
h_k(t)& =& \frac{r\circ p_k^{-1}(t)}{|\widehat{s_k}(n_k)|\alpha_k\circ p_k^{-1}(t)}.
\end{eqnarray*}

\textbf{Step 4:} Apply the discrete Fourier transform on $h_k(t)$ in $[0,1]$ to obtain $\widehat{h_k}(\xi)$ for $k=1$, $\dots$, $K$ and solve the following optimization problem,
\[
(\tau,j)=\underset{(\xi,k)}{\arg\max} |\widehat{h_{k}}(\xi)|.
\]
Then $\tau\approx \frac{nm_j}{n_j}$ for some $n$ such that $\widehat{s_j}(n)\neq 0$.

\textbf{Step 5:} Let $g(t)=e^{2\pi i\tau t}$. Warp the harmonic $g(t)$ with the $j$th instantaneous phase profile $p_j(t)$ and multiply the warped harmonic by the $j$th instantaneous amplitude $|\widehat{s_j}(n_j)|\alpha_j(t)$ to obtain
\begin{eqnarray*}
|\widehat{s_j}(n_j)|\alpha_j(t) g\circ p_j(t)&\approx&|\widehat{s_j}(n_j)|\alpha_j(t) e^{2\pi i\frac{nm_j}{n_j}p_j(t)}\nonumber\\
&=&|\widehat{s_j}(n_j)|\alpha_j(t) e^{2\pi i n N_j(\phi_j(t)-\phi_j(0))}\nonumber\\
&=&|\widehat{s_j}(n_j)|e^{-2\pi i n N_j\phi_j(0)}\alpha_j(t) e^{2\pi i n N_j\phi_j(t)}.
\end{eqnarray*}

\textbf{Step 6:} Solve the $L^2$ minimization problem for a complex factor $\beta\in \C$ such that
\[
\beta =  \underset{\beta\in\C}{\arg\min} \|r(t)-\beta|\widehat{s_j}(n_j)|\alpha_j(t) g\circ p_j(t)\|_{L^2}.
\]
Then 
\[
\beta|\widehat{s_j}(n_j)|\alpha_j(t) g\circ p_j(t)\approx\widehat{s_j}(n)\alpha_j(t)e^{2\pi inN_j\phi_j(t)},
\]
which implies
\[
|\beta|=\frac{|\widehat{s_j}(n)|}{|\widehat{s_j}(n_j)|}.
\]

\textbf{Step 7:} Update: Compute the new residual \[ r(t)=r(t)-\beta|\widehat{s_j}(n_j)|\alpha_j(t) g\circ p_j(t).\] Update the $j$th recovered general mode \[f_j(t)=f_j(t)+\beta|\widehat{s_j}(n_j)|\alpha_j(t) g\circ p_j(t),\] 
and the $j$th spectrum information set \[S_j=S_j\cup\{(\tau,|\beta|)\}.\]

\textbf{Step 8:} If $\|r(t)\|_{L^2}>\epsilon$, repeat step $3$-$7$. Otherwise, stop iterating and export the general mode estimates $f_k$ and the spectrum information $S_k$ for $k=1$, $\dots$, $K$.

Notice that for each pair $(\tau,|\beta|)\in S_k$, $(\tau,|\beta|)\approx(\frac{nm_k}{n_k},\frac{|\widehat{s_k}(n)|}{|\widehat{s_k}(n_k)|})$ for some $n$ such that $\widehat{s_k}(n)\neq 0$. For each $k$, let 
  \begin{equation}
    d_k(\xi)=
    \begin{cases}
      |\beta|,
      & \text{for }\xi=\tau\text{ such that }(\tau,\beta)\in S_k\nonumber\\
      0, 
      & otherwise.
    \end{cases}
  \end{equation}
Then 
\[
|\widehat{s_k}(n)|\approx |\widehat{s_k}(n_k)|d_k(\frac{m_kn}{n_k})\Rightarrow d_k(\xi)\approx \frac{1}{|\widehat{s_k}(n_k)|}\left|\widehat{s_k}(\frac{n_k}{m_k}\xi)\right|.
\]
Hence, $d_k(\xi)$ is an approximation of the spectral energy $|\widehat{s_k}(\xi)|$ up to a constant factor and a scaling.

The DSA method proposed above can take into account all the Fourier expansion terms, even if there are weak energy terms and crossover frequencies. Let us consider the \textbf{Example $1$} again. The reconstructed general modes recovered by the DSA method shown in Figure \ref{fig:s1s2mode2}  are exactly the desired general modes.

\begin{figure}[ht!]
  \begin{center}
    \begin{tabular}{cc}
\includegraphics[height=1.6in]{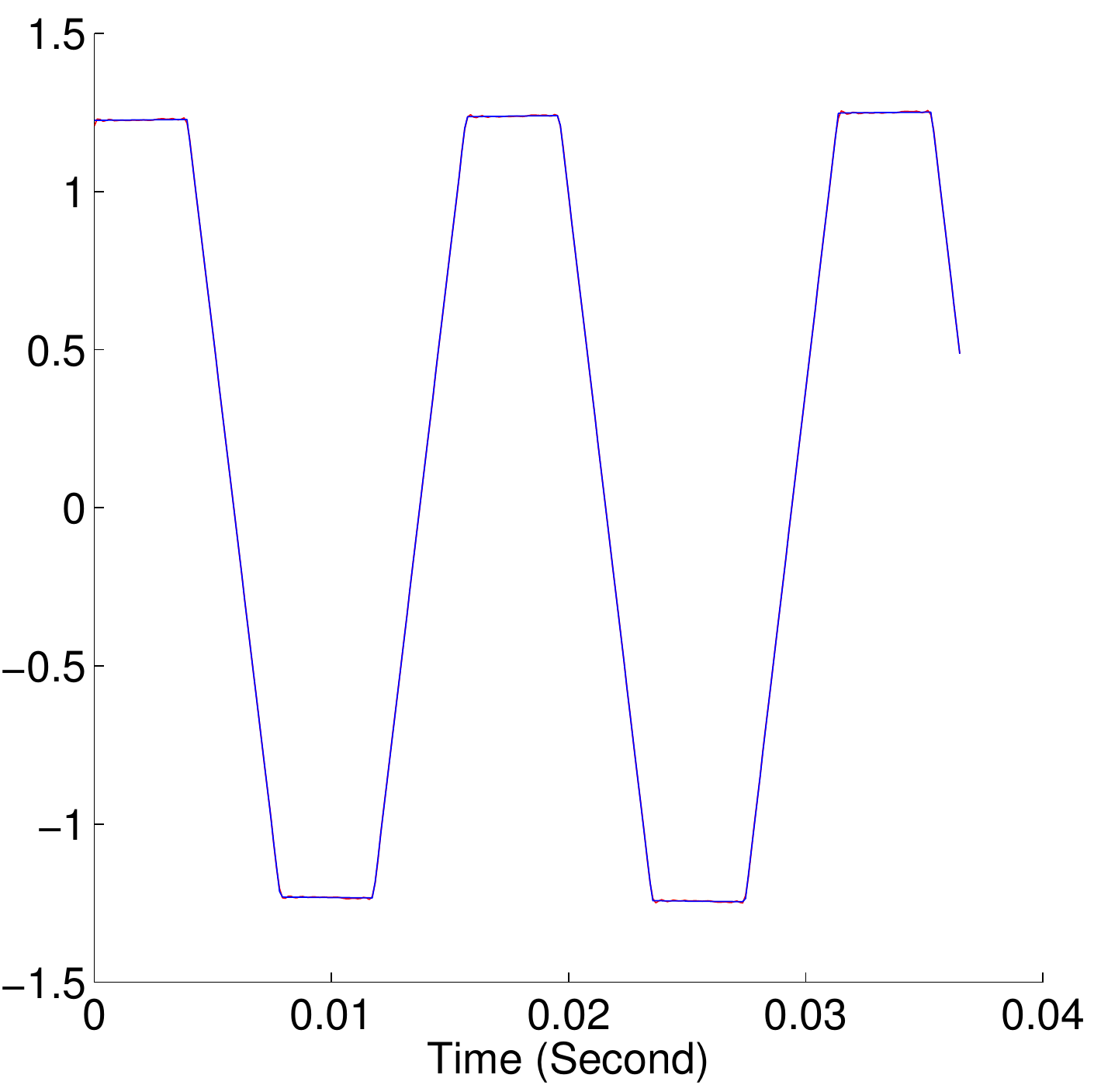} & \includegraphics[height=1.6in]{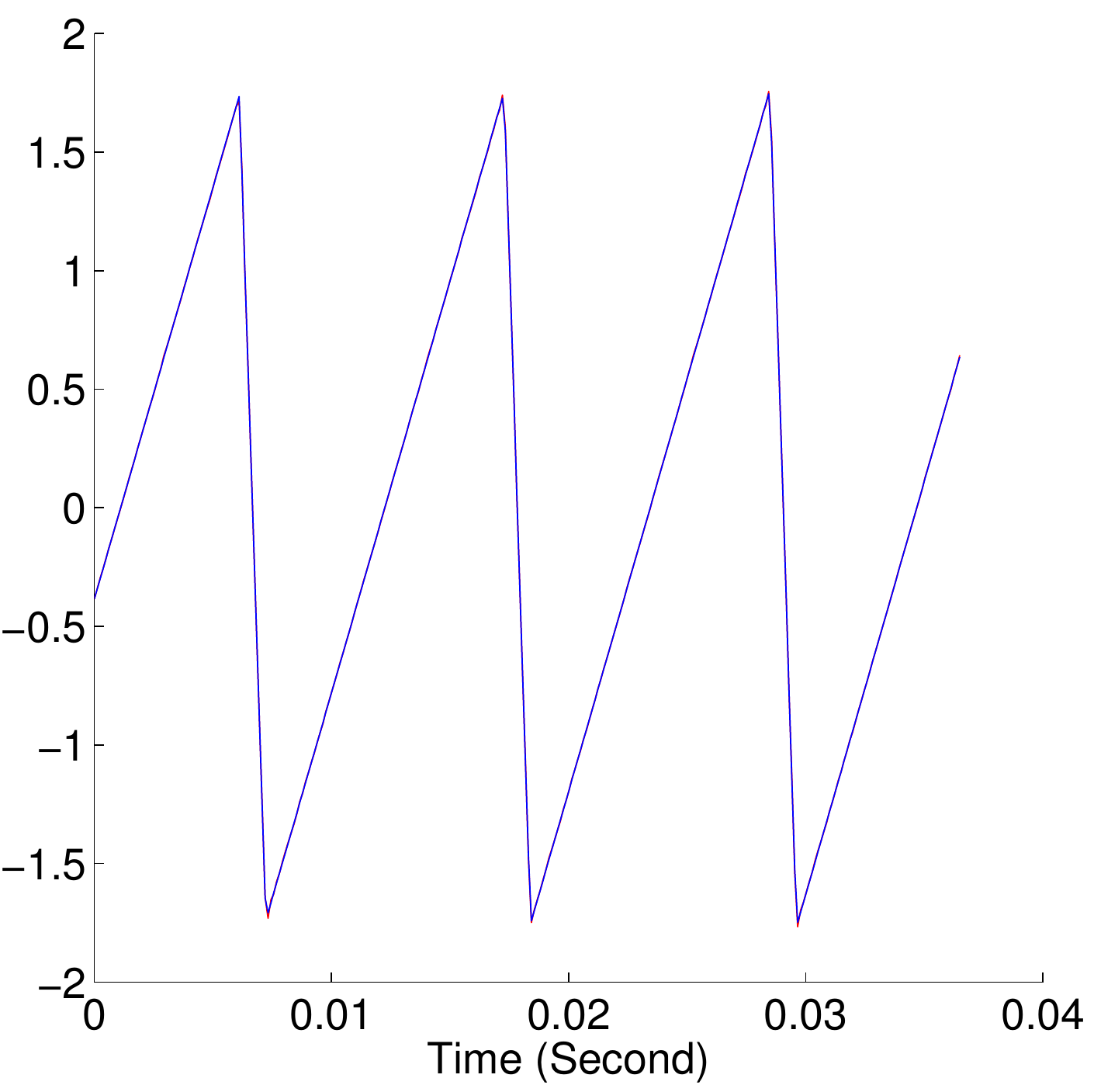}
    \end{tabular}
  \end{center}
  \caption{Blue: Real signals. Red: Reconstructed results. Two recovered general modes provided by the DSA method.}
\label{fig:s1s2mode2}
\end{figure}

\section{Analysis of the 1D SSWPT}
\label{sec:SSWPT}
In this section, we provide rigorous analysis of the 1D SSWPT for the general mode decomposition problem following the model in \cite{SSWPT}. 
\subsection{General mode decomposition problems}
\label{sub:GMD}
\begin{definition} General shape functions:
\label{def:GSF}

The general shape function class ${\cal S}_M$ consists of $2\pi$-periodic functions $s(t)$ in the Wiener Algebra with a unit $L^2([-\pi,\pi])$-norm and a $L^\infty$-norm bounded by $M$ satisfying the following spectral conditions:
\begin{enumerate}
\item The Fourier series of $s(t)$ is uniformly convergent;
\item $\sum_{n=-\infty}^{\infty}|\widehat{s}(n)|\leq M$ and $\widehat{s}(0)=0$;
\item Let $\Lambda$ be the set of integers $\{|n|: \widehat{s}(n)\neq 0\}$. The greatest common divisor $\gcd(s)$ of all the elements in $\Lambda$ is $1$.
\end{enumerate}
\end{definition}
In fact, if $\gcd(s)>1$, then the general mode $s(2\pi N\phi(t))$ can be considered as a more oscillatory mode $\tilde{s}(2\pi \gcd(s)N\phi(t))$ with $\gcd(\tilde{s})=1$ and the Fourier coefficients $\widehat{\tilde{s}}(n)=\widehat{s}(\gcd(s)n)$. The requirement that $\widehat{s}(0)=0$ and $s$ has a unite $L^2([-\pi,\pi])$-norm is to normalize the general shape function.

\begin{definition}
  \label{def:GIMTF}
  A function $f(t)=\alpha(t)s(2\pi N \phi(t))$ is a general intrinsic mode type
  function (GIMT) of type $(M,N)$, if $s(t)\in {\cal S}_M$ and $\alpha(t)$ and $\phi(t)$ satisfy the conditions below.
\begin{align*}
    \alpha(t)\in C^\infty, \quad |\alpha'|\leq M, \quad 1/M \leq \alpha\leq M \\
    \phi(t)\in C^\infty,  \quad  1/M \leq | \phi'|\leq M, \quad |\phi''|\leq M.
   \end{align*}
\end{definition}

\begin{definition}
  \label{def:GSWSIMC}
  A function $f(t)$ is a well-separated general superposition of type
  $(M,N,K,s)$, if
  \[
  f(t)=\sum_{k=1}^K f_k(t),
  \] 
  where each $f_k(t)=\alpha_k(t)s_k(2\pi N_k \phi_k(t))$ is a GIMT of type $(M,N_k)$ such that $N_k\geq N$ and the phase functions satisfy the
  separation condition: for any pair $(a,b)$, there exists at most one pair $(n,k)$ such that $\widehat{s_k}(n)\neq 0$ and that
\[
|a|^{-s}|a-nN_k\phi_k'(b)|< d.
\]
We denote by $GF(M,N,K,s)$ the set of all such functions.
\end{definition}

Notice that the mode decomposition problem of the form \eqref{P1} is a special case of the general mode decomposition problem of the form \eqref{P2}. We therefore only analyze the 1D SSWPT for the general mode decomposition problem. Besides, the components are not necessarily defined in the whole domain $\R$, because the synchrosqueezed transforms are localized so as to capture the non-linear, non-stationary features of signals as illustrated in \cite{Daubechies2011,SSWPT}. For the sake of convenience, we omit the discussion of the data segments.

\subsection{Instantaneous frequency estimates}
With the definitions above, we are ready to present the theorems for the 1D SSWPT with a geometric scaling parameter $s\in(1/2,1)$. The estimates of the instantaneous frequencies $\{N_k\phi_k'(t)\}_{k=1}^K$ rely on Theorem \ref{thm:main}, Algorithm \ref{alg:curve}, and Theorem \ref{thm:inst2} below. 

\begin{theorem}
  \label{thm:main}
  For a function $f(t)$ and $\eps>0$, we define
  \[
  R_{\eps} = \{(a,b): |W_f(a,b)|\geq |a|^{-s/2}\sqrt \eps\}
  \]
  and 
  \[
  Z_{n,k} = \{(a,b): |a-nN_k\phi_k'(b)|\leq d|a|^s \}
  \]
  for $1\le k\le K$ and $|n|\geq 1$. For fixed $M$ and $K$, for any $\eps>0$, there
  exists a constant $N_0(M,K,s,\eps)>0$ such that for any
  $N>N_0(M,K,s,\eps)$ and $f(t)\in GF(M,N,K,s)$ the following statements
  hold.
  \begin{enumerate}[(i)]
  \item $\{Z_{n,k}: 1\le k \le K, \widehat{s_k}(n)\neq 0\}$ are disjoint and $R_{\eps}
    \subset \bigcup_{1\le k \le K} \bigcup_{ \widehat{s_k}(n)\neq 0} Z_{n,k}$;
  \item For any $(a,b) \in R_{\eps} \cap Z_{n,k}$, 
    \[
    \frac{|v_f(a,b)-nN_k\phi_k'(b)|}{ |nN_k \phi_k'(b)|}\lesssim\sqrt \eps.
    \]
  \end{enumerate}
\end{theorem}
The proof of Theorem \ref{thm:main} relies on two lemmas as follows to estimate the asymptotic behavior of $W_f(a,b)$ and $\partial_b W_f(a,b)$ as $N$ going to infinity. In what follows, when we write $O(\cdot)$, $\lesssim$, or $\gtrsim$, the implicit constants may depend on $M$ and $K$.
\begin{lemma} 
\label{lemma1}
Suppose $\Omega_a=\{(k,n):a\in[\frac{nN_k}{2M},2MnN_k]\}$. Under the assumption of Theorem \ref{thm:main}, for any $\eps>0$, we have
\[
W_f(a,b)=|a|^{-s/2}\left(\sum_{(k,n)\in\Omega_a}\widehat{s_k}(n)\alpha_k(b)e^{2\pi inN_k\phi_k(b)}\widehat{w}\left(\left(a-nN_k\phi_k'(b)\right)|a|^{-s}\right)+O(\eps)\right),
\]
when $N$ is sufficiently large.
\end{lemma}
\begin{proof}
Without loss of generality, we can simply assume $N_k=N$ for all $k$ and only prove the case for $a>1$. Because $w(t)$ decays rapidly, the wave packet transform $W_f(a,b)$ is well defined. By the uniform convergence of the Fourier series of $s_k(t)$ and the change of variables, we have
\begin{eqnarray*}
W_f(a,b)&=&\int_{\R}\sum_{k=1}^K\alpha_k(t)s_k(2\pi N\phi_k(t))a^{s/2}w(a^{s}(t-b))e^{-2\pi i(t-b)a}dt\\
&=& a^{-s/2}\sum_{k=1}^K\sum_{|n|\geq 1}\widehat{s_k}(n)\int_{\R}\alpha_k(a^{-s}x+b)w(x)e^{2\pi i(nN\phi_k(a^{-s}x+b)-a^{1-s}x)}dx.
\end{eqnarray*}
Let us estimate $I_{kn}=\widehat{s_k}(n)\int_{\R}\alpha_k(a^{-s}x+b)w(x)e^{2\pi i(nN\phi_k(a^{-s}x+b)-a^{1-s}x)}dx$. Let $h(x)=\widehat{s_k}(n)\alpha_k(a^{-s}x+b)w(x)$ and $g(x)=2\pi(nN\phi_k(a^{-s}x+b)-a^{1-s}x)$, then
\[
I_{kn}=\int_{\R}h(x)e^{ig(x)}dx,
\]
and 
\[
g'(x)=2\pi a^{-s}(nN\phi_k'(a^{-s}x+b)-a).
\]
If $a<\frac{nN}{2M}$, then $|g'(x)|\gtrsim a^{-s}nN\gtrsim (nN)^{1-s}$. If $a>2MnN$, then $|g'(x)|\gtrsim a^{1-s}\gtrsim(nN)^{1-s}$. So, if $a\notin [\frac{nN}{2M},2MnN]$, then $|g'(x)|\gtrsim(nN)^{1-s}$. For real smooth functions $g(x)$, we define the differential operator
\[
L=\frac{1}{i}\frac{\partial_x}{g'}.
\]
Because $h(x)$ decays sufficiently fast at infinity, we perform integration by parts $r$ times to get
\[
\int_{\R}he^{ig}dx=\int_{\R}h(L^re^{ig})dx=\int_{\R}\left((L^*)^rh\right)e^{ig}dx,
\]
where $L^*$ is the adjoint of $L$. A few algebraic calculation shows that $L^*$ contributes a factor of order $\frac{1}{|g'|}\lesssim\frac{1}{(nN)^{1-s}}$ if $a\notin [\frac{nN}{2M},2MnN]$, and we therefore have 
\[
|I_{kn}|=\left|\int_{\R}e^{ig}\left((L^*)^rh\right)dx\right|\lesssim|\widehat{s_k}(n)|(nN)^{-(1-s)r}\lesssim |\widehat{s_k}(n)|\eps.
\]
Since $s<1$ and $\sum_{n=-\infty}^{\infty}|\widehat{s_k}(n)|\leq M$, if $N\gtrsim \eps^{\frac{-1}{(1-s)r}}$, then
\begin{equation}
\label{est:1}
a^{-s/2}\sum_{(k,n)\notin \Omega_a} I_{kn}\lesssim a^{-s/2}\sum_{(k,n)\notin \Omega_a}|\widehat{s_k}(n)|O(\eps)\lesssim a^{-s/2}O(\eps).
\end{equation}

Now let us estimate $I_{kn}$ when $a\in[\frac{nN}{2M},2MnN]$. Recall that
\begin{equation}
I_{kn}=\widehat{s_k}(n)\int_{\R}\alpha_k(a^{-s}x+b)w(x)e^{2\pi i(nN\phi_k(a^{-s}x+b)-a^{1-s}x)}dx.\nonumber
\end{equation}
By Taylor expansion,
\begin{equation}
\alpha_k(a^{-s}x+b)=\alpha_k(b)+\alpha_k'(b^*)a^{-s}x\nonumber
\end{equation}
and 
\begin{equation}
\phi_k(a^{-s}x+b)=\phi_k(b)+\phi_k'(b)a^{-s}x+\frac{1}{2}\phi_k''(b^{**})a^{-2s}x^2\nonumber
\end{equation}
for some $b^*$ and $b^{**}$. Notice that, if $N\gtrsim \eps^{-1/s}$, then
\begin{eqnarray}
& &|I_{kn}-\widehat{s_k}(n)\alpha_k(b)\int_{\R}w(x)e^{2\pi i(nN\phi_k(a^{-s}x+b)-a^{1-s}x)}dx|\nonumber\\
&\lesssim& |\widehat{s_k}(n)|\alpha_k'(b^*)a^{-s}\int_{\R}|x||w(x)|dx\nonumber\\
&\lesssim& |\widehat{s_k}(n)|O(\eps).\nonumber
\end{eqnarray}
This implies that
\begin{equation}
I_{kn}=\left(\widehat{s_k}(n)\alpha_k(b)\int_{\R}w(x)e^{2\pi i(nN\phi_k(a^{-s}x+b)-a^{1-s}x)}dx+|\widehat{s_k}(n)|O(\eps)\right)\nonumber
\end{equation}
for $a\in[\frac{nN}{2M},2MnN]$ and $N\gtrsim\eps^{-1/s}$.
Since $|e^{ix}-1|\leq|x|$, if $N\gtrsim \eps^{-1/(2s-1)}$, then we have 
\begin{eqnarray}
& &|I_{kn}-\widehat{s_k}(n)\alpha_k(b)\int_{\R}w(x)e^{2\pi i(nN\phi_k(b)+nN\phi_k'(b)a^{-s}x-a^{1-s}x)}dx|\nonumber\\
&\lesssim&|\widehat{s_k}(n)|\left(O(\eps)+\left|\alpha_k(b)\int_{\R}w(x)e^{2\pi i(nN\phi_k(b)+nN\phi_k'(b)a^{-s}x-a^{1-s}x)}\left(e^{2\pi inN\frac{1}{2}\phi_k''(b^{**})a^{-2s}x^2}-1\right)dx\right|\right)\nonumber\\
&\lesssim&|\widehat{s_k}(n)|\left(O(\eps)+nNa^{-2s}\int_{\R}x^2|w(x)|dx\right)\nonumber\\
&\lesssim&|\widehat{s_k}(n)|O(\eps).\nonumber
\end{eqnarray}
Hence, it holds that
\begin{eqnarray}
\label{est:2}
I_{kn}
&=&\left(\widehat{s_k}(n)\alpha_k(b)e^{2\pi inN\phi_k(b)}\widehat{w}\left((a-nN\phi_k'(b))a^{-s}\right)+|\widehat{s_k}(n)|O(\eps)\right),
\end{eqnarray}
if $a\in[\frac{nN}{2M},2MnN]$ and $N\gtrsim\max\{\eps^{-1/s},\eps^{-1/(2s-1)}\}=\eps^{-1/(2s-1)}$.

In sum, by \eqref{est:1} and \eqref{est:2}, we arrive at
\begin{eqnarray}
W_f(a,b)\nonumber
&=&a^{-s/2}\left( \sum_{(k,n)\in\Omega_a}I_{kn}+\sum_{(k,n)\notin\Omega_a}I_{kn} \right)\nonumber\\
&=&|a|^{-s/2}\left(\sum_{(k,n)\in\Omega_a}\widehat{s_k}(n)\alpha_k(b)e^{2\pi inN\phi_k(b)}\widehat{w}\left(\left(a-nN\phi_k'(b)\right)|a|^{-s}\right)+O(\eps)\right),\nonumber
\end{eqnarray}
if $N\gtrsim \max\{\eps^{\frac{-1}{(1-s)r}},\eps^{\frac{-1}{2s-1}}\}$.

Similar argument can prove the above conclusion for $a<-1$ and it is simple to generalize it for different $N_k$ to complete the proof.
\end{proof}

The next lemma is to estimate $\partial_b W_f(a,b)$ when $\Omega_a=\{(k,n):a\in[\frac{nN_k}{2M},2MnN_k]\}$ is not empty, i.e., when $W_f(a,b)$ is relevant.

\begin{lemma} 
\label{lemma2}
Suppose $\Omega_a=\{(k,n):a\in[\frac{nN_k}{2M},2MnN_k]\}$ is not empty. Under the assumption of Theorem \ref{thm:main}, for any $\eps>0$, we have
\begin{eqnarray}
& &\partial_b W_f(a,b)\nonumber\\
&=&|a|^{-s/2}\left(\sum_{(k,n)\in\Omega_a}2\pi inN_k\widehat{s_k}(n)\alpha_k(b)\phi_k'(b)e^{2\pi inN_k\phi_k(b)}\widehat{w}\left(\left(a-nN_k\phi_k'(b)\right)|a|^{-s}\right)+|a|O(\eps)\right),\nonumber
\end{eqnarray}
when $N$ is sufficiently large.
\end{lemma}
\begin{proof}
Similar to the proof of Lemma \ref{lemma1}, we can assume $N_k=N$ for all $k$ and only need to prove the case when $a>1$. By the definition of the wave packet transform, we have
\begin{eqnarray*}
\partial_bW_f(a,b)
&=& \sum_{k=1}^K2\pi ia^{1+s/2}\int_{\R}\alpha_k(t)s_k(2\pi N\phi_k(t))w(a^s(t-b))e^{-2\pi i(t-b)a}dt\\
& &-\sum_{k=1}^Ka^{3s/2}\int_{\R}\alpha_k(t)s_k(2\pi N\phi_k(t))w'(a^s(t-b))e^{-2\pi i(t-b)a}dt.
\end{eqnarray*}
Denote the first term by $T_1$ and the second term by $T_2$. By a similar discussion in the proof of Lemma \ref{lemma1}, we have the following asymptotic estimates when $N$ is sufficiently large.
\begin{eqnarray}
T_2\nonumber
&=&-a^{s/2}\sum_{k=1}^K\sum_{|n|\geq 1}\widehat{s_k}(n)\int_{\R}\alpha_k(a^{-s}x+b)w'(x)e^{2\pi i(nN\phi_k(a^{-s}x+b)-a^{1-s}x)}dx\nonumber\\
&=&-a^{s/2}\sum_{(k,n)\in\Omega_a}\int_{\R}\widehat{s_k}(n)\alpha_k(a^{-s}x+b)w'(x)e^{2\pi i(nN\phi_k(a^{-s}x+b)-a^{1-s}x)}dx+a^{s/2}O(\eps)\nonumber\\
&=&a^{s/2}\sum_{(k,n)\in \Omega_a}\int_{\R}\widehat{s_k}(n)w(x)\alpha_k(a^{-s}x+b)e^{2\pi i(nN\phi_k(a^{-s}x+b)-a^{1-s}x)}\nonumber\\
& &\left(2\pi inN\phi_k'(a^{-s}x+b)a^{-s}-2\pi i a^{1-s}\right)dx+a^{-s/2}\sum_{(k,n)\in \Omega_a}\int_{\R}\widehat{s_k}(n)w(x)\nonumber\\
& &\alpha_k'(a^{-s}x+b)e^{2\pi i(nN\phi_k(a^{-s}x+b)-a^{1-s})}dx+a^{s/2}O(\eps)\nonumber\\
&=&a^{-s/2}\sum_{(k,n)\in \Omega_a}2\pi inN\int_{\R}\widehat{s_k}(n)\phi_k'(a^{-s}x+b)\alpha_k(a^{-s}x+b)w(x)e^{2\pi i(nN\phi_k(a^{-s}x+b)-a^{1-s}x)}dx\nonumber\\
& &-a^{1-s/2}\sum_{(k,n)\in \Omega_a}2\pi i\int_{\R}\widehat{s_k}(n)w(x)\alpha_k(a^{-s}x+b)e^{2\pi i(nN\phi_k(a^{-s}x+b)-a^{1-s}x)}dx\nonumber\\
& &+a^{-s/2}O(1)+a^{s/2}O(\eps)\nonumber\\
&=&a^{-s/2}\sum_{(k,n)\in \Omega_a}2\pi inN\left(\widehat{s_k}(n)\phi_k'(b)\alpha_k(b)e^{2\pi inN\phi_k(b)}\widehat{w}(a^{-s}(a-nN\phi_k'(b)))+|\widehat{s_k}(n)|O(\eps)\right)\nonumber\\
& &-a^{1+s/2}\sum_{(k,n)\in \Omega_a}\widehat{s_k}(n)2\pi i\int_{\R}\alpha_k(t)w(a^s(t-b))e^{2\pi i(nN\phi_k(t)-(t-b)a)}dt\nonumber\\
& &+a^{-s/2}O(1)+a^{s/2}O(\eps),\nonumber
\end{eqnarray}
if $N\gtrsim\max\{\eps^{\frac{-1}{(1-s)r}},\eps^{\frac{-1}{2s-1}}\}$. The third equality holds by integration by parts and the last equality holds by changing variables. Notice that
\begin{eqnarray}
T_1\nonumber
&=&a^{1+s/2}\sum_{(k,n)\in \Omega_a}\widehat{s_k}(n)2\pi i\int_{\R}\alpha_k(t)w(a^s(t-b))e^{2\pi i(nN\phi_k(t)-(t-b)a)}dt\nonumber\\
& &+\sum_{(k,n)\notin \Omega_a}2\pi ia^{1-s/2}\widehat{s_k}(n)\int_{\R}\alpha_k(a^{-s}x+b)w(x)e^{2\pi i(nN\phi_k(a^{-s}x+b)-a^{1-s}x)}dx\nonumber\\
&=&a^{1+s/2}\sum_{(k,n)\in \Omega_a}\widehat{s_k}(n)2\pi i\int_{\R}\alpha_k(t)w(a^s(t-b))e^{2\pi i(nN\phi_k(t)-(t-b)a)}dt+a^{1-s/2}O(\eps),\nonumber
\end{eqnarray}
if $N\gtrsim \eps^{\frac{-1}{(1-s)r}}$ for any $r>=1$. Hence $T_1+T_2$ results in 
\begin{eqnarray}
& &\partial_bW_f(a,b)\nonumber\\
&=&a^{-s/2}\sum_{(k,n)\in \Omega_a}2\pi inN\left(\widehat{s_k}(n)\phi_k'(b)\alpha_k(b)e^{2\pi inN\phi_k(b)}\widehat{w}(a^{-s}(a-nN\phi_k'(b)))+|\widehat{s_k}(n)|O(\eps)\right)\nonumber\\
& &+a^{-s/2}O(1)+a^{s/2}O(\eps)+a^{1-s/2}O(\eps)\nonumber\\
&=&|a|^{-s/2}\left(\sum_{(k,n)\in\Omega_a}2\pi inN\widehat{s_k}(n)\alpha_k(b)\phi_k'(b)e^{2\pi inN\phi_k(b)}\widehat{w}\left(\left(a-nN\phi_k'(b)\right)|a|^{-s}\right)+|a|O(\eps)\right),\nonumber
\end{eqnarray}
if $N$ is sufficiently large. So, the Lemma \ref{lemma2} is proved.
\end{proof}
We are now ready to prove Theorem \ref{thm:main} with Lemma \ref{lemma1} and Lemma \ref{lemma2}.
\begin{proof}
Let us first consider $(i)$. The well-separation condition implies that $\{Z_{n,k}:1\leq k\leq K, \widehat{s_k}(n)\neq 0\}$ are disjoint. Let $(a,b)$ be a point in $R_{\eps}$, then $|W_f(a,b)|\geq a^{-s/2}\sqrt{\eps}$, which means that $\Omega_a$ is not empty and $\exists (k,n)\in\Omega_a$ such that $\widehat{w}((a-nN_k\phi_k'(b))a^{-s})\neq 0$. Because the support of $\widehat{w}(\xi)$ is $(-d,d)$, we know $|a-nN_k\phi_k'(b)|\leq a^sd$, i.e., $(a,b)\in Z_{n,k}$. Hence, $R_\eps\subset \bigcup_{1\le k \le K} \bigcup_{ \widehat{s_k}(n)\neq 0} Z_{n,k}$.

To show $(ii)$, let us recall that $v_f(a,b)$ is defined as
\[
v_f(a,b)=\frac{\partial_bW_f(a,b)}{2\pi iW_f(a,b)},
\]
for $W_f(a,b)\neq 0$. If $(a,b)\in R_\eps\bigcap Z_{n,k}$, then by Lemma \ref{lemma1}
\begin{eqnarray}
W_f(a,b)&=&|a|^{-s/2}\left(\sum_{(k,n)\in\Omega_a}\widehat{s_k}(n)\alpha_k(b)e^{2\pi inN_k\phi_k(b)}\widehat{w}\left(\left(a-nN_k\phi_k'(b)\right)|a|^{-s}\right)+O(\eps)\right)\nonumber\\
&=&|a|^{-s/2}\left(\widehat{s_k}(n)\alpha_k(b)e^{2\pi inN_k\phi_k(b)}\widehat{w}\left(\left(a-nN_k\phi_k'(b)\right)|a|^{-s}\right)+O(\eps)\right),\nonumber
\end{eqnarray}
as the other terms drop out, since $\{Z_{n,k}\}$ are disjoint.
Similarly, by Lemma \ref{lemma2}
\begin{eqnarray}
& &\partial_b W_f(a,b)\nonumber\\
&=&|a|^{-s/2}\left(2\pi inN_k\widehat{s_k}(n)\alpha_k(b)\phi_k'(b)e^{2\pi inN_k\phi_k(b)}\widehat{w}\left(\left(a-nN_k\phi_k'(b)\right)|a|^{-s}\right)+|a|O(\eps)\right).\nonumber
\end{eqnarray}
Let $g$ denote the term $\widehat{s_k}(n)\alpha_k(b)e^{2\pi inN_k\phi_k(b)}\widehat{w}\left(\left(a-nN_k\phi_k'(b)\right)|a|^{-s}\right)$, then 
\begin{eqnarray}
v_f(a,b)&=&\frac{nN_k\phi_k'(b)g+|a|O(\eps)}{g+O(\eps)}\nonumber\\
&=&\frac{nN_k\phi_k'(b)(g+O(\eps))}{g+O(\eps)},\nonumber
\end{eqnarray}
since $a\in[\frac{nN_k}{2M},2MnN_k]$. Because $|W_f(a,b)|\geq a^{-s/2}\sqrt{\eps}$ for $(a,b)\in R_\eps$, then $|g|\gtrsim \sqrt{\eps}$. Therefore
\[
\frac{|v_f(a,b)-nN_k\phi_k'(b)|}{|nN_k\phi_k'(b)|}\lesssim\left|\frac{O(\eps)}{g+O(\eps)}\right|\lesssim\sqrt{\eps}.
\]
\end{proof}

Theorem \ref{thm:main} shows that the instantaneous frequency information function $v_f(a,b)$ can estimate $nN_k\phi_k'(t)$ accurately for a class of superpositions of general mode functions if their phases are sufficiently steep. This guarantees the well concentration of the synchrosqueezed energy distribution $T_f(v,b)$ around $nN_k\phi_k'(t)$. Hence, $\psi_{kn}(b)$ defined in the introduction of the GMDWP method is an accurate estimate of $nN_k\phi_k'(b)$. Next, a curve classification method and an instantaneous frequency identification method are introduced below.

Let us reindex the functions $\{\psi_{kn}:1\leq k\leq K,|n|\geq 1\}$ by $\{\psi_j\}_{1\leq j\leq L}$. Our goal is to obtain $K$ index sets $\{\Lambda_k\}_{1\leq k\leq K}$ such that $\{\psi_{kn}\}_{|n|\geq 1}=\{\psi_j\}_{j\in \Lambda_k}$. Because $\frac{\psi_{kn_1}}{\psi_{kn_2}}\approx \frac{n_1}{n_2}$, $\{\psi_{kn}\}_{|n|\geq 1}$ can approximately be considered as a set in a one dimensional point set in a high dimensional space. Hence, the curve classification can be considered as a subspace clustering problem studied in \cite{subspace,subspace2}. This motivates the following method to classify $\{\psi_j\}_{1\leq j\leq L}$. This method is similar to the method in \cite{subspace} for subspace clustering. 

\begin{algo}
\label{alg:curve}
  Curve classification of $\{\psi_j\}_{1\leq j\leq L}$
  \begin{algorithmic}[1]
    \STATE For each pair $(k,j)$ with $k\neq j$, let $m_k=\|\psi_k\|_{L^\infty}$ and $m_j=\|\psi_j\|_{L^\infty}$. Compute the linear regression of $\frac{\psi_km_j}{\psi_jm_k}$ and its residual $r_{kj}$. Let $R$ be the $L\times L$ matrix such that $R_{kj}=r_{kj}$.

   \STATE Set up a variance parameter $\sigma^2$ and define a Gaussian function $g(x)=e^{-\frac{x^2}{2\sigma^2}}$. 

   \STATE Form the affinity graph $G$ with nodes representing the $L$ instantaneous frequencies and edge weights given by $g(R_{kj})+g(R_{jk})$.

   \STATE Compute the eighenvalues of the normalized Laplacian of $G$ and sort them in descending order $\sigma_1\geq\sigma_2\geq\cdots\geq\sigma_L$. Let \[K=L-\underset{i=1,\cdots,L-1}{\arg\max}\left(\sigma_i-\sigma_{i+1}\right).\]
    
    \STATE Apply the spectral clustering method proposed in \cite{Andrew2001} and used in \cite{SSWPT} to the affinity graph using $K$ as the number of curve class to divide the index set $\{1,\dots,L\}$ into $K$ subsets $\{\Lambda_k\}_{1\leq k\leq K}$.

    \STATE Associate each index set $\Lambda_k$ with a class of curves $\{\psi_j\}_{j\in \Lambda_k}$.
    
  \end{algorithmic}
\end{algo}

If $\{\phi_k'(t)\}_{1\leq k\leq K}$ are not very similar, then the residual of the linear regression of $\frac{\phi_k'(t)}{\phi_j'(t)}$ is large for $k\neq j$. By setting up a proper parameter $\sigma$, Algorithm \ref{alg:curve} can classify $\{\psi_{kn}:1\leq k\leq K,|n|\geq 1\}$ accurately with high probability. In the case in which instantaneous frequencies are disturbed by noise, robust subspace clustering techniques in \cite{subspace2} can be applied.

Algorithm \ref{alg:curve} results in $K$ classes of curves $\{\psi_j\}_{j\in \Lambda_k}=\{\psi_{kn}\}_{|n|\geq 1}\approx \{nN_k\phi_k'\}_{|n|\geq 1}$ for $1\leq k\leq K$. The theorems below show how to estimate the instantaneous frequency $N_k\phi_k'(t)$ of the general mode $\alpha_k(t)s_k(2\pi N_k\phi_k(t))$.

\begin{theorem}
\label{thm:inst}
Suppose $\alpha(t)s(2\pi N_0\phi(t))$ is a general intrinsic mode function of type $(M_0,N_0)$ and $s(t)$ has some nonzero Fourier coefficients $\{\widehat{s}(n_i)\}_{1\leq i\leq N}$ such that $\widehat{s}(n)=0$ for $n<|n_1|$ and $\gcd (|n_1|,\dots,|n_N|)=1$. If we know $\psi_i(t)= n_iN_0\phi'(t)$ for ${1\leq i\leq N}$ and
\begin{equation}
\label{P3}
n_0=\min \left\{n:1\leq n\leq M,\frac{n\psi_i(t)}{\psi_1(t)} \text{ is a constant integer for }2\leq i\leq N\right\},
\end{equation}
where $M=M_0\min{|\psi_1(t)|}$, then the instantaneous frequency $N_0\phi'(t)$ of $\alpha(t)s(2\pi N_0\phi(t))$ is $\frac{|\psi_1(t)|}{n_0}$.
\end{theorem}
\begin{proof}
\begin{eqnarray}
& &\frac{n\psi_i(t)}{\psi_1(t)} \text{ is a constant integer for }2\leq i\leq N \Rightarrow \frac{nn_i}{n_1} \text{ is an integer for }2\leq i\leq N\nonumber\\
&\Rightarrow& \frac{n_1}{\gcd (n,n_1)}\big| n_i \text{ for }2\leq i\leq N \Rightarrow \frac{n_1}{\gcd (n,n_1)}\big| \gcd (n_2,\dots,n_N)\nonumber\\
&\Rightarrow& \gcd (n,n_1)=n_1 \Rightarrow n=k|n_1| \text{ for some integer }k\geq 1.\nonumber
\end{eqnarray}
Hence, $ |n_1|=\min \left\{n:1\leq n\leq M,\frac{n\psi_i(t)}{\psi_1(t)} \text{ is a constant integer for }2\leq i\leq N\right\}=n_0$, which implies $N_0\phi'(t)=\frac{|\psi_1(t)|}{n_0}$.
\end{proof}

Determining whether $\frac{n\psi_i(t)}{\psi_1(t)}$ is a constant integer is not practical unless the instantaneous frequencies are exactly recovered.
This motivates the design of the following method.

\begin{theorem}
\label{thm:inst2}
Suppose the same condition of Theorem \ref{thm:inst} holds and $n_0$ is the solution of the following minimization problem,
\begin{equation}
\label{P4}
n_0=\min\left(  \underset{1\leq n\leq M}{\arg\min} \left\{\frac{1}{N-1}\sum_{i=2}^N \|\frac{n\psi_i(t)}{\psi_1(t)} - \lfloor \frac{n\psi_i(t)}{\psi_1(t)} +0.5\rfloor \|_{L^2}^2\right\} \right),
\end{equation}
where $M=M_0\min{|\psi_1(t)|}$. Then the instantaneous frequency $N_0\phi'(t)$ of $s(2\pi N_0\phi(t))$ is $\frac{|\psi_1(t)|}{n_0}$.
\end{theorem}
\begin{proof}
Let $f(n)=\frac{1}{N-1}\sum_{i=2}^N \|\frac{n\psi_i(t)}{\psi_1(t)} - \lfloor\frac{n\psi_i(t)}{\psi_1(t)}+0.5\rfloor \|_{L^2}^2$. Notice that 
\[
\frac{n\psi_i(t)}{\psi_1(t)} - \lfloor\frac{n\psi_i(t)}{\psi_1(t)} +0.5\rfloor=\frac{nn_i}{n_1}-\lfloor\frac{nn_i}{n_1}+0.5\rfloor. 
\]
If $n=k|n_1|$ for some integer $k\geq 1$, then $f(n)=0$.

Otherwise, then by the proof of Theorem \ref{thm:inst}, $\exists i$, $2\leq i \leq N$ such that $\frac{nn_i}{n_1}$ is not an integer. Then 
\[
\|\frac{n\psi_i(t)}{\psi_1(t)} - \lfloor \frac{n\psi_i(t)}{\psi_1(t)}+0.5\rfloor \|_{L^2}^2>0,
\]
which implies $f(n)>0$. So, $|n_1|=\min\left(\underset{1\leq n\leq M}{\arg\min}  f(n)\right)$. Therefore, $N_0\phi'(t)=\frac{|\psi_1(t)|}{n_0}$.
\end{proof}
Notice that the function $g(x)=|x- \lfloor x+0.5\rfloor |$ is absolutely continuous. If $\psi_i(t)\approx n_iN_0\phi'(t)$ is sufficiently accurate, the conclusion of Theorem \ref{thm:inst2} is still true.

For each $k$, Theorem \ref{thm:inst2} estimates the instantaneous frequency of the general mode $\alpha_k(t)s_k(2\pi N_k\phi_k(t))$ using the result of Theorem \ref{thm:main} and Algorithm \ref{alg:curve}. This completes the estimates of the instantaneous frequencies of the general modes from a superposition of the form \eqref{P2}.

\subsection{The analysis of spectral resolution}
\label{sub:RA}
In Theorem \ref{thm:main}, the lower bound $s>1/2$ ensures that the wave packets is sufficiently localized in space so that it can reflect the second order properties of the phase functions precisely. The upper bound $s<1$ enables the SSWPT to detect a more general class of shape functions compared to the wave shape functions defined in \cite{Hau-Tieng2013}. An intuitive reason for this more general result is that the supports of wave packets in the Fourier domain are more localized than those of wavelets, resulting in a better resolution for mode decompositions. In what follows, the single scale resolution and the multiscale resolution of synchrosqueezed transforms will be defined and studied.

\begin{definition}
\label{def:singlescale}
The single scale resolution at a level $N$ of a synchrosqueezed transform is $\frac{1}{N\lambda_0}-\frac{1}{N}$, where $\lambda_0\in(0,1)$ is the critical number such that $\forall \lambda\in(0,\lambda_0)$, the synchrosqueezed transform is able to distinguish two modes $f_1(t)=e^{2\pi i Nt}$ and $f_2(t)=e^{2\pi i\lambda N t}$ from their superposition $f(t)=f_1(t)+f_2(t)$.
\end{definition}
The single scale resolution analysis is related to the beating phenomenon of the EMD method in \cite{EMD_ans}. In \cite{Syn_ans}, the authors have proved a conclusion which is equivalent to the fact that the single scale resolution at the level $1$ of the SSWT with a mother wavelet supported in an interval of size $2d$ is $\frac{2d}{1-d}$. As we shall prove below, the SSWPT has a higher single scale resolution than the SSWT and a smaller geometric scale parameter $s$ benefits a higher resolution. This means that the SSWPT has a better ability to distinguish two harmonics with close frequencies. 

Recall that the wave packet transform is controlled by the geometric parameter $s$ and the parameter $d$ for the size of the support of the mother wave packet in the Fourier domain. Consider two complex harmonics $f_1(t)=e^{2\pi i Nt}$, $f_2(t)=e^{2\pi i\lambda N t}$, and their superposition $f(t)=f_1(t)+f_2(t)$. Then the wave packet transform of $f$ is
\begin{equation}
\label{rel:1}
W_{f}(a,b)=|a|^{-s/2}e^{2\pi iNb}\widehat{w}(|a|^{-s}(a-N))+|a|^{-s/2}e^{2\pi iN\lambda b}\widehat{w}(|a|^{-s}(a-N\lambda)),\nonumber
\end{equation}
and the instantaneous frequency information function is
\[
v_f(a,b)=\frac{N\left(e^{2\pi iNb}\widehat{w}(|a|^{-s}(a-N))+\lambda e^{2\pi iN\lambda b}\widehat{w}(|a|^{-s}(a-N\lambda ))\right)}{e^{2\pi iNb}\widehat{w}(|a|^{-s}(a-N))+ e^{2\pi iN\lambda b}\widehat{w}(|a|^{-s}(a-N\lambda ))}.
\]
The necessary and sufficient condition of an exact decomposition by the SSWPT is $v_f(a,b)=N$ in $Z_{f_1}=\{(a,b):W_{f_1}(a,b)\neq 0\}$ and $v_f(a,b)=N\lambda $ in $Z_{f_2}=\{(a,b):W_{f_2}(a,b)\neq 0\}$. This is equivalent to say $Z_{f_1}$ and $Z_{f_2}$ are disjoint. Since 
\[
W_{f_1}(a,b)=|a|^{-s/2}e^{2\pi iNb}\widehat{w}(|a|^{-s}(a-N)),
\]
\[
W_{f_2}(a,b)=|a|^{-s/2}e^{2\pi iN\lambda b}\widehat{w}(|a|^{-s}(a-N\lambda )),
\]
and the support of $\widehat{w}(\xi)$ is $(-d,d)$, the condition for $Z_{f_1}$ and $Z_{f_2}$ being disjoint is that the supports of the wave packets $\widehat{w_{N\lambda  b}}$ and $\widehat{w_{Nb}}$ are disjoint for all $b$, i.e., $\widehat{w}(|N\lambda |^{-s}(x-N\lambda ))$ and $\widehat{w}(|N|^{-s}(x-N))$ have non-overlapping supports. So, the critical number $\lambda_0$ of the SSWPT with a geometric scaling parameter $s$ is the solution of the following equations:
\begin{equation}
\left\{
\begin{array}{rl}
N-a_1&=d|a_1|^s, \nonumber\\ 
a_2-N\lambda _0&=d|a_2|^s,\nonumber\\
a_1&=a_2,\nonumber
\end{array}
\right.
\end{equation}
where $N$ and $d$ are known. When $s=1$, the solution is $\lambda _0=\frac{1-d}{1+d}$, which accords with the result in \cite{Syn_ans}. Let $a_1=a_2=a$, then we have $\lambda _0=\frac{2a-N}{N}$. Notice that $\lambda _0$ is increasing, when $s$ is decreasing (see Figure \ref{fig:resolution}). Therefore, a smaller $s$ benefits a higher single scale resolution, especially for high frequency signals.

\begin{figure}[ht!]
  \begin{center}
    \begin{tabular}{ccc}
      \includegraphics[height=1.6in]{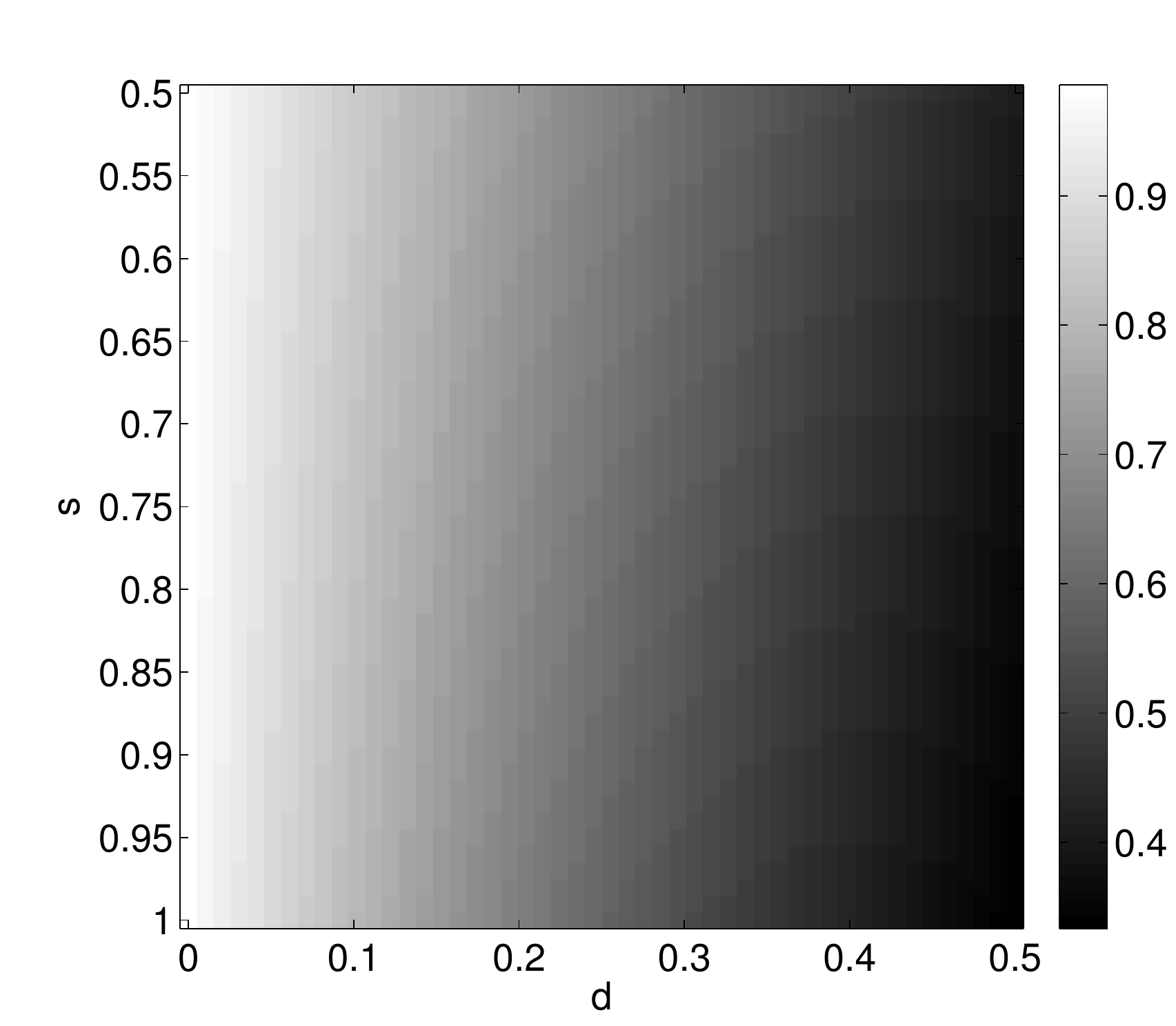} & \includegraphics[height=1.6in]{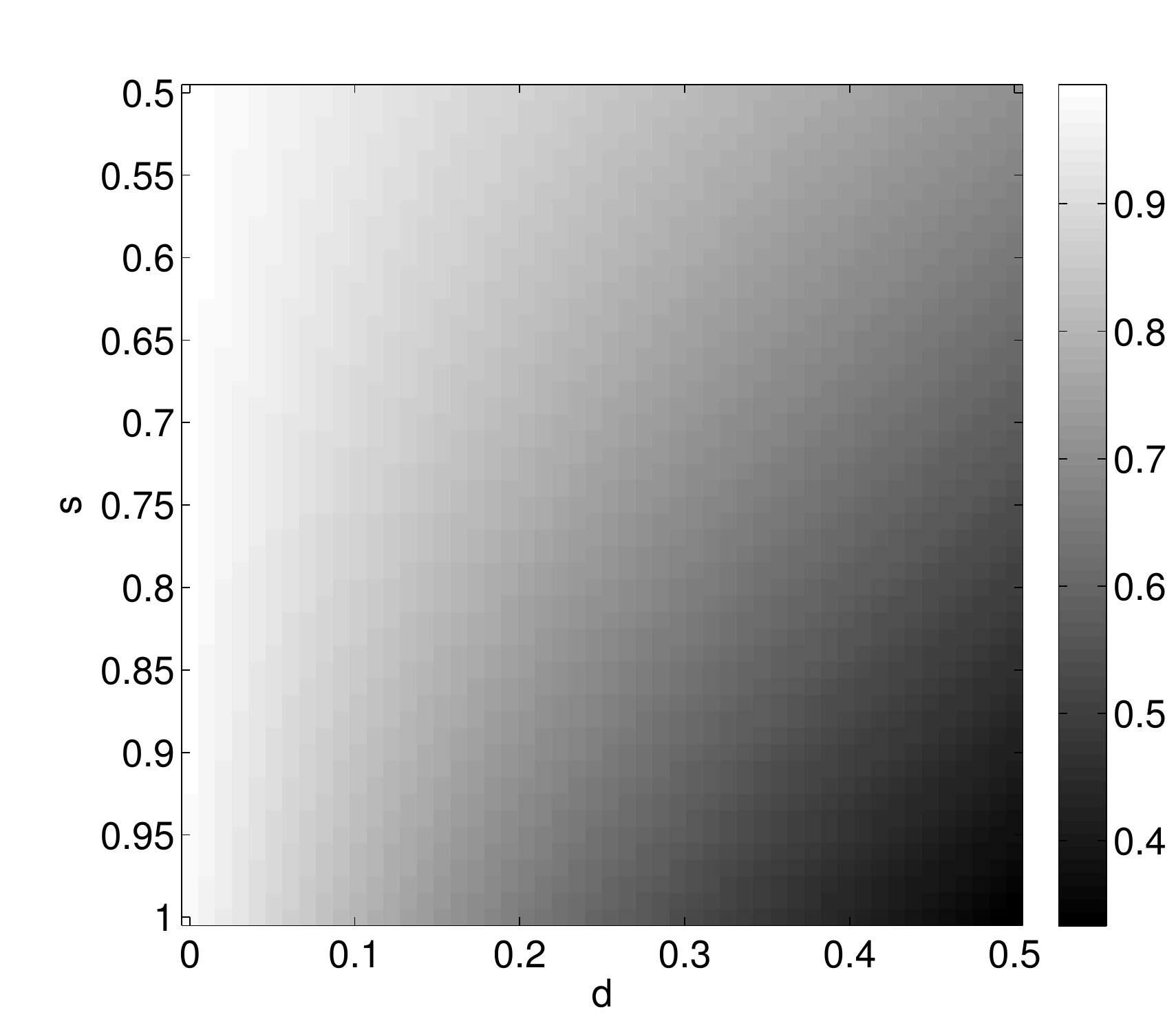} & \includegraphics[height=1.6in]{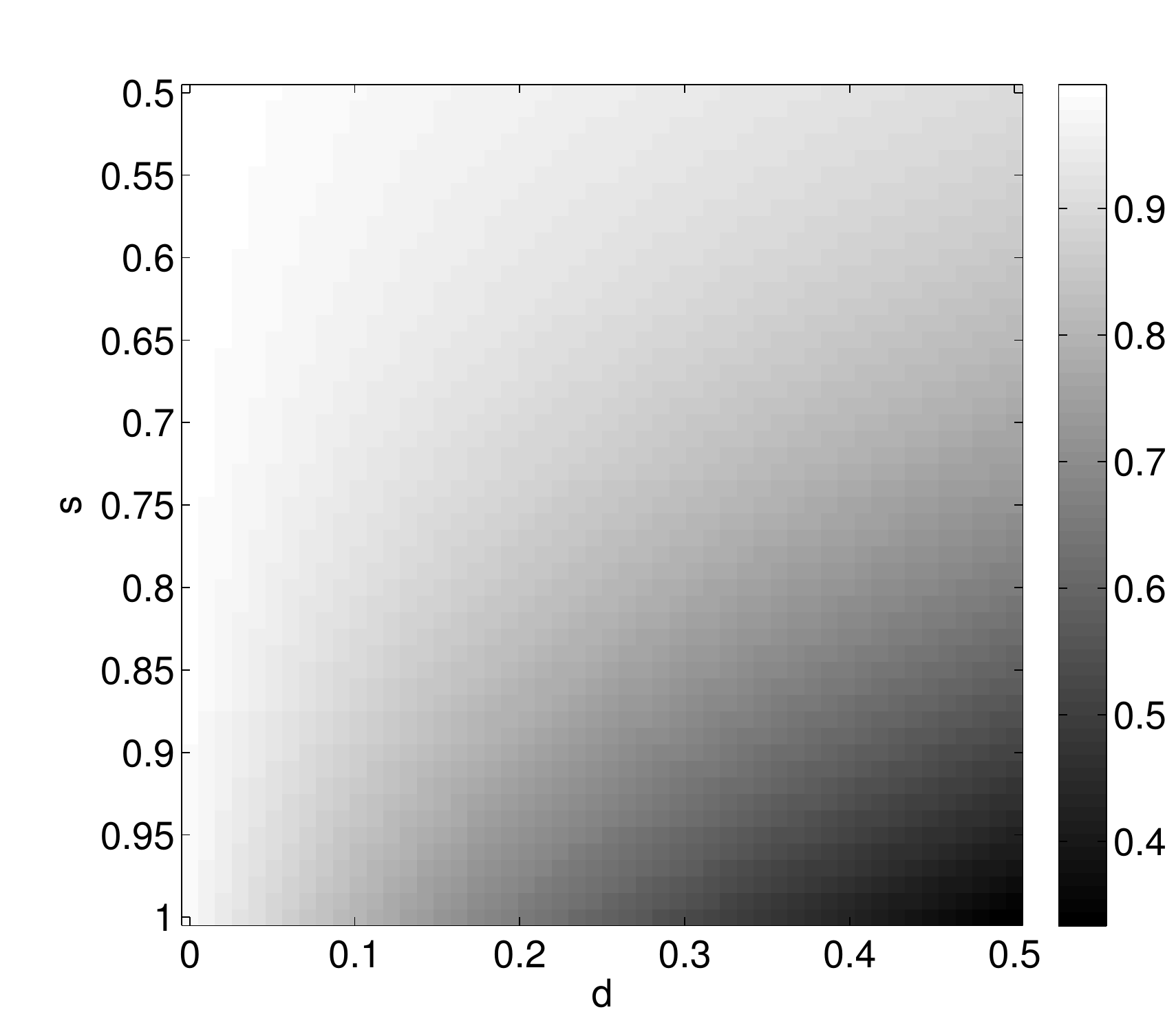}\\
      \includegraphics[height=1.6in]{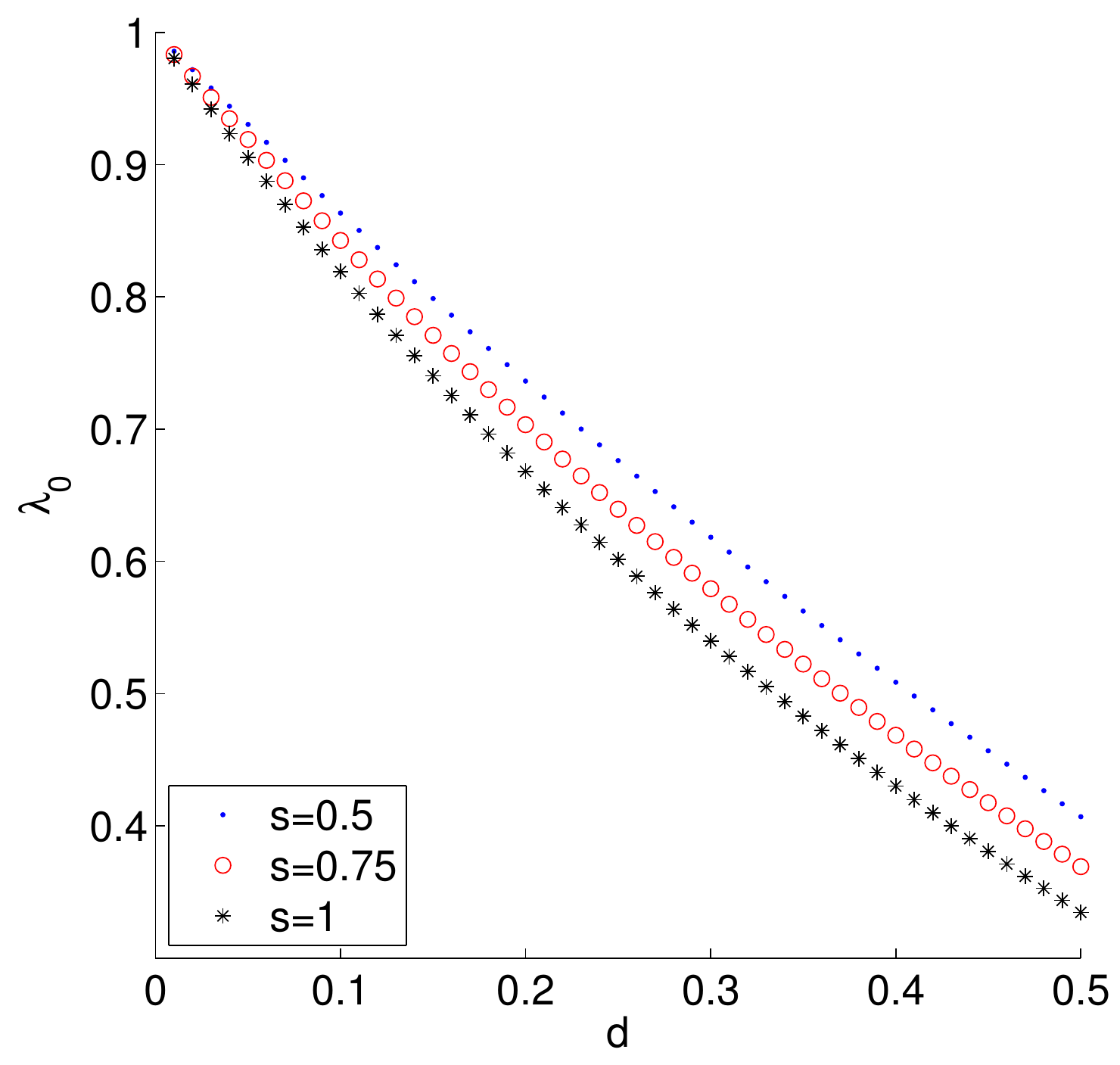} & \includegraphics[height=1.6in]{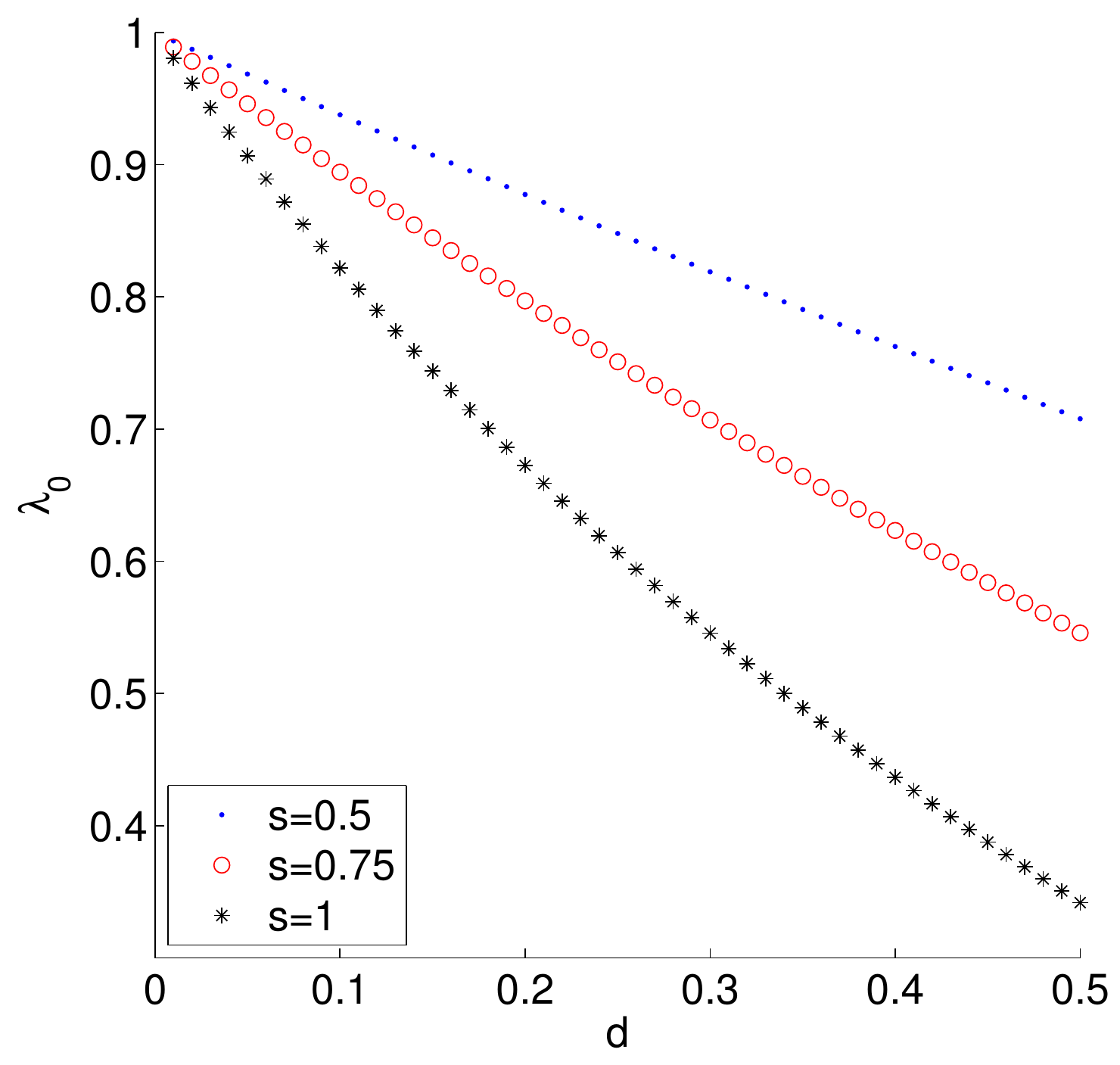} & \includegraphics[height=1.6in]{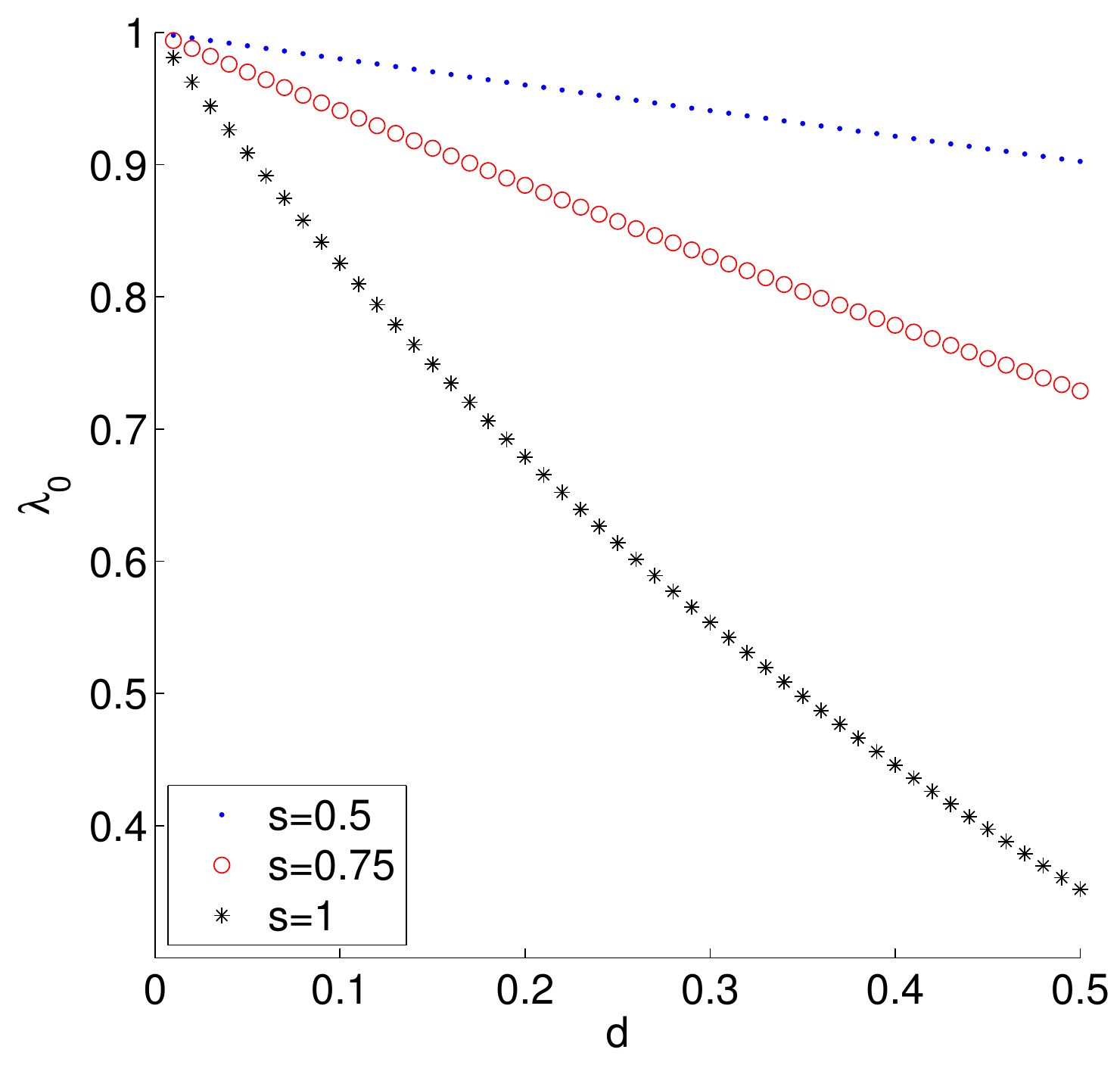}\\
    \end{tabular}
  \end{center}
  \caption{Top row: The critical number $\lambda_0$ when $s\in[0.5,1]$ and $d\in[0,0.5]$. Bottom row: The critical number $\lambda_0$ for fixed $s=0.5$, $0.75$, and $1$. Left column: $N=2$. Middle column: $N=10$. Right column: $N=100$. The critical number $\lambda_0$ for $s\approx 0.5$ is much larger than that for $s=1$ when the frequency level $N$ is large and the difference becomes more significant when the frequency is higher.}
\label{fig:resolution}
\end{figure}

For the mode decomposition problems of the form \eqref{P1}, the single scale resolution is enough to quantify the resolution of a certain synchrosqueezed transform. However, for the general mode decomposition problems of the form \eqref{P2}, each general mode $\alpha_k(t)s_k(2\pi N_k\phi_k(t))$ would result in multiple instantaneous frequencies $nN_k\phi_k'(t)$, i.e., a superposition of infinitely many terms $\widehat{s_k}(n)\alpha_k(t)e^{2\pi i nN_k\phi_k(t)}$. It is important to know how many multiple instantaneous frequencies can be identified by the synchrosqueezed transform.

\begin{definition}
The multiscale resolution at a level $N$ of a certain synchrosqueezed transform is $\frac{1}{Nk}-\frac{1}{N(k+1)}$, where $k$ is the largest number of the components such that the synchrosqueezed transform can distinguish all the components in 
\begin{equation*}
f(t)=\sum_{n=1}^{k} e^{2\pi inNt}.
\end{equation*}
\end{definition}

By definition, a synchrosqueezed transform with a smaller multiscale resolution can distinguish more Fourier expansion terms $\widehat{s_k}(n)\alpha_k(t)e^{2\pi i nN_k\phi_k(t)}$, so that one would obtain a better recovery of $\alpha_k(t)s_k(2\pi N_k\phi_k(t))$ by combining these recovered Fourier expansion terms.

For a superposition 
\begin{equation}
\label{multi}
f(t)=\sum_{n=1}^{K} e^{2\pi inNt}
\end{equation}
with $K$ arbitrarily large, the wave packet transform is
\[
W_f(a,b)=\sum_{n=1}^{K}|a|^{-s/2}e^{2\pi inNb}\widehat{w}\left(|a|^{-2}(a-nN)\right)
\]
and the instantaneous frequency information function is 
\[
v_f(a,b)=\frac{N\sum_{n=1}^{K}ne^{2\pi inNb}\widehat{w}\left(|a|^{-s}(a-nN)\right)}{\sum_{n=1}^{K}e^{2\pi inNb}\widehat{w}\left(|a|^{-s}(a-nN)\right)}.
\]
Each term $\widehat{w}\left(|a|^{-s}(a-nN)\right)$ is supported in $Z_{n}=\{a:|a-nN|<d|a|^{s}\}$ centered at $a=nN$. Similar to previous discussions, an exact recovery of the $n$th term is equivalent to $Z_{n-1}\bigcap Z_{n}=\emptyset$ and $Z_{n}\bigcap Z_{n+1}=\emptyset$. Since the size of the interval $Z_n$ is monotonously increasing as $n$ increases and the space between their centers is fixed, one should identify $n_0$, the greatest n such that $Z_n\bigcap Z_{n+1}= \emptyset$, i.e., $n_0=\max\{n\in \Z^+:nN<a_0\}$, where $a_0=\left(\frac{N}{2d}\right)^{1/s}$ is a solution of the following equations,
\begin{equation}
\left\{
\begin{array}{rl}
(n+1)N-a&=d|a|^s, \nonumber\\ 
a-nN&=d|a|^s.\nonumber
\end{array}
\right.
\end{equation}
Hence, $n_0=\lfloor  \frac{N^{\frac{1}{s}-1}}{(2d)^{1/s}} \rfloor$. Therefore, the multiscale resolution at the level $N$ of the SSWPT with a scaling parameter $s$ and a mother wave packet supported in $(-d,d)$ is 
\[
\frac{1}{(n_0-1)N}-\frac{1}{n_0N}=\frac{1}{N(\lfloor  \frac{N^{\frac{1}{s}-1}}{(2d)^{1/s}} \rfloor-1) \lfloor  \frac{N^{\frac{1}{s}-1}}{(2d)^{1/s}} \rfloor}  \approx O(N^{1-2/s}).
\] 

Notice that the synchrosqueezed wavelet transform (i.e., $s=1$) can only distinguish $O(1)$ terms in \eqref{multi}. This limits its application to general mode decompositions of the form \eqref{P2}. However, the SSWPT is able to identify $O(N^{\frac{1}{s}-1})$ terms exactly in \eqref{multi}. This motivates its application to the general mode decomposition problems of the form \eqref{P2}.  Concrete examples will be presented to support this argument in Section \ref{sec:results}.

\section{Theory for the diffeomorphism based spectral analysis (DSA)}
\label{sec:diff}
The analysis of the DSA method essentially consists of two main theorems. Theorem \ref{thm:main2} proves that the 1D SSWPT is able to provide accurate input instantaneous frequency estimates if the weak well-separation condition defined below holds. Then Theorem \ref{thm:main3} proves that Step $4$, the key idea of the DSA method, can provide precise spectral analysis for general shape functions, if their corresponding phase functions are well-different and steep enough. We omit the proof of the other steps in the DSA method to save space.

\begin{definition}
 A function $f(t)$ is a weak well-separated general superposition of type
  $(M,N,K,s)$ if
  \[
  f(t)=\sum_{k=1}^K f_k(t)
  \] 
  where each $f_k(t)=\alpha_k(t)s_k(2\pi N_k \phi_k(t))$ is a GIMT of type $(M,N_k)$ such that $N_k\geq N$ and the phase functions satisfy the following weak well-separation conditions. 
\begin{enumerate}
\item Suppose 
\[
Z_{nk}=\left\{(a,b):|a-nN_k\phi_k'(b)|\leq d|a|^{s}\right\}.
\]
For each $k$, there exists $n_k$ such that $\widehat{s_k}(n_k)\neq 0$ and $Z_{n_kk}\cap Z_{nj}=\emptyset$ for all pairs $(n,j)\neq (n_k,k)$ and $\widehat{s_j}(n)\neq 0$.
\item $\exists K_0<\infty$ such that $\forall a\in\R$ and $\forall b\in \R$ there exists at most $K_0$ pairs of $(n,k)$ such that $(a,b)\in Z_{nk}$.
\end{enumerate}
We denote by $wGF(M,N,K_0,K,s)$ the set of all such functions. Note that $wGF(M,N,1,K,s)=GF(M,N,K,s)$.
\end{definition}
In what follows, when we write $O(\cdot)$, $\lesssim$, or $\gtrsim$, the implicit constants may depend on $M$, $K$ and $K_0$.
\begin{theorem}
  \label{thm:main2}
  For a function $f(t)$ and $\eps>0$, we define
  \[
  R_{\eps} = \{(a,b): |W_f(a,b)|\geq |a|^{-s/2}\sqrt \eps\}
  \]
  and 
  \[
  Z_{n,k} = \{(a,b): |a-nN_k\phi_k'(b)|\leq d|a|^s \}
  \]
  for $1\le k\le K$ and $|n|\geq 1$. For fixed $M$, $K_0$, $K$ and $\forall \eps>0$, there
  exists a constant $N_0(M,K_0,K,s,\eps)>0$ such that
  $\forall N>N_0$ and $f(t)\in wGF(M,N,K_0,K,s)$ the following statements
  hold.
  \begin{enumerate}[(i)]
  \item For each $j$, there exists $n_j$ such that $\widehat{s_j}(n_j)\neq 0$ and $Z_{n_jj}\cap Z_{nk}=\emptyset$ for all pairs $(n,k)\neq (n_j,j)$ and $\widehat{s_k}(n)\neq 0$;
  \item For any $(a,b) \in R_{\eps} \cap Z_{n_j,j}$, 
    \[
    \frac{|v_f(a,b)-n_jN_j\phi_j'(b)|}{ |n_jN_j \phi_j'(b)|}\lesssim\sqrt \eps.
    \]
\item For each $j$, let 
\[
l_{n_j}(b)=\min \left\{a:(a,b)\in R_\epsilon\cap Z_{n_jj}\right\},\quad u_{n_j}(b)=\max\left\{a:(a,b)\in R_\epsilon\cup Z_{n_jj}\right\}.
\]
Suppose $v_f(a,b)\neq \infty$. If $a\leq l_{n_j}(b)$, then $v_f(a,b)\leq l_{n_j}(b)(1+O(\sqrt{\epsilon}))$. If $a\geq u_{n_j}(b)$, then $v_f(a,b)\geq u_{n_j}(b)(1-O(\sqrt{\epsilon}))$.
  \end{enumerate}
\end{theorem}
\begin{proof}
The weak well-separation condition implies $(i)$. $(ii)$ is true by the same argument of Theorem \ref{thm:main} $(ii)$. We only need to prove $(iii)$. Let us recall that \[\Omega_a=\{(k,n):a\in[\frac{nN_k}{2M},2MnN_k]\}.\] 
By Lemma \ref{lemma1}
\begin{eqnarray}
W_f(a,b)&=&|a|^{-s/2}\left(\sum_{(k,n)\in\Omega_a}\widehat{s_k}(n)\alpha_k(b)e^{2\pi inN_k\phi_k(b)}\widehat{w}\left(\left(a-nN_k\phi_k'(b)\right)|a|^{-s}\right)+O(\eps)\right),\nonumber\nonumber
\end{eqnarray}
as the other terms drop out.
Similarly, by Lemma \ref{lemma2}
\begin{eqnarray}
& &\partial_b W_f(a,b)\nonumber\\
&=&|a|^{-s/2}\left(\sum_{(k,n)\in\Omega_a }2\pi inN_k\widehat{s_k}(n)\alpha_k(b)\phi_k'(b)e^{2\pi inN_k\phi_k(b)}\widehat{w}\left(\left(a-nN_k\phi_k'(b)\right)|a|^{-s}\right)+|a|O(\eps)\right).\nonumber
\end{eqnarray}
Let $g_{n,k}$ denote the term $\widehat{s_k}(n)\alpha_k(b)e^{2\pi inN_k\phi_k(b)}\widehat{w}\left(\left(a-nN_k\phi_k'(b)\right)|a|^{-s}\right)$, then 
\begin{eqnarray}
\frac{v_f(a,b)-a}{a}&=&\frac{\partial_bW_f(a,b)-2\pi iaW_f(a,b)}{2\pi iaW_f(a,b)}\nonumber\\
&=&\frac{\sum_{(k,n)\in\Omega_a }(nN_k\phi_k'(b)-a)g_{n,k}+|a|O(\eps)}{a|a|^{s/2}W_f(a,b)},\nonumber
\end{eqnarray}
since $a\in[\frac{nN_k}{2M},2MnN_k]$ for $(k,n)\in\Omega_a$. Because $f(t)\in wGF(M,N,K_0,K,s)$, the number of $g_{n,k}$ not vanishing is at most $K_0$. Because $|W_f(a,b)|\geq |a|^{-s/2}\sqrt{\eps}$ for $(a,b)\in R_\eps$, $\left|a-nN_k\phi_k'(b)\right|\leq|a|^{s}d$ for $g_{n,k}$ not vanishing, and $\left|g_{n,k}\right|\lesssim 1$, then
\[
\left| \frac{v_f(a,b)-a}{a} \right|\lesssim \frac{K_0|a|^sd+|a|O(\epsilon)}{|a|\sqrt{\epsilon}}\lesssim \sqrt{\epsilon},
\]
if $|a|\gtrsim N\gtrsim \epsilon^{\frac{-1}{1-s}}$.
Therefore, if $a\leq l_{n_j}(b)$, then
\[
v_f(a,b)\leq a(1+O(\sqrt{\epsilon}))\leq l_{n_j}(b)(1+O(\sqrt{\epsilon}))
\]
for $N$ sufficiently large.
If $a\geq u_{n_j}(b)$, then 
\[
v_f(a,b)\geq a(1-O(\sqrt{\epsilon}))\geq u_{n_j}(b)(1-O(\sqrt{\epsilon}))
\]
for $N$ large enough.
\end{proof}
Theorem \ref{thm:main2} $(ii)$ shows that the instantaneous frequency information $v_f(a,b)$ can estimate instantaneous frequencies $\{n_jN_j\phi_j'(t)\}_{j=1}^K$ of some well-separated Fourier expansion terms accurately so that the energy of $\widehat{s_j}(n_j)\alpha_j(t)e^{2\pi in_jN_j\phi_j(t)}$ is squeezed to sharpened areas around $n_jN_j\phi_j'(t)$. Theorem \ref{thm:main2} $(iii)$ implies that the synchrosqueezed energy distribution $T_f(a,b)$ has well-separated and sharp supports around $\{n_jN_j\phi_j'(t)\}_{j=1}^K$, each of which only corresponds to $\widehat{s_j}(n_j)\alpha_j(t)e^{2\pi in_jN_j\phi_j(t)}$. This guarantees the accurate estimate of $n_jN_j\phi_j'(t)$ and the precise extraction of $\widehat{s_j}(n_j)\alpha_j(t)e^{2\pi in_jN_j\phi_j(t)}$.

Next, Theorem \ref{thm:main3} below shows that the DSA method with exact estimates of the instantaneous frequencies is able to provide accurate spectral analysis of the general shape functions, if the phase functions are well-different and steep sufficiently.

Since $f(t)$ is defined in $\R$ with non-vanishing amplitudes, we consider the following short-time Fourier transform with real-valued, non-negative and smooth window function $w_1(t)$ compactly supported in $(-1,1)$ such that $|\widehat{w_1}|$ has a sheer peak around the origin and rapidly decays elsewhere.
\begin{definition}
Given the window function $w_1(t)$ and a parameter $T>1$, the short-time Fourier transform of a function $f(t)$ with a parameter $T$ is a function
\[
\F_T(f)(a,b)=\int_{\R}f(t)w_T(t-b)e^{-2\pi iat}dt
\]
for $a,b\in\R$, where $w_T(t)=w_1(t/T)$ and $\F_T$ denote the short-time Fourier transform operator with the parameter $T$.
\end{definition}
\begin{definition}
  \label{def:WD} 
For $M>0$ and $K>0$, the phase functions $\{\phi_k(t)\}_{1\leq k\leq K}$ are well-different of type $(M,K)$ at $b\in \R$, if they satisfy the following conditions.
\begin{enumerate}
\item For any $T>0$, the number of extrema of $\phi_k\circ\phi_j^{-1}(t)$ in $(b-T,b+T)$ is at most $TM$ for $k\neq j$.
\item For any $T>0$ there exists $\eta_0>0$, $\eta_1>0$ and $N_0(M,K,T,b)$ such that $\forall a\in (\frac{1}{2M^2},2M^2)$ and $\forall N>N_0(M,K,T,b)$
\begin{align*}
\lambda^*\left( \left\{   t:\left|\partial_t\left(\phi_k(\phi_j^{-1}(t))\right)-a\right|\leq \frac{1}{N^{1-\eta_0}}\right\}\cap \left\{ t:b-T\leq t \leq b+T \right\}\right)\lesssim O(\frac{1}{N^{\eta_1}})
\end{align*}
for $k\neq j$, where $\lambda^*(\cdot)$ denotes the Lebesgue measure and $\lesssim$ means the implicit constant may depend on $M$, $K$, $T$ and $b$.
\end{enumerate}
\end{definition}
The definition of well-different phase functions is crucial to general mode decompositions. The difference of phase functions is the key feature for grouping the Fourier expansion terms of the general modes. If two phase functions are similar, their corresponding general modes would have similar evolution patterns. It is reasonable to combine them as one general mode. On the other hand, the well-difference of phase functions guarantees that the key idea of the DSA method can provide accurate spectral information of general shape functions, as proved in the following theorem.
\begin{theorem}
  \label{thm:main3}
Suppose $f(t)=\sum_{k=1}^Kf_k(t)$, where $f_k(t)=\alpha_k(t)s_k(2\pi N_k\phi_k(t))$ is a GIMT of type $(M,N_k)$ with $N_k\geq N$ and the phase functions $\{\phi_k(t)\}_{1\leq k\leq K}$ are well-different of type $(M,K)$ at $b$. Let $s_0=\max\limits_{(k,n)}\left|\widehat{s_k}(n)\right|$. Define 
\[
h_k(t)=\frac{f\circ \phi_k^{-1}(t)}{\alpha_k\circ \phi_k^{-1}(t)}
\]
for $1\leq k\leq K$. For fixed $M$, $K$, $b$, $s_0$ and $\delta>0$, $\exists T_0(M,K,s_0,\delta,b)$, $\forall T>T_0$, $\exists N_0(M,K,s_0,T,b)>0$ such that $\forall N>N_0$ the solution of the following optimization problem
\[
(a_0,k_0)=\underset{(a,k)}|{\arg\max} \F_T(h_k)(a,b)|
\]
satisfies $|a_0-nN_{k_0}|<\delta$ for some $n$ such that $\widehat{s_{k_0}}(n)\neq 0$.
\end{theorem}
In what follows, when we write $O(\cdot)$, $\lesssim$, or $\gtrsim$, the implicit constants may depend on $M$, $K$, $T$ and $b$.
\begin{proof}
Notice that
\[
h_k(t)=\frac{f\circ \phi_k^{-1}(t)}{\alpha_k\circ \phi_k^{-1}(t)}=\sum_{n=-\infty}^{\infty}\widehat{s_k}(n)e^{2\pi i nN_kt}+ \sum_{j\neq k} \sum_{n=-\infty}^{\infty} \widehat{s_j}(n) \frac{\alpha_j\circ \phi_k^{-1}(t)}{\alpha_k\circ \phi_k^{-1}(t)} e^{2\pi i n N_j \phi_j\circ \phi_k^{-1}(t)},
\]
then
\begin{eqnarray}
\F_T(h_k)(a,b)&=&\sum_{n=-\infty}^{\infty}\widehat{s_k}(n)\int_{\R}w_T(t-b)e^{2\pi i (nN_k-a)t}dt\nonumber\\
& &+ \sum_{j\neq k} \sum_{n=-\infty}^{\infty}\widehat{s_j}(n)  \int_{\R}\frac{\alpha_j\circ \phi_k^{-1}(t)}{\alpha_k\circ \phi_k^{-1}(t)} w_T(t-b)e^{2\pi i (n N_j \phi_j\circ \phi_k^{-1}(t)-at)}dt\nonumber
\end{eqnarray}
by the uniform convergence of the Fourier series of $s_k(t)$. 
The first part of $\F_T(h_k)(a,b)$ is
\begin{eqnarray}
I_1(a,k)&=&\sum_{n=-\infty}^{\infty}\widehat{s_k}(n)\int_{\R}w_T(t-b)e^{2\pi i (nN_k-a)t}dt\nonumber\\
&=&\sum_{n=-\infty}^{\infty}T\widehat{s_k}(n)e^{2\pi ib(nN_k-a)}\int_{\R}w_1(x)e^{2\pi i T(nN_k-a)x}dx\nonumber\\
&=&\sum_{n=-\infty}^{\infty}T\widehat{s_k}(n)e^{2\pi ib(nN_k-a)}\widehat{w_1}\left(T(a-nN_k)\right).\nonumber
\end{eqnarray}
Hence, $\exists T_0(M,K,s_0,\delta,b)$ such that, if $T>T_0$, then $\left|I_1(a,k)\right|$ has well-separated sheer energy peaks at $a=nN_k$ of order $T\left|\widehat{s_k(n)}\right|$ and  $\left|I_1(a,k)\right|<\frac{Ts_0}{3}$ if $|a-nN_k|\geq \delta$ for all $n$. The estimate of the second part 
\[
I_2(a,k)=\sum_{j\neq k} \sum_{n=-\infty}^{\infty}\widehat{s_j}(n)  \int_{\R}\frac{\alpha_j\circ \phi_k^{-1}(t)}{\alpha_k\circ \phi_k^{-1}(t)} w_T(t-b)e^{2\pi i (n N_j \phi_j\circ \phi_k^{-1}(t)-at)}dt
\]
relies on the estimate of each term 
\[
I_{jn}=\widehat{s_j}(n)  \int_{\R}\frac{\alpha_j\circ \phi_k^{-1}(t)}{\alpha_k\circ \phi_k^{-1}(t)} w_T(t-b)e^{2\pi i (n N_j \phi_j\circ \phi_k^{-1}(t)-at)}dt.
\]
Notice that $\frac{\alpha_j\circ \phi_k^{-1}(t)}{\alpha_k\circ \phi_k^{-1}(t)} w_T(t-b)$ and $2\pi(n N_j \phi_j\circ \phi_k^{-1}(t)-at)$ are real smooth functions and $w_T(t-b)$ has a compact support in $(b-T,b+T)$. If $\partial_t(n N_j \phi_j\circ \phi_k^{-1}(t)-at)\neq 0$ in $(b-T,b+T)$, a similar argument of the integration by parts in Lemma \ref{lemma1} shows that 
\[
|I_{jn}|\lesssim |\widehat{s_j}(n)|\frac{1}{\left|n N_j\partial_t( \phi_j\circ \phi_k^{-1})(t)-a\right|}.
\]
Therefore, the order of $|I_{jn}|$ is determined by points $t$ such that $\left|n N_j\partial_t( \phi_j\circ \phi_k^{-1})(t)-a\right|$ is vanishing or relatively small. 

If $a\notin (\frac{nN_j}{2M^2},2nN_jM^2)$, then by the fact that $\partial_t( \phi_j\circ \phi_k^{-1})(t)\in [\frac{1}{M^2},M^2]$, we have $\left|n N_j\partial_t( \phi_j\circ \phi_k^{-1})(t)-a\right|\gtrsim nN_j$, which implies 
\begin{eqnarray}
\label{eqn:I1}
|I_{jn}|\lesssim \frac{|\widehat{s_j}(n)|}{nN_j}\lesssim \frac{1}{N}.
\end{eqnarray} 

If $a\in (\frac{nN_j}{2M^2},2nN_jM^2)$, then $\frac{a}{nN_j}\in(\frac{1}{2M^2},2M^2)$. Let 
\[
A=\left\{   t:\left|\partial_t\left(\phi_j\circ \phi_k^{-1}(t)\right)-\frac{a}{nN_j}\right|\leq \frac{1}{(nN_j)^{1-\eta_0}}\right\}\cap \left\{ t:b-T\leq t \leq b+T \right\}.
\]
Because the phase functions are well-different of type $(M,K)$ at $b$, for fixed $T$ there exists $\eta_0>0$, $\eta_1>0$ and $N_1(M,K,T,b)$ such that for $\frac{a}{nN_j}\in (\frac{1}{2M^2},2M^2)$ and $nN_j>N_1(M,K,T,b)$, we have
$\lambda^*(A)\lesssim O(\frac{1}{(nN_j)^{\eta_1}})$. This gives 
\[
\left|\widehat{s_j}(n)  \int_{A}\frac{\alpha_j\circ \phi_k^{-1}(t)}{\alpha_k\circ \phi_k^{-1}(t)} w_T(t-b)e^{2\pi i (n N_j \phi_j\circ \phi_k^{-1}(t)-at)}dt\right|\lesssim O(\frac{\left|\widehat{s_j}(n)\right|}{(nN_j)^{\eta_1}}).
\]
By the definition of well-difference of type $(M,K)$, $\left(\R\setminus A\right)\cap (b-T,b+T)$ is a union of at most $O(TM)$ intervals. Hence, similar to the method of stationary phase, we have
\[
\left|\widehat{s_j}(n)  \int_{\R\setminus A}\frac{\alpha_j\circ \phi_k^{-1}(t)}{\alpha_k\circ \phi_k^{-1}(t)} w_T(t-b)e^{2\pi i (n N_j \phi_j\circ \phi_k^{-1}(t)-at)}dt\right|\lesssim O(\frac{\left|\widehat{s_j}(n)\right|}{(nN_j)^{\eta_0}}).
\]
In sum,\begin{eqnarray*}
\left|I_{jn}\right|&\leq&\left|\widehat{s_j}(n)  \int_{\R\setminus A}\frac{\alpha_j\circ \phi_k^{-1}(t)}{\alpha_k\circ \phi_k^{-1}(t)} w_T(t-b)e^{2\pi i (n N_j \phi_j\circ \phi_k^{-1}(t)-at)}dt\right|\\
& &+\left|\widehat{s_j}(n)  \int_{A}\frac{\alpha_j\circ \phi_k^{-1}(t)}{\alpha_k\circ \phi_k^{-1}(t)} w_T(t-b)e^{2\pi i (n N_j \phi_j\circ \phi_k^{-1}(t)-at)}dt\right|\\
&\lesssim& O(\frac{\left|\widehat{s_j}(n)\right|}{(nN_j)^{\eta_1}})+O(\frac{\left|\widehat{s_j}(n)\right|}{(nN_j)^{\eta_0}}).
\end{eqnarray*}
Recall that $N_k\geq N$ and $\sum \limits_{n=-\infty}^{\infty}|\widehat{s_k}(n)|\leq M$ for $1\leq k\leq K$. So, if $N>N_1(M,K,T,b)$
\begin{eqnarray}
\label{eqn:I2}
\left|I_2(a,k)\right|\lesssim \sum_{j\neq k} \sum_{n=-\infty}^{\infty}\left(O(\frac{\left|\widehat{s_j}(n)\right|}{(nN_j)^{\eta_1}})+O(\frac{\left|\widehat{s_j}(n)\right|}{(nN_j)^{\eta_0}})\right)\lesssim O(\frac{(K-1)M}{N^{\eta}})\lesssim O(\frac{1}{N^\eta}),
\end{eqnarray}
where $\eta=\min \{\eta_0,\eta_1\}$. 

By \eqref{eqn:I1} and \eqref{eqn:I2}, $\exists N_0=\max\left\{N_1(M,K,T,b),\left(\frac{3}{Ts_0}\right)^{1/\eta},\frac{3}{Ts_0}\right\}$ such that $\forall N>N_0$, we have $\left|I_2(a,k)\right|<\frac{Ts_0}{3}$.

Let $\Xi_k$ be the index set $\{n:\widehat{s_k}(n)\neq 0\}$ and $(\tilde{n},\tilde{k})=\underset{(n,k)}{\arg\max}  \left|\widehat{s_k}(n)\right|$. Now suppose $N>N_0$. Let $\left|\F_T(h_k)(a,b)\right|$ take the maximum value at the pair $(a_0,k_0)$. If there is no $n\in\Xi_{k_0}$ such that $\left| a_0-nN_{k_0}\right|<\delta$, then 
\[
\left|\F_T(h_{k_0})(a_0,b)\right|\leq \left|I_1(a_0,k_0)\right|+\left|I_2(a_0,k_0)\right|<\frac{2Ts_0}{3}.
\]
However, for the pair $(\tilde{n},\tilde{k})$, we have
\[
\left|\F_T(h_{\tilde{k}})(\tilde{n},b)\right|\geq \left|I_1(\tilde{n},\tilde{k})\right|-\left|I_2(\tilde{n},\tilde{k})\right|>Ts_0-\frac{Ts_0}{3}>\frac{2Ts_0}{3}.
\]
This conflicts with the fact that $\left|\F_T(h_k)(a,b)\right|$ takes the maximum value less than $\frac{2Ts_0}{3}$ at the pair $(a_0,k_0)$. Hence, there exists $n\in\Xi_{k_0}$ satisfying that $\left|a_0-nN_{k_0}\right|<\delta$. This completes the proof.
\end{proof}

In practice, the signal $f(t)$ is defined in a bounded interval, e.g., $[0,1]$ without loss of generality. Applying the Fourier transform on $f(t)$ in $[0,1]$ is equivalent to applying the short-time Fourier transform on $f(t)$ with a rectangle window function centered at $t=\frac{1}{2}$. In this sense, Theorem \ref{thm:main3} implies that the DSA method can accurately decompose $f(t)$ into GIMTs $\{\alpha_k(t)s_k(2\pi N_k\phi_k(t))\}_{k=1}^K$ and analyzes the spectra of general shape functions $\{\alpha_k(t)\}_{k=1}^K$ by extracting the Fourier expansion terms $\widehat{s_k}(n)\alpha_k(t)e^{2\pi inN_k\phi_k(t)}$ one by one from the one with highest energy.

\section{Numerical examples}
\label{sec:results}
In this section, some numerical examples of synthetic and real data are provided to demonstrate the proposed properties of 1D SSWPT and the efficiency of the GMDWP method and the DSA method in fruitful applications. In all of these examples, the 1D SSWPT is implemented using a fast algorithm similar to the one in \cite{Demanet2007,SSWPT} and the complexity is $O(L\log(L))$, where $L$ is the number of sample points of a given signal. The mother wave packet $w(t)$ is constructed using the same method in \cite{Demanet2007} with a support parameter $d=1$. The threshold parameter in the main theorems is $\epsilon=10^{-6}$ and the scaling parameter $s$ is equal to $2/3$. For the purpose of convenience, the synthetic data is defined in $[0,1]$ and the number of samples is  between $2^{13}$ and $2^{15}$.

\subsection{The comparison of multiscale resolutions}
Let us start by repeating the comparison of the resolutions of the SSWPT and the SSWT, since it is a fundamental issue in general mode decomposition problems. As we shall see in the following two examples, the SSWT would mix up high frequency terms and would provide misleading high frequency information. However, the SSWPT can relieve much of this trouble. In the following two examples, the mother wavelet and the mother wave packet have the same size of supports $d=1$ in the Fourier domain.

\textbf{Example $2$:} Let us consider
\[
f(t)=e^{200\pi i(x+5x^2)}+\sum_{n=1}^{20}e^{2\pi i nN(x+0.005*\sin(2\pi x))}
\]
where $N=100$. Figure \ref{fig:comp} left shows the real instantaneous frequencies of all the oscillatory terms in $f(t)$. The SSWPT of $f(t)$ shown in Figure \ref{fig:comp} middle agrees with Theorem \ref{thm:main} and \ref{thm:main2} that the essential support of the synchrosqueezed energy distribution $T_f(v,b)$ is concentrating around isolated instantaneous frequencies, even with crossover frequencies on the scene. Recall that the SSWPT is able to identify $O(N^{1/s-1})=O(10)$ components when $s=2/3$. The number of clearly identified components is $20$ according with the multiscale resolution analysis of the SSWPT.  However, the SSWT can only identify $O(1)$ components as shown in Subsection \ref{sub:RA}. This explains why the SSWT of $f(t)$ is misleading as shown in Figure \ref{fig:comp} right. The SSWT mixes up the high frequency terms and cannot reflect the true instantaneous information of signals. Nevertheless, the high frequency information is crucial to precise reconstructions of general shape functions.

\begin{figure}
  \begin{center}
    \begin{tabular}{ccc}
      \includegraphics[height=1.6in]{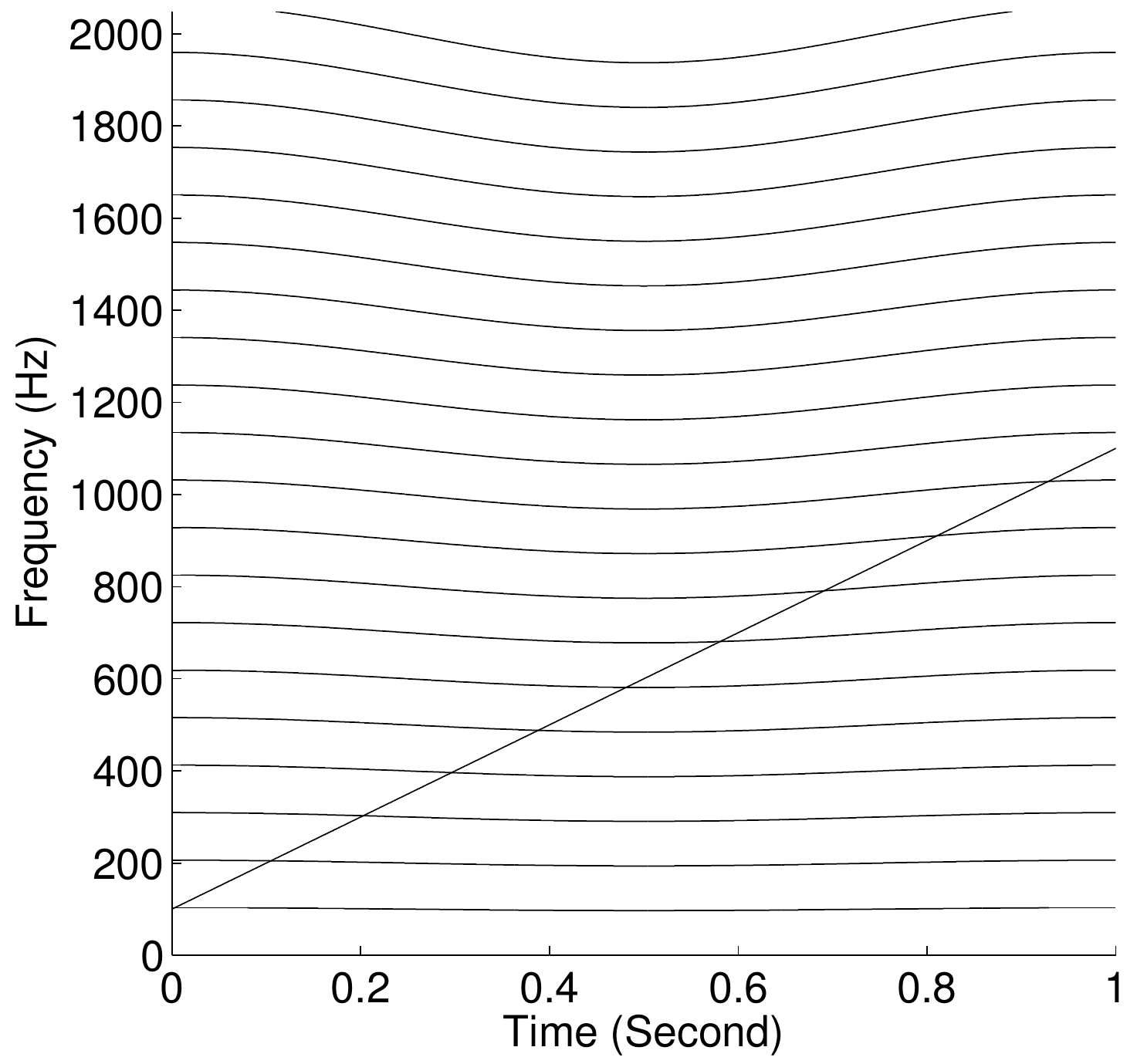}&  \includegraphics[height=1.6in]{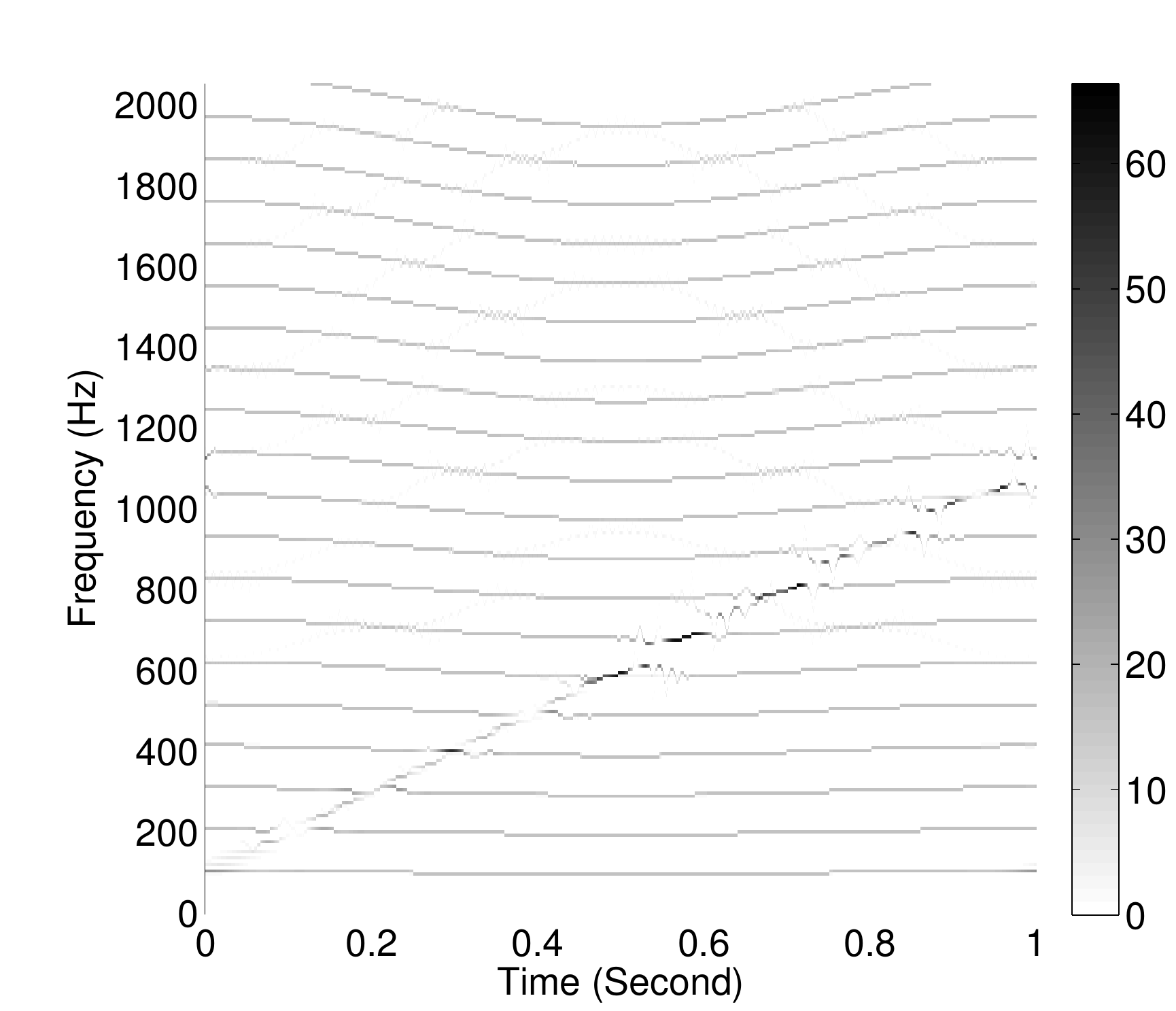} & \includegraphics[height=1.6in]{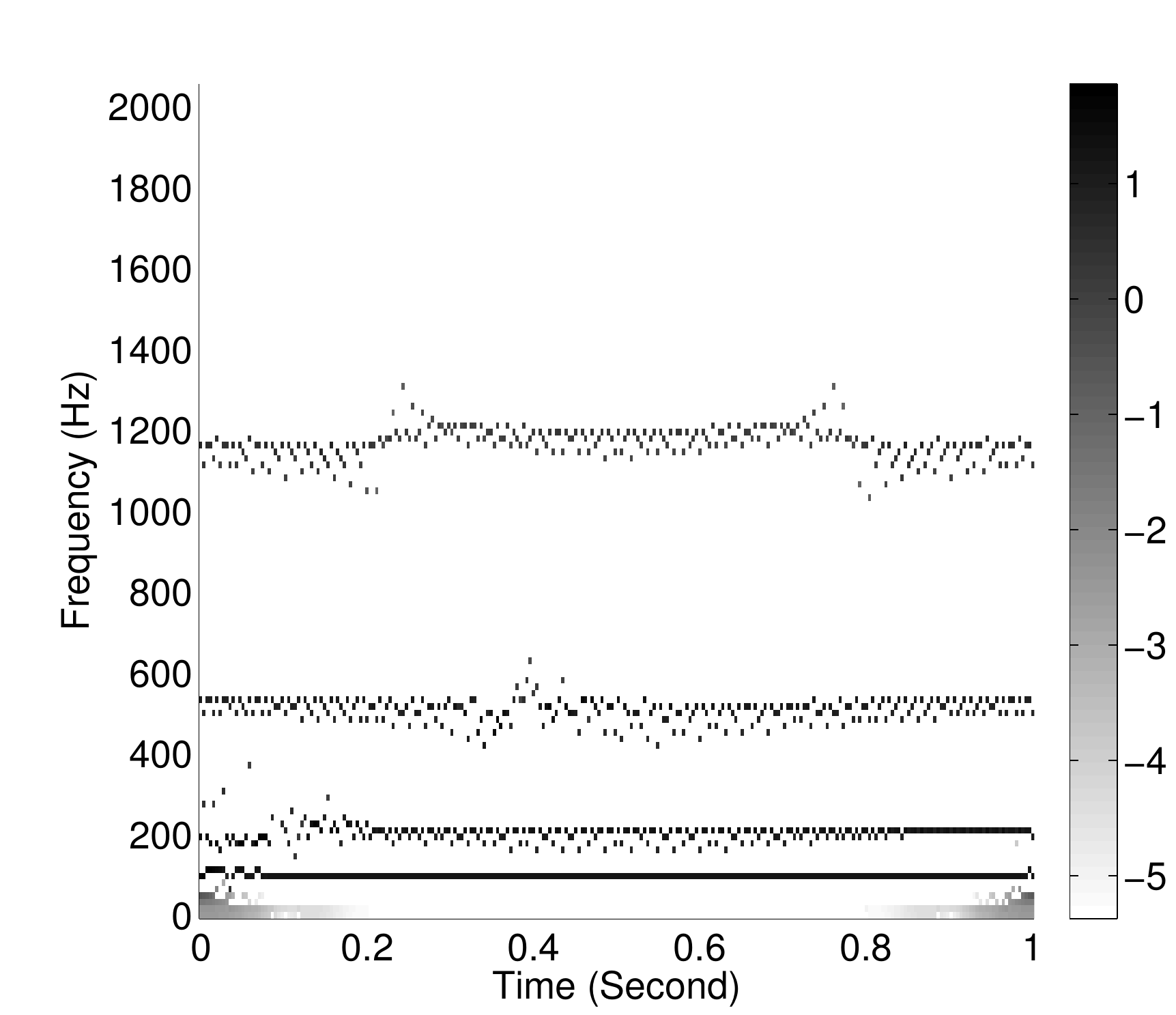}
    \end{tabular}
  \end{center}
  \caption{Left: Real instantaneous frequencies of Fourier expansion terms of $f(t)$ in Example $2$. Middle: The synchrosqueezed energy distribution of $T_f(v,b)$ provided by SSWPT.  Right: The logarithm of the synchrosqueezed energy distribution $\log_{10}(T_f(v,b))$ provided by SSWT.}
\label{fig:comp}
\end{figure}

\textbf{Example $3$:} In \cite{Daubechies2011}, the SSWT has been shown the ability to estimate the instantaneous frequencies of real ECG signals accurately.  Figure \ref{fig:multiscale} left and middle show the synchrosqueezed energy distribution of the SSWT of the real ECG signal in \cite{Daubechies2011}.  The component with the lowest frequency can reflect the instantaneous frequency of the ECG signal according with the conclusion of \cite{Daubechies2011}. However, the synchrosqueezed result in high frequency part is blurry and useless due to severe interferences between high frequency components. This limits the application of the SSWT to recover spike shape functions. Fortunately, as the Figure \ref{fig:multiscale} right shows, the SSWPT can identify $O(N^{1/s-1})=O(10)$ components, which are the Fourier expansion terms of the general mode in a spike shape. This encourages the application of the SSWPT for general mode decompositions.

\begin{figure}
  \begin{center}
    \begin{tabular}{ccc}
      \includegraphics[height=1.6in]{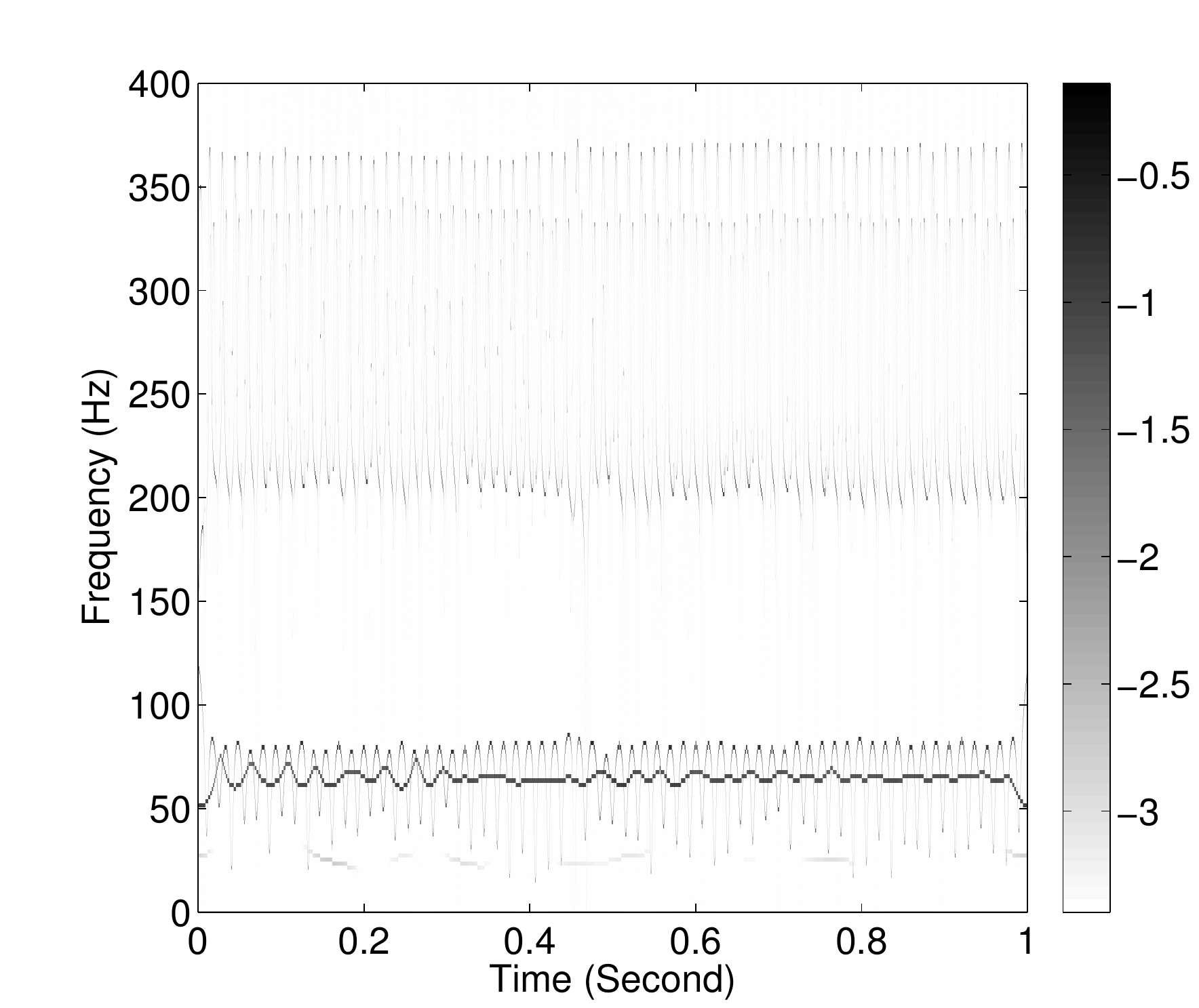} &\includegraphics[height=1.6in]{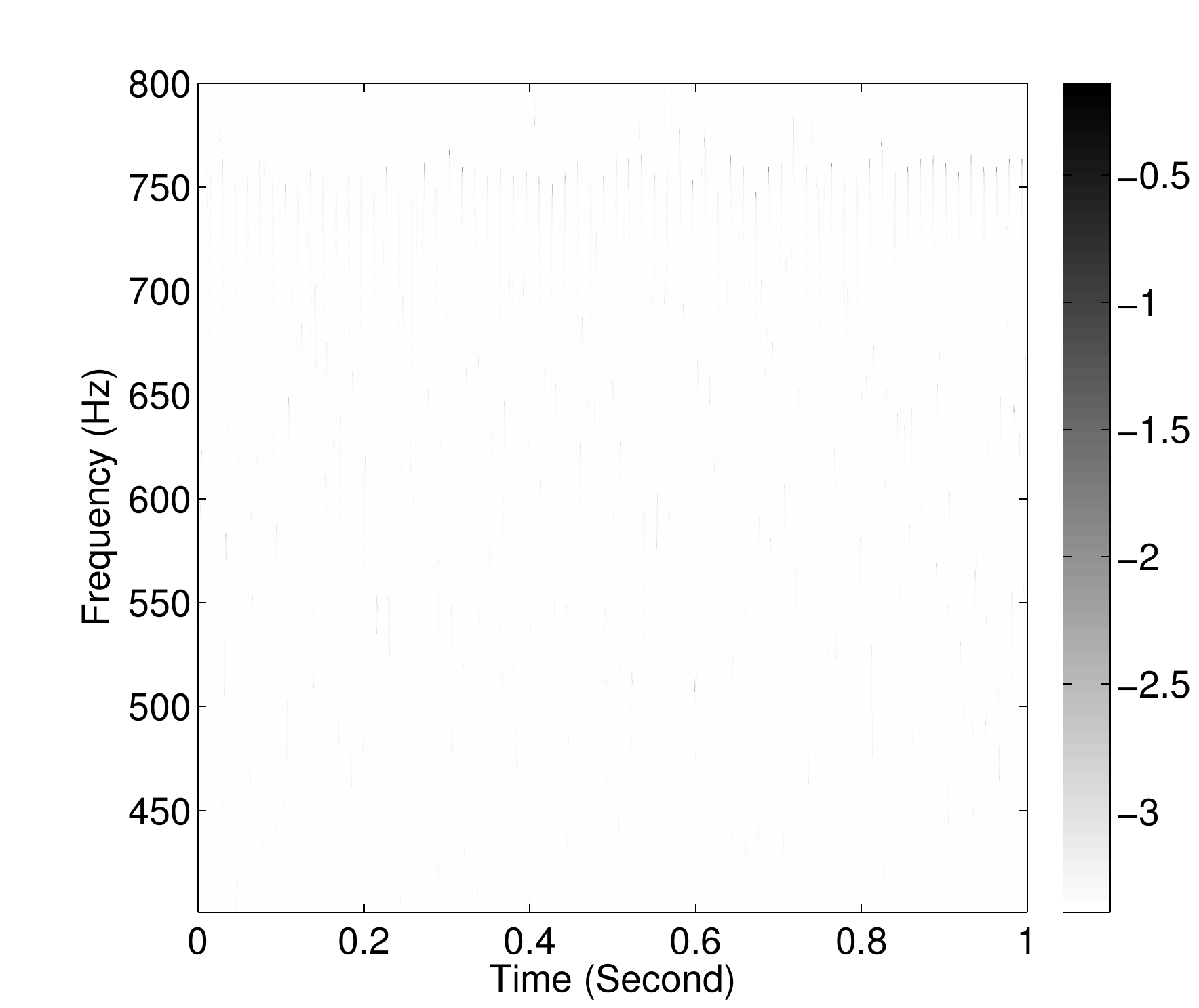} & \includegraphics[height=1.6in]{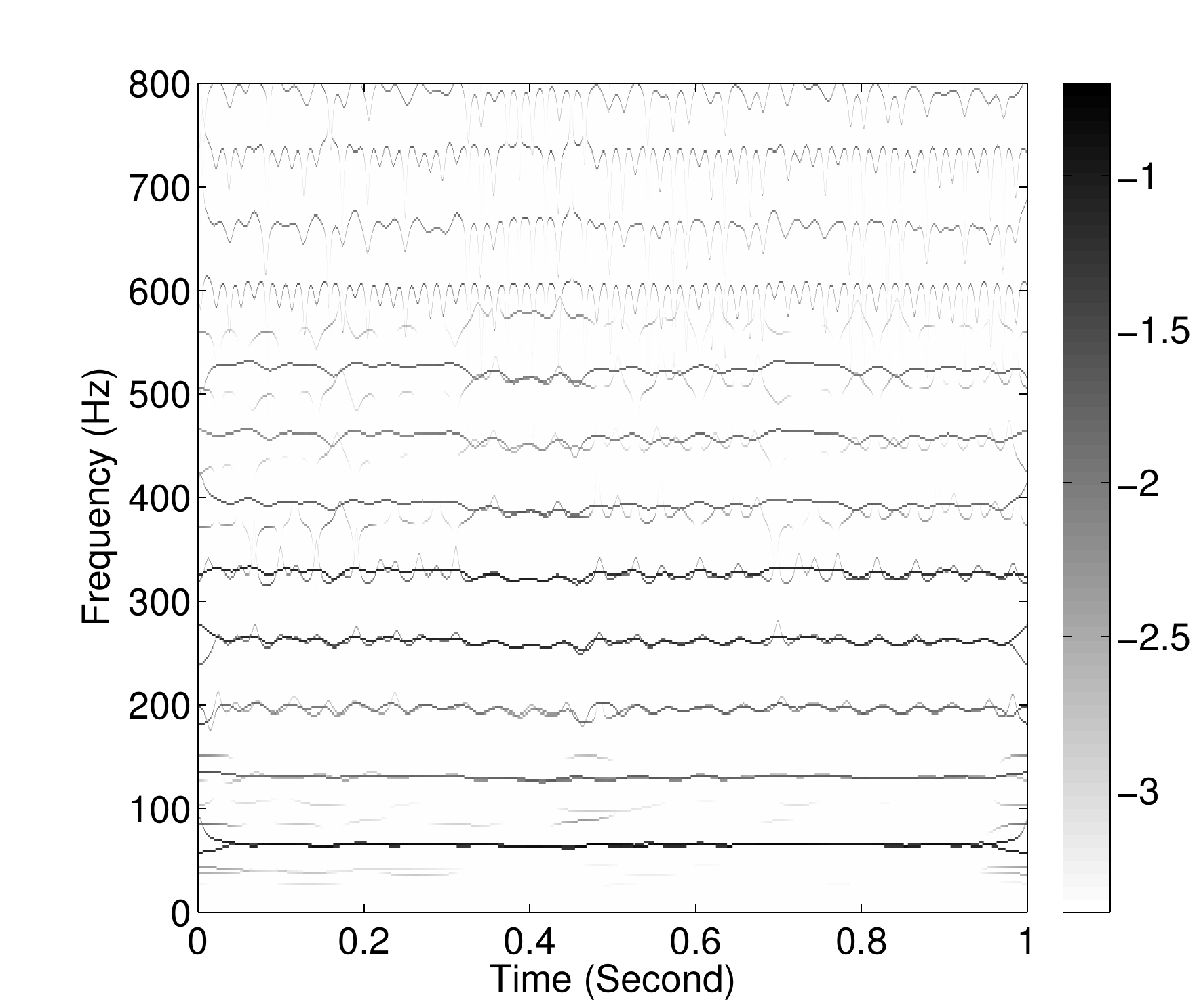} 
    \end{tabular}
  \end{center}
  \caption{Left: The logarithm of the synchrosqueezed energy distribution $\log_{10}(T_f(v,b))$ of the ECG signal in the low frequency part using the SSWT. Middle: $\log_{10}(T_f(v,b))$ of the same ECG signal in the high frequency part using the SSWT. Right: $\log_{10}(T_f(v,b))$ of the same ECG signal provided by the SSWPT. }
\label{fig:multiscale}
\end{figure}
\subsection{General mode decompositions and the robustness}
As we have seen in Example $1$, the GMDWP method and the DSA method can exactly recover general modes. In what follows, we would study the robustness against noise and present some more examples of general shape functions. The shapes of general modes are determined by all the Fourier expansion terms, including those weak energy terms which have been concealed by noise. We will show the recovered results in noisy cases and then present an example about denoising according to the feature of recovered modes. The noise used here is a Gaussian random noise $n(t)$ with zero mean and variance $\sigma^2$. To quantify the influence of the noise on each general mode, we introduce the following Signal-to-Noise Ratio ($\SNR$)
\[
\SNR [dB]=\min \left\{10\log_{10}\left(\frac{\|f_i\|_{L^2}}{\sigma^2}\right),1\leq i\leq K\right\},
\]
where $\{f_i\}_{i=1}^K$ are the general modes contained in the original signal $f(t)$. 

\textbf{Example $1$}: Let us revisit Example $1$ in Figure \ref{fig:s1s2freq} and add study its noisy case,
\begin{eqnarray*}
f(t)=\alpha_1(t)s_1(2\pi N_1\phi_1(t))+\alpha_2(t)s_2(2\pi N_2\phi_2(t))+n(t).
\end{eqnarray*}
Figure \ref{fig:EX1_ns_signal} shows three superpositions with different noise levels. As the reconstructed results show in Figure \ref{fig:EX1_ns_rec}, the instantaneous frequencies are accurately estimated, even if the signal is disturbed by severe noise. The essential feature of the general modes are recovered. When the noise is overwhelming the general modes, additional denoising procedure is application dependent, as we will show in the next example.

\begin{figure}[ht!]
  \begin{center}
    \begin{tabular}{c}
      \includegraphics[height=1.3in]{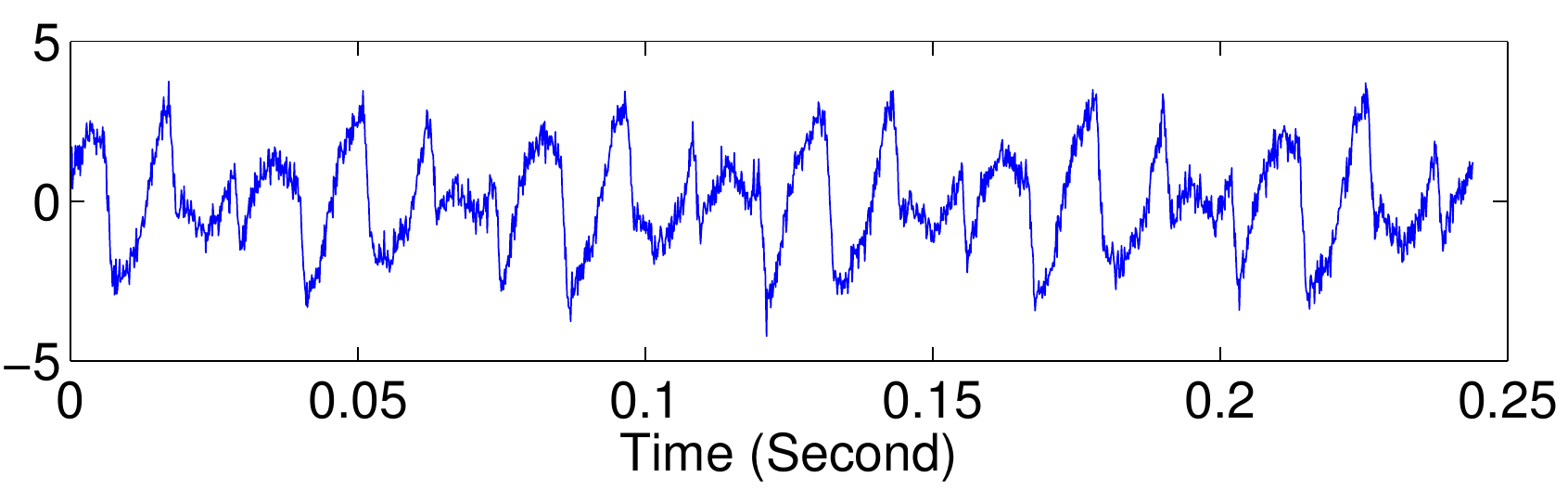}  \\
\includegraphics[height=1.3in]{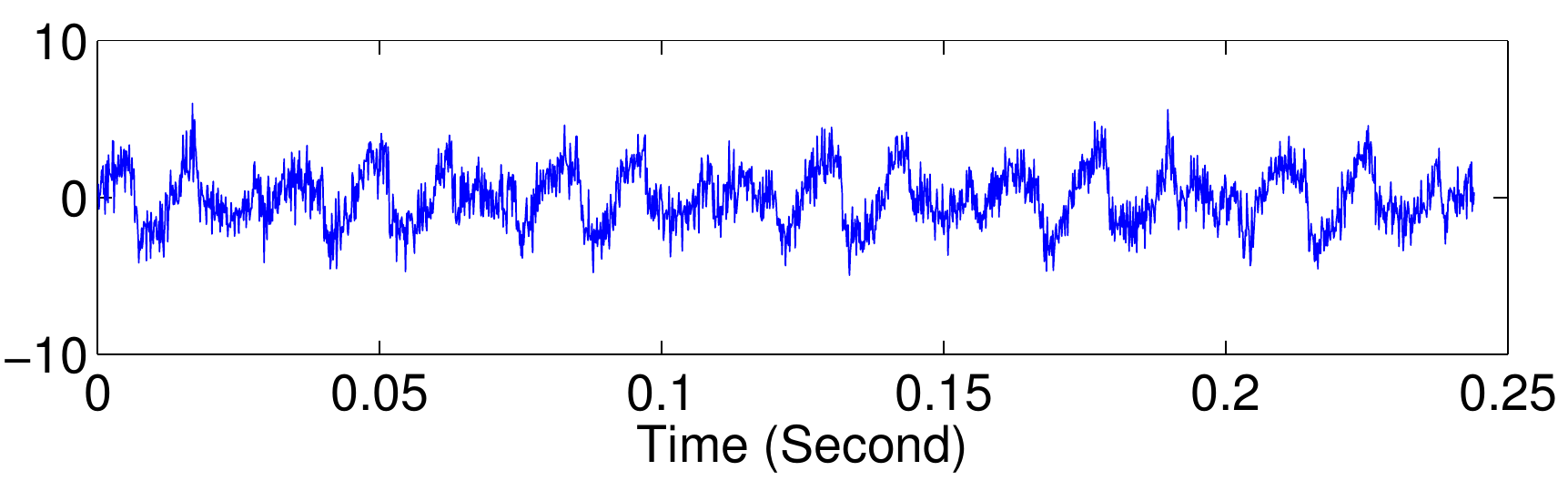}\\
\includegraphics[height=1.3in]{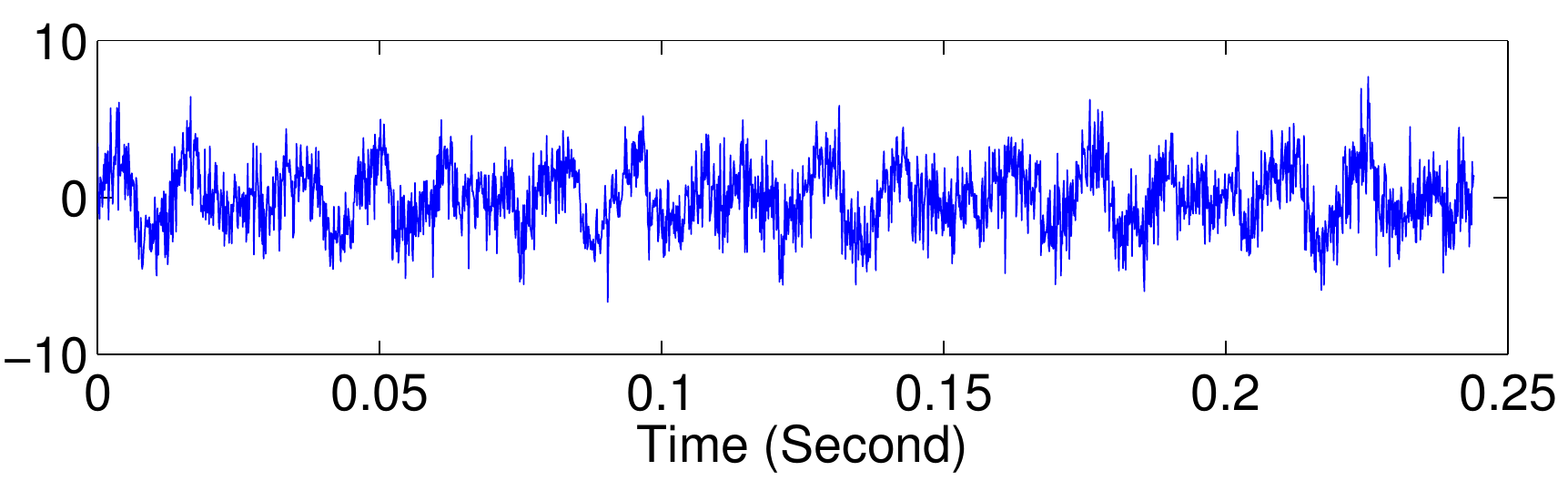}
    \end{tabular}
  \end{center}
  \caption{Noisy signals of Example $1$ and their $\SNR$s are $6$, $0$, and $-3$, respectively. }
\label{fig:EX1_ns_signal}
\end{figure}

\begin{figure}
  \begin{center}
    \begin{tabular}{cccc}
      \includegraphics[height=1.3in]{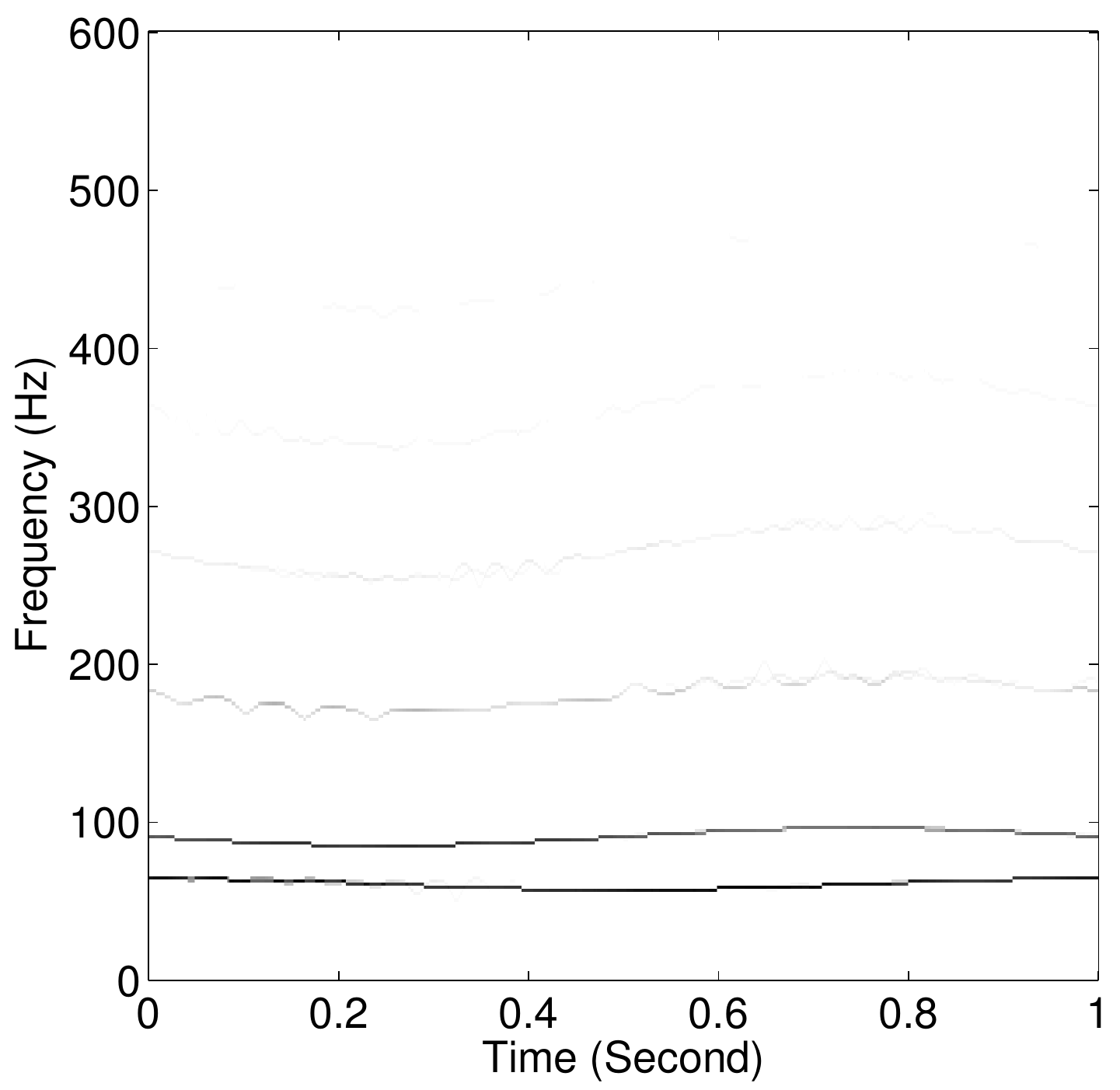}&  \includegraphics[height=1.3in]{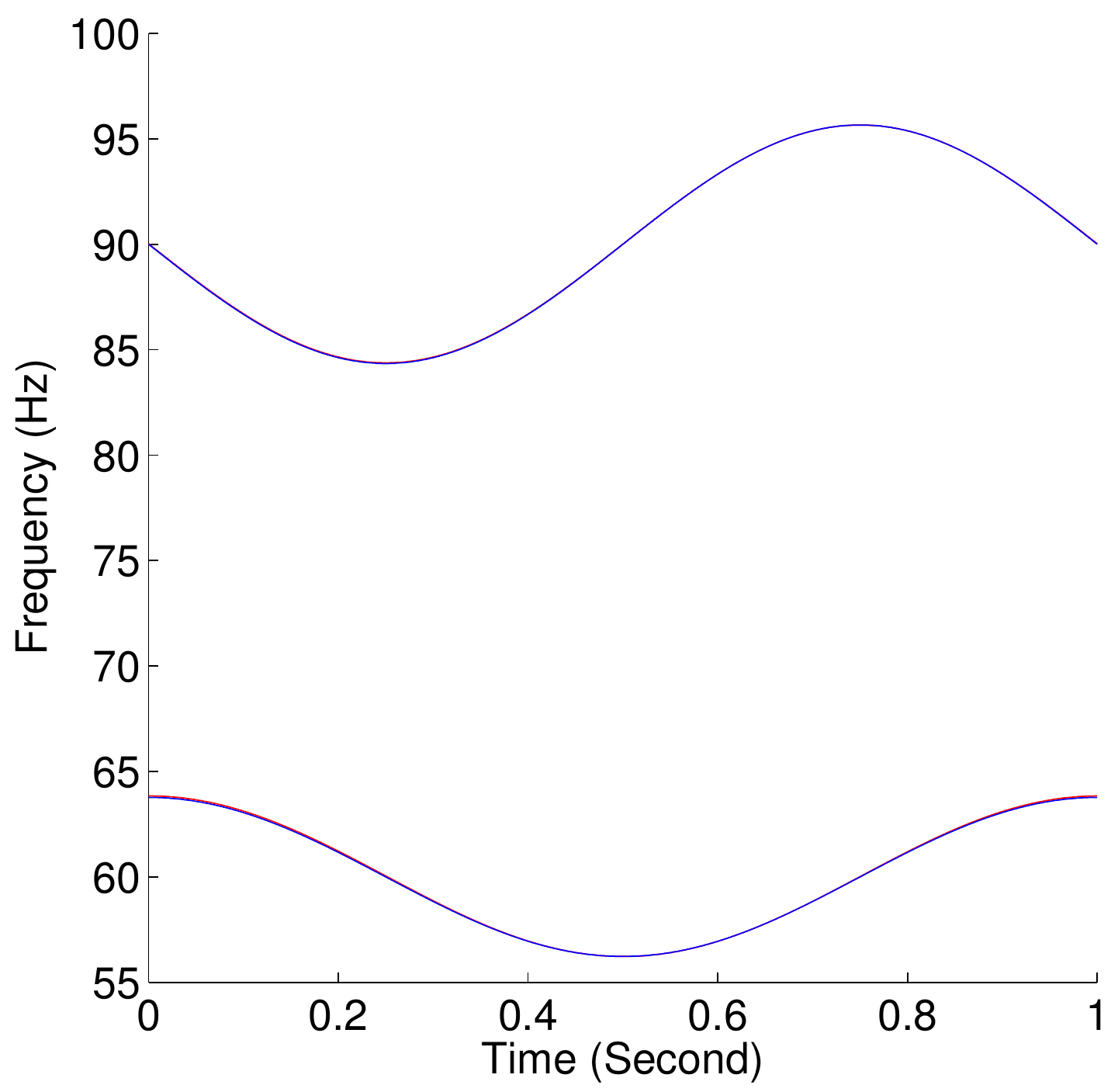} 
     & \includegraphics[height=1.3in]{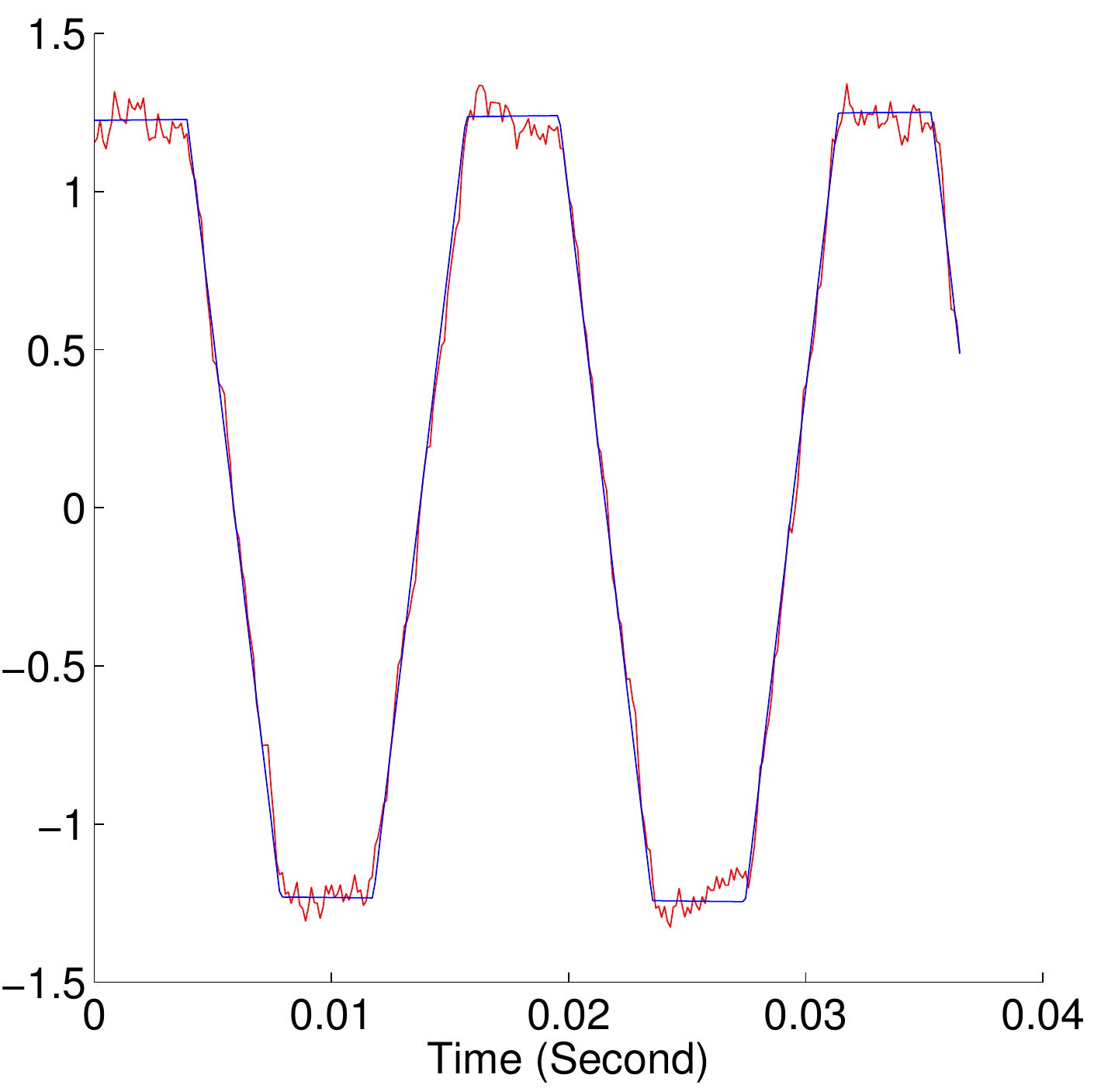} & \includegraphics[height=1.3in]{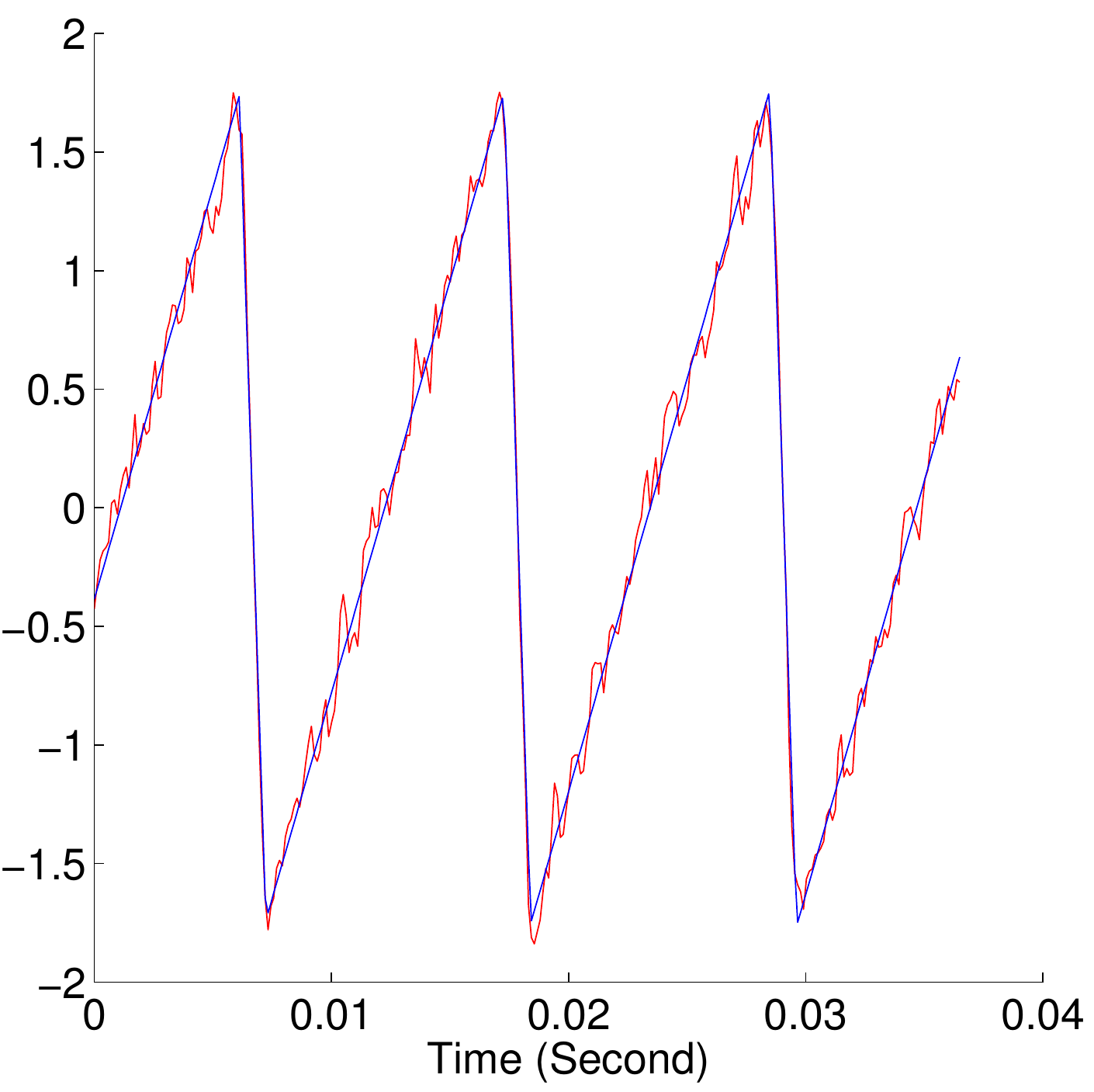}\\
      \includegraphics[height=1.3in]{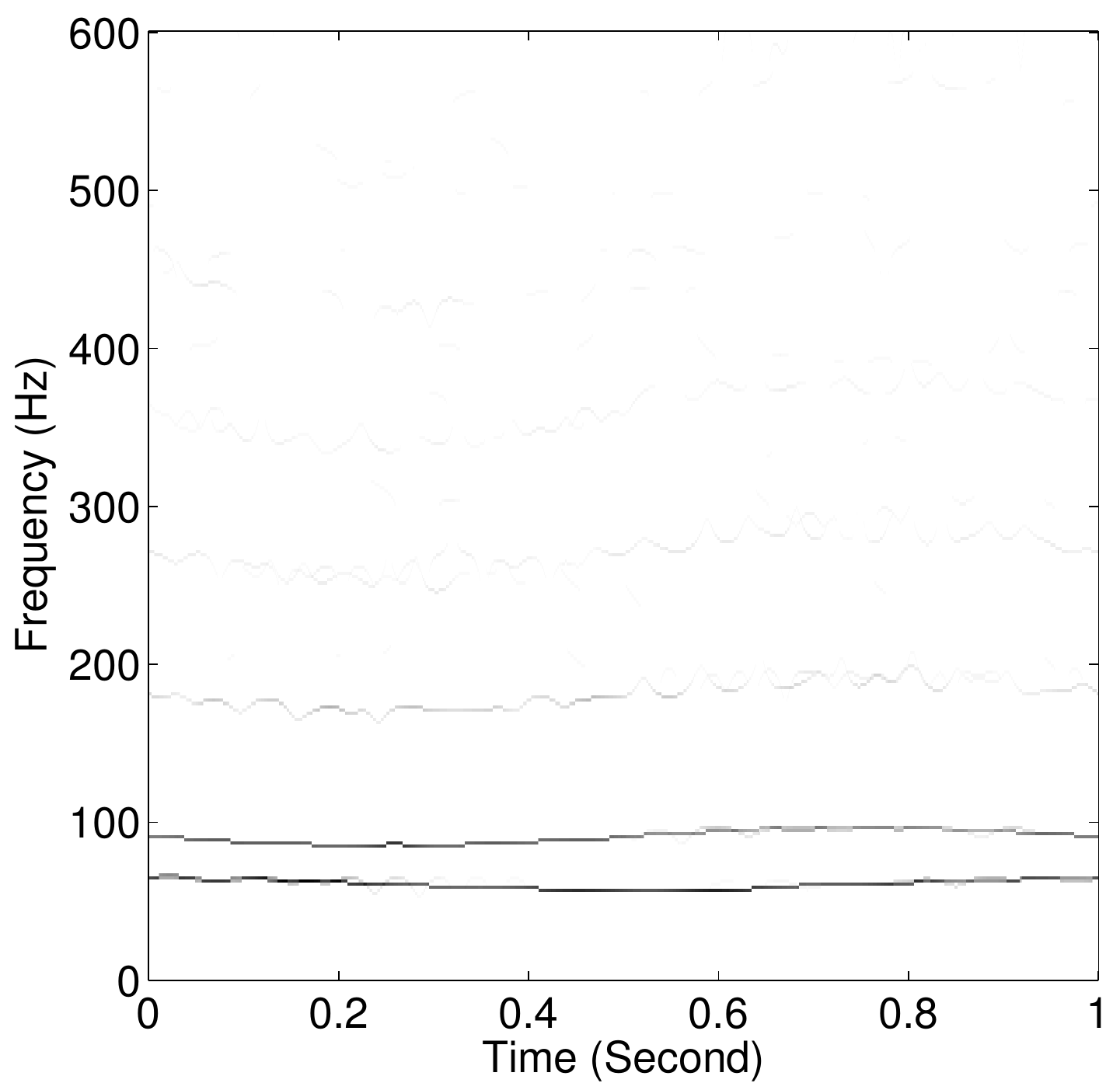}&  \includegraphics[height=1.3in]{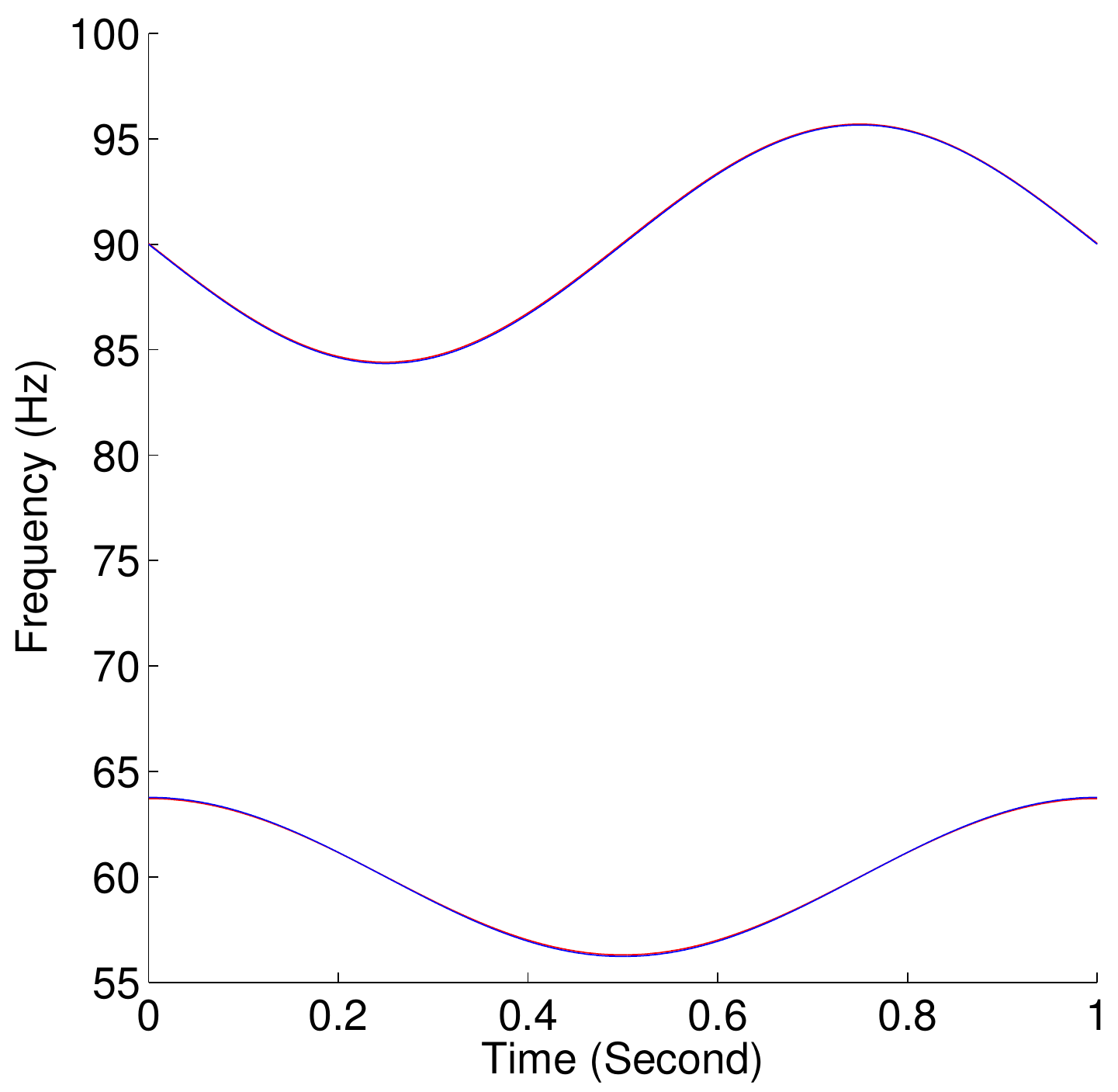} 
     & \includegraphics[height=1.3in]{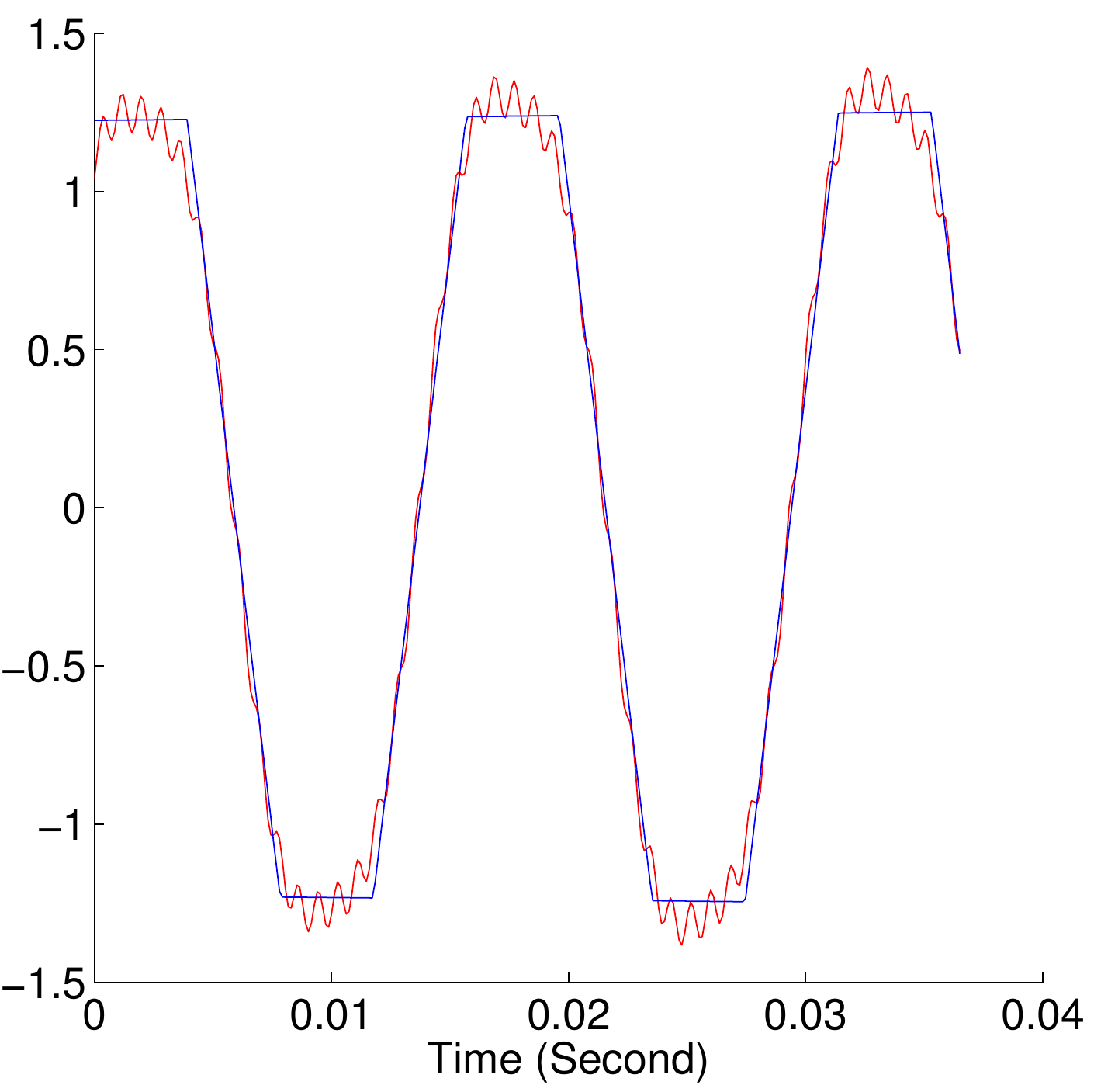} & \includegraphics[height=1.3in]{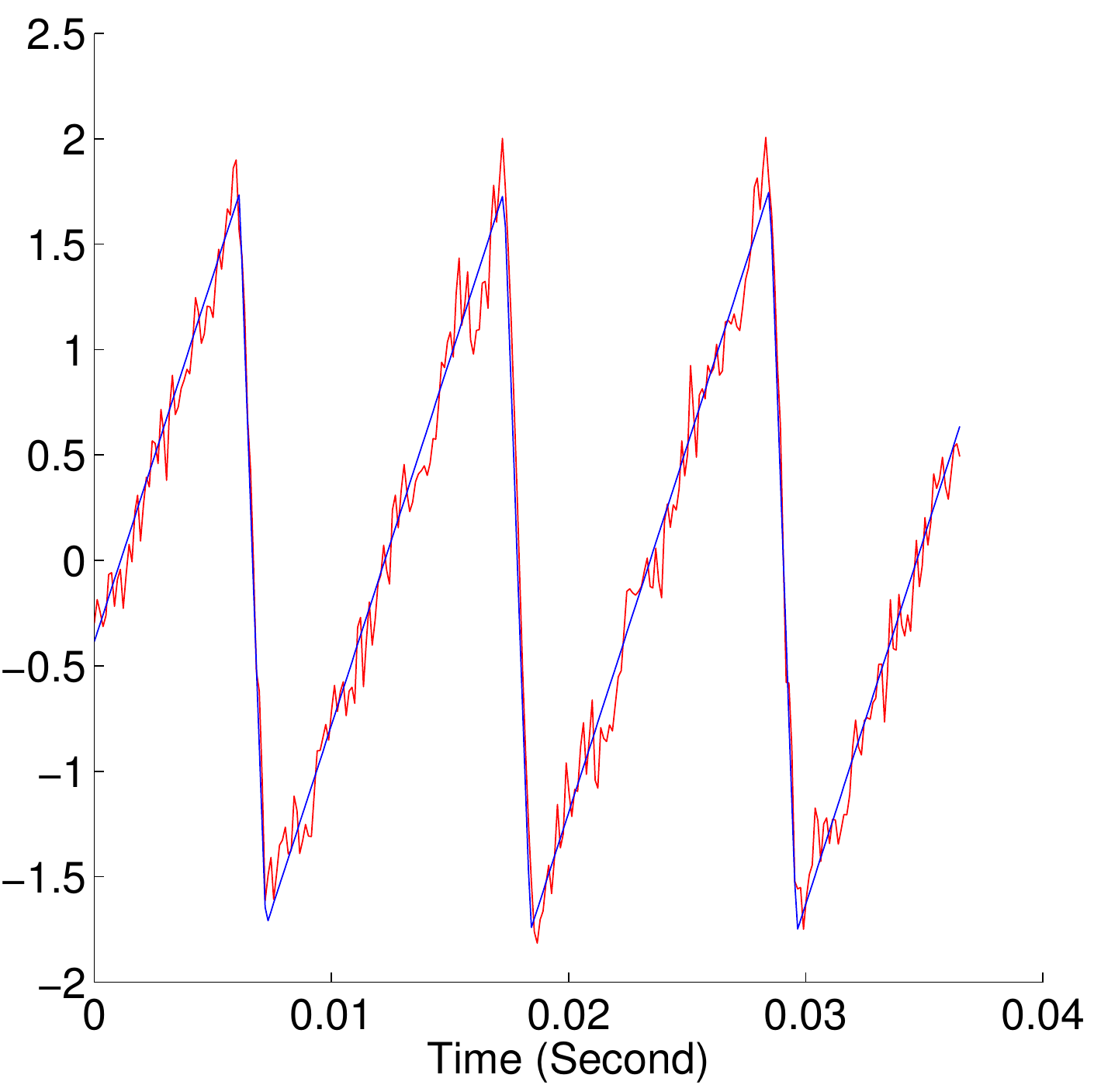}\\
 \includegraphics[height=1.3in]{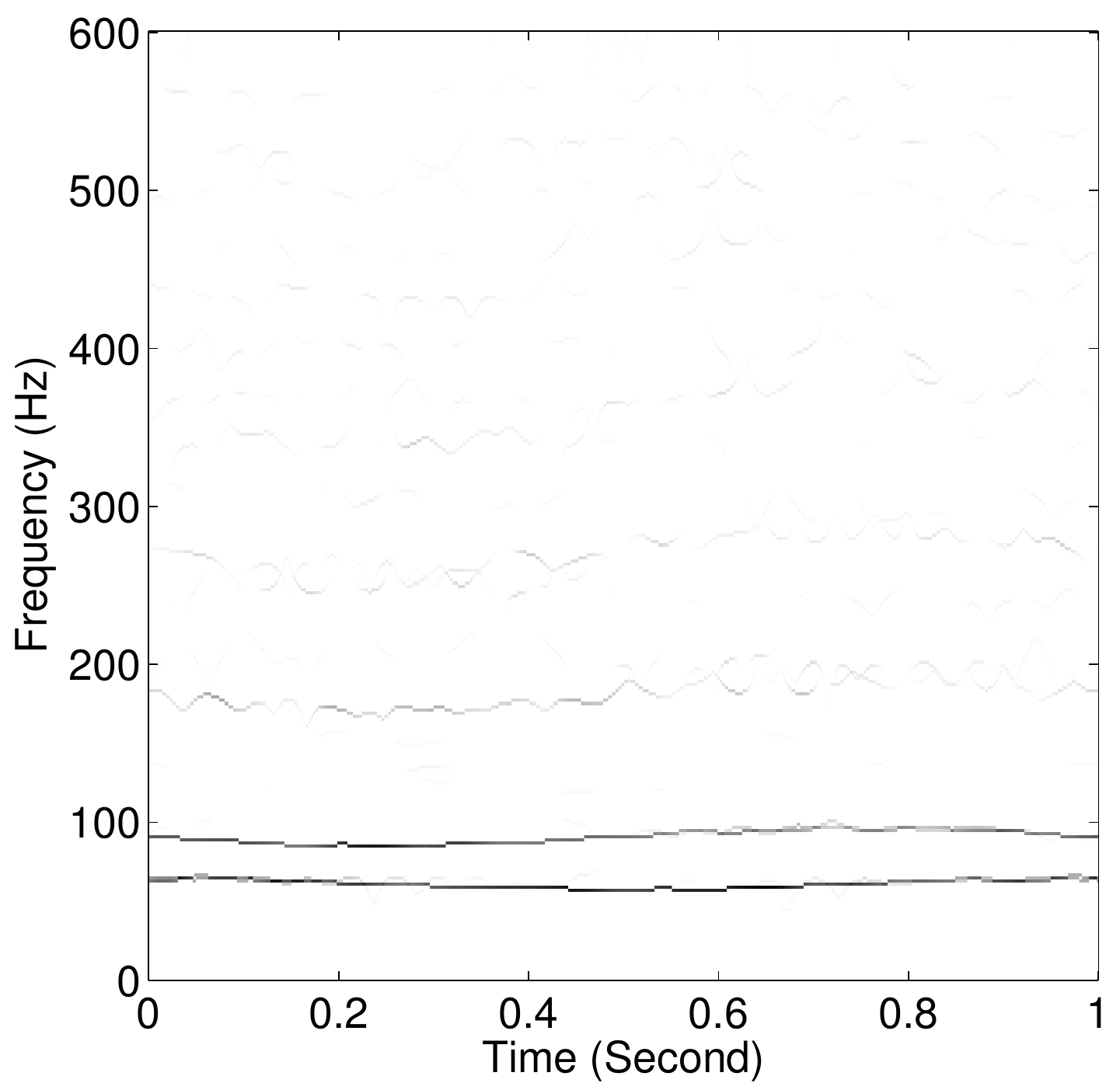}&  \includegraphics[height=1.3in]{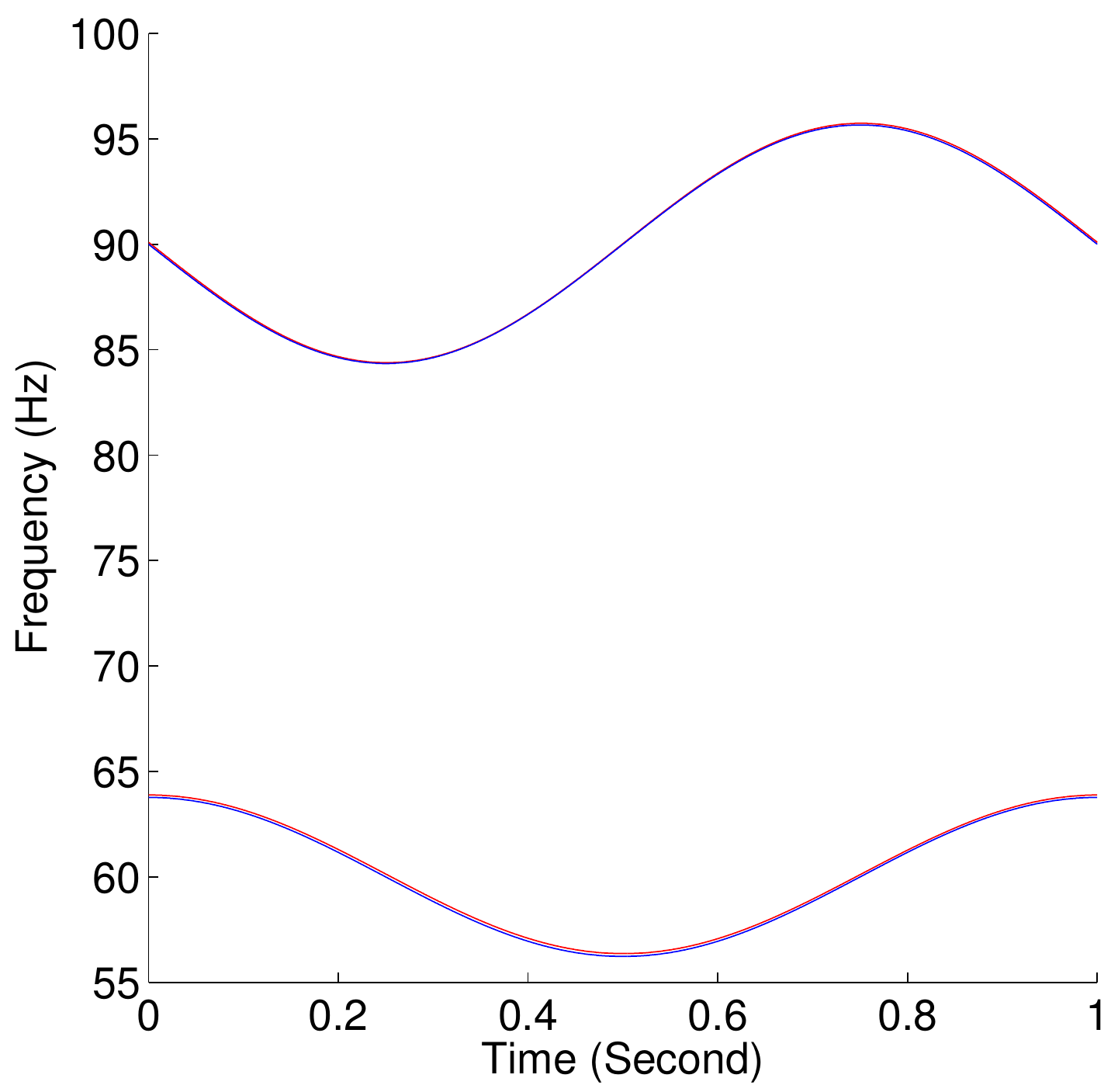} 
     & \includegraphics[height=1.3in]{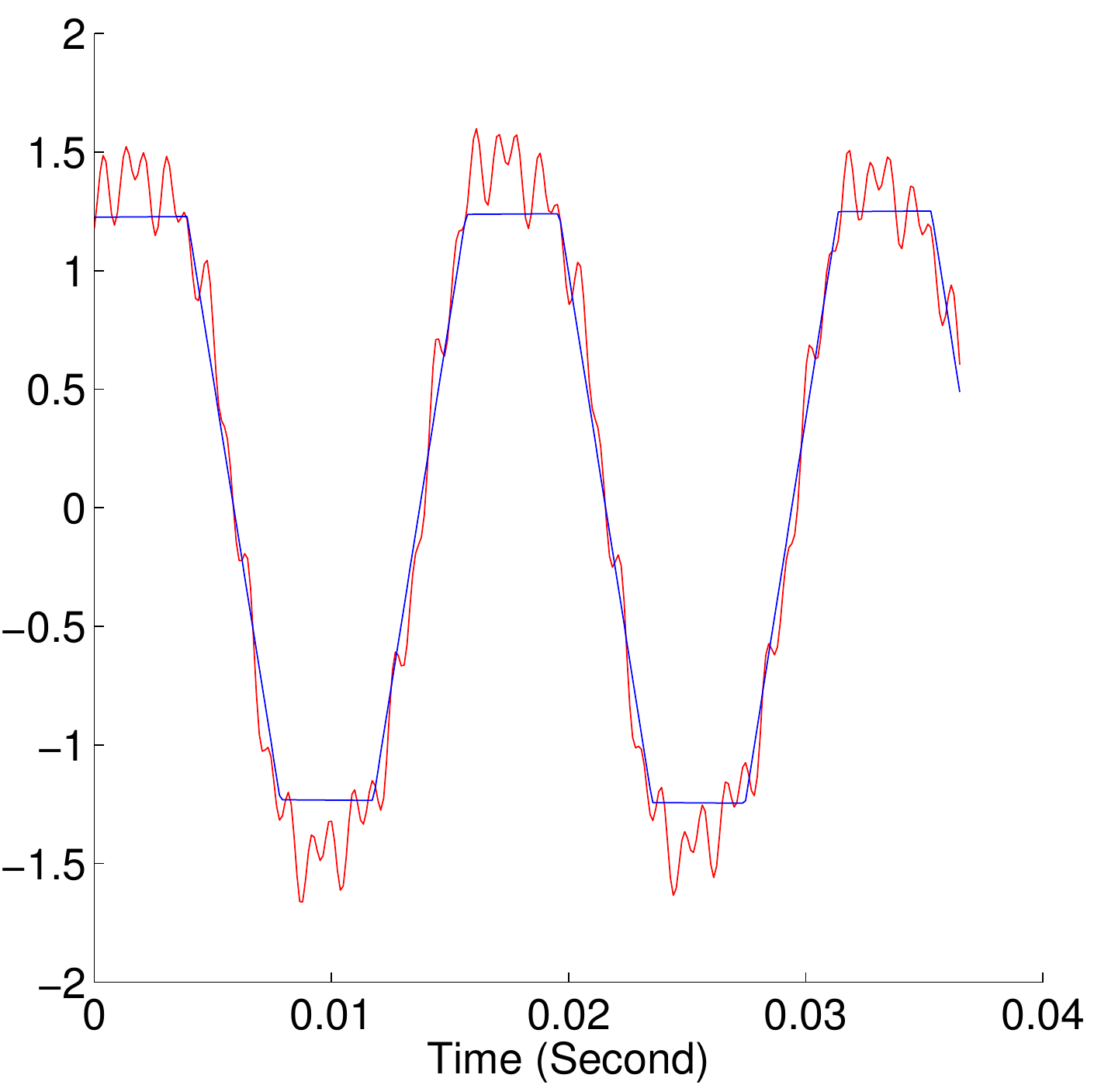} & \includegraphics[height=1.3in]{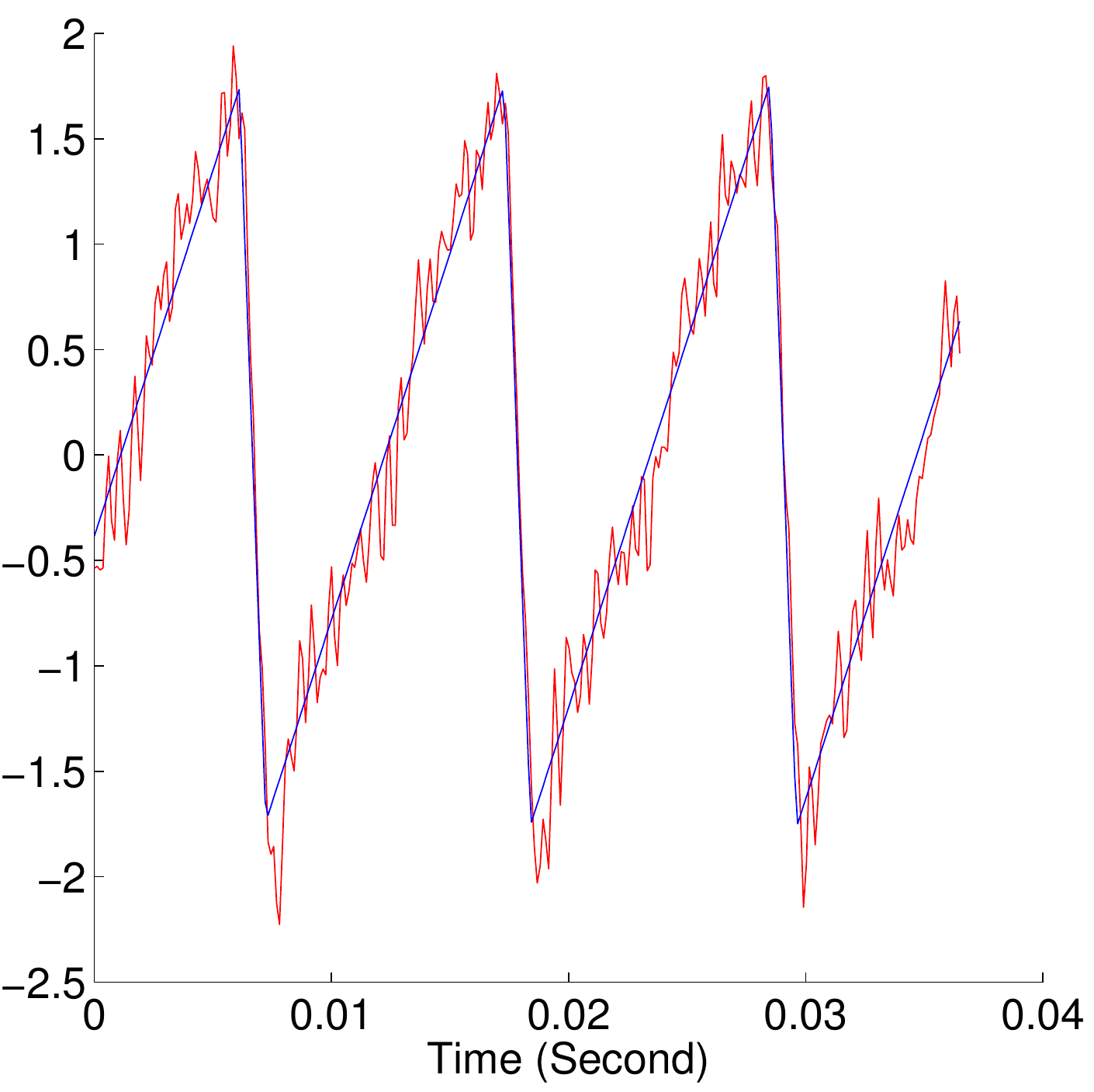}\\
    \end{tabular}
  \end{center}
  \caption{From top to bottom are the results of noisy Example $1$ with different $\SNR$s $6$, $0$, and $-3$, respectively. Left: The synchrosqueezed energy distributions of signals. Middle left: The real instantaneous frequencies (blue) and the estimated instantaneous frequencies (red). Middle right and right: Recovered general modes.}
\label{fig:EX1_ns_rec}
\end{figure}

\textbf{Example $4$:}
Combining the proposed methods with some post processing techniques can detect general shape functions in a wider class than the one defined in Definition \ref{def:GSF}.  For example, we study the superposition of two general modes with piecewise constant shape functions.  A noisy superposition of general modes is generated as follows.
\begin{eqnarray*}
f(t)=\alpha_3(t)s_3(2\pi N_3\phi_3(t))+\alpha_4(t)s_4(2\pi N_4\phi_4(t))+n(t),
\end{eqnarray*}
where $s_3(t)$ and $s_4(t)$ are defined in Figure \ref{fig:s3s4}, $\alpha_3(t) = 1+0.4\sin(4\pi t)$, $\alpha_4(t) = 1-0.3\sin(2\pi t)$, $N_3=120$, $N_4=185$, $\phi_3(t) = t+0.005\sin(2\pi t)$, and $\phi_4(t)= t+0.01\cos(4\pi t)$. In this example, the 1D SSWPT is applied to estimate the instantaneous information first and then the DSA method is applied to decompose $f(t)$ into two general modes. Finally, a TV norm minimization is applied to obtain the final results shown in Figure \ref{fig:s3s4decom}. The DSA method is able to detect the basic feature of these general modes and the post processing TV norm minimization helps to reduce the noise.

\begin{figure}
  \begin{center}
    \begin{tabular}{cccc}
       \includegraphics[height=1.3in]{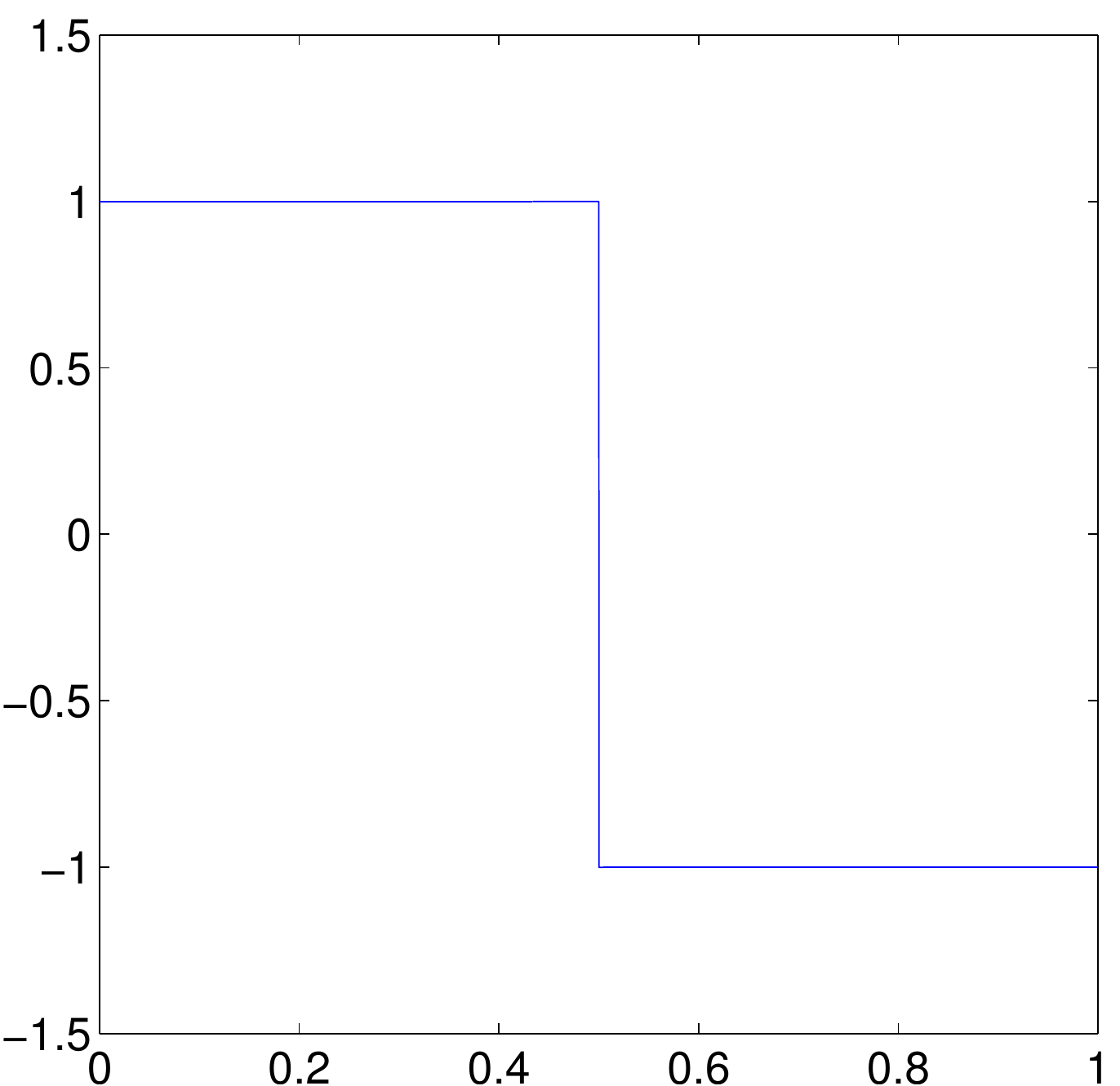}&  \includegraphics[height=1.3in]{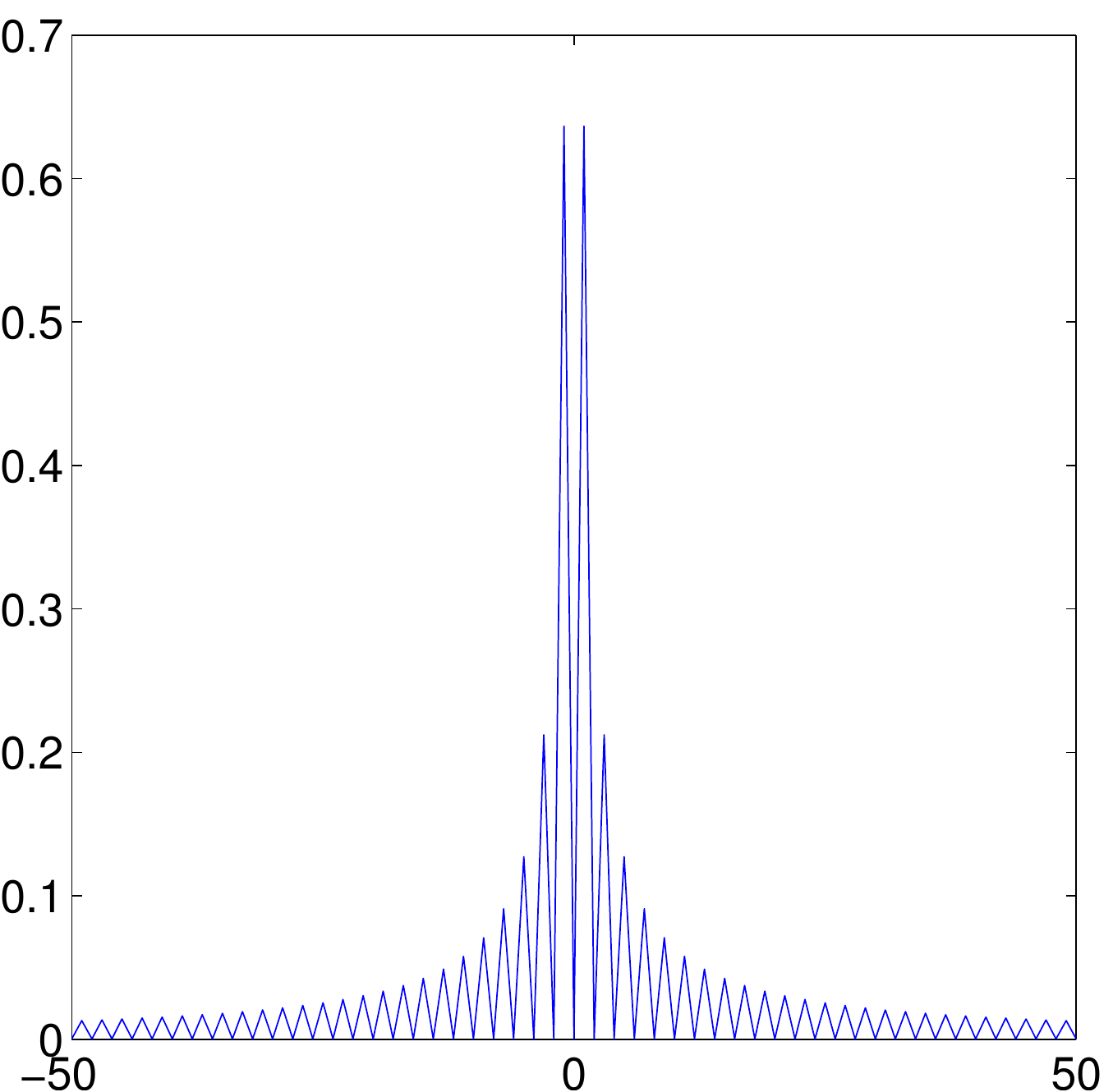}& 
\includegraphics[height=1.3in]{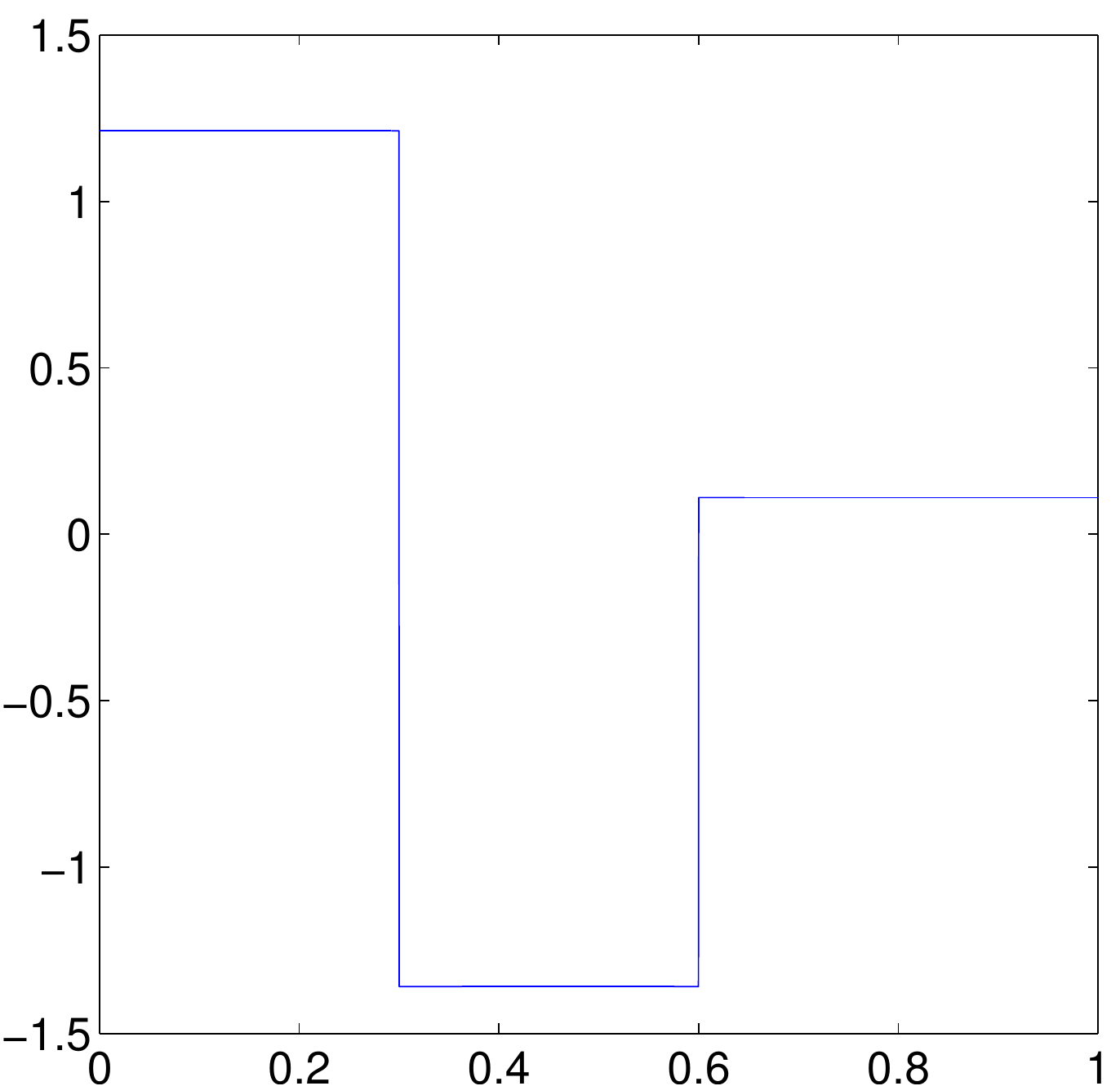} & \includegraphics[height=1.3in]{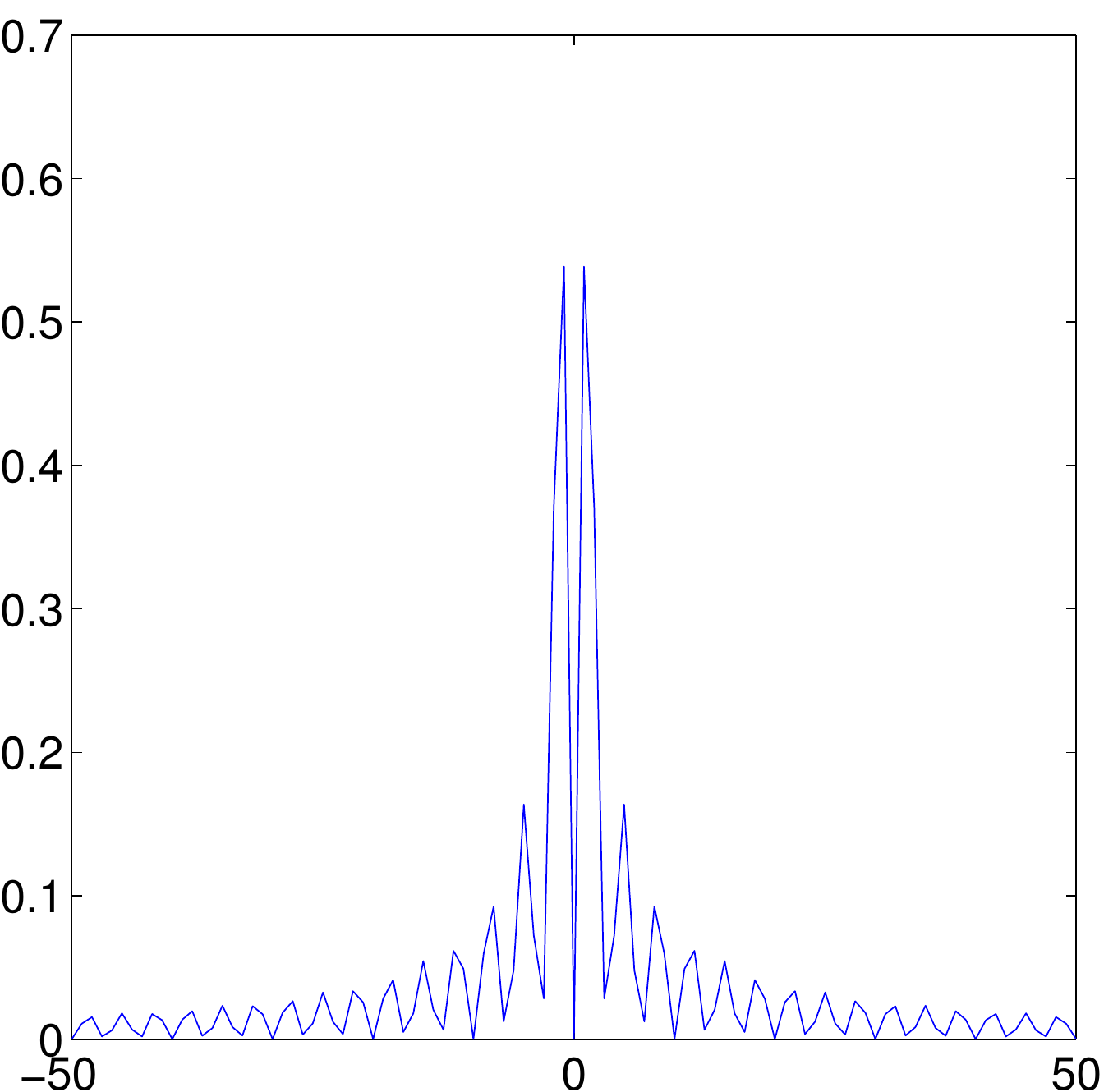} 
    \end{tabular}
  \end{center}
  \caption{Left: A piecewise constant general shape function  $s_3(t)$ and its spectral energy $|\widehat{s_3}(\xi)|$. Right: $s_4(t)$ and its spectral energy $|\widehat{s_4}(\xi)|$. }
\label{fig:s3s4}
\end{figure}

\begin{figure}
  \begin{center}
    \begin{tabular}{cc}
      \includegraphics[height=1.8in]{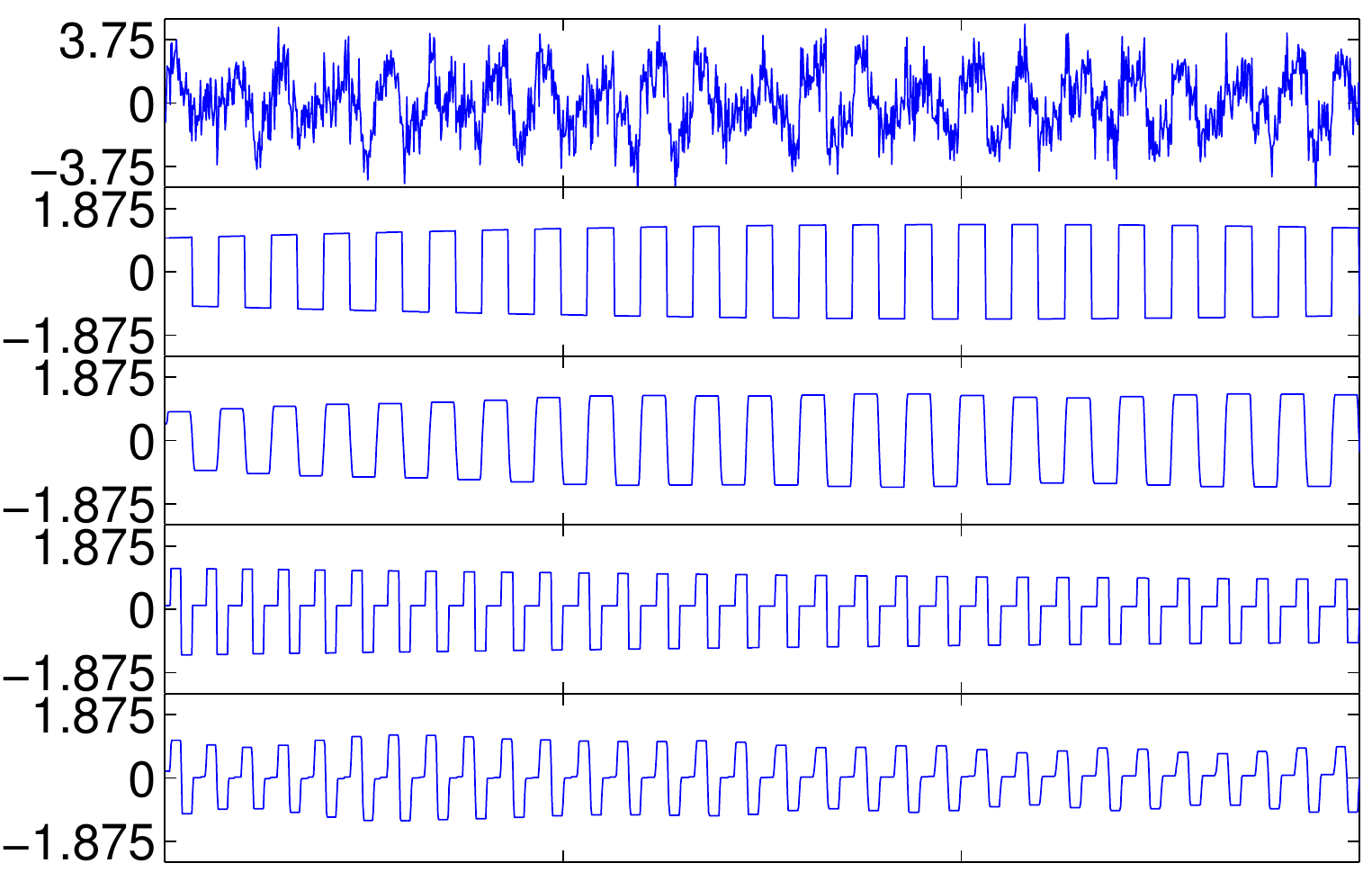}&
\includegraphics[height=1.8in]{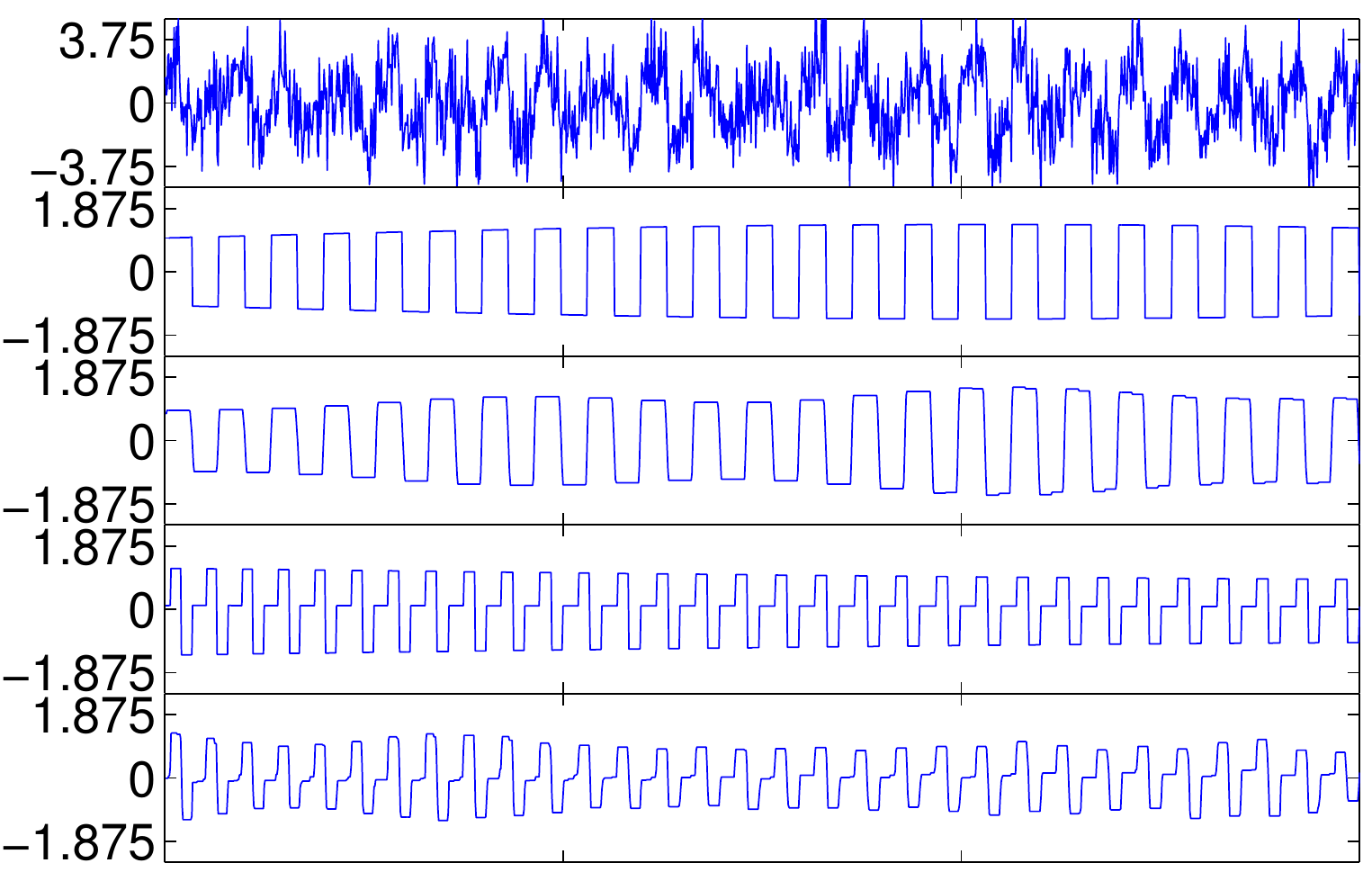}
    \end{tabular}
  \end{center}
  \caption{Left: Decomposition results of a noisy signal with $\SNR=0$. Right: Decomposition results of a noisy signal with $\SNR=-3$.  The first row: Noisy superpositions. The second and the third row: The first noiseless general mode and its recovered result.  The forth and the fifth row: The second noiseless general mode and its recovered result. }
\label{fig:s3s4decom}
\end{figure}

\subsection{Real applications}
\textbf{Example $5$:} In the first example of real applications, we study the ECG signals. Two real ECG general shape functions $s_5(t)$ and $s_6(t)$ (see Figure \ref{fig:s5s6}) are cut out from real ECG signals in \cite{PhysioNet} and \cite{Hau-Tieng2013} . A noiseless superposition is generated as follows.
\begin{eqnarray*}
f(t)=\alpha_5(t)s_5(2\pi N_5\phi_5(t))+\alpha_6(t)s_6(2\pi N_6\phi_6(t)),
\end{eqnarray*}
where $\alpha_5(t) = 1+0.05\sin(2\pi t)$, $\alpha_6(t) = 1+0.05\cos(2\pi t)$, $N_5=150$, $N_6=220$, $\phi_5(t) = t+0.006\sin(2\pi t)$, and $\phi_6(t)= t+0.006\cos(2\pi t)$. The instantaneous frequencies and the real shape functions are accurately estimated as shown in Figure \ref{fig:s5s6}. To demonstrate the robustness of the proposed methods for ECG signals, a noisy superposition is generated by adding a noise term $n(t)$. As shown in Figure \ref{fig:s5s6ns}, the synchrosqueezed energy distribution is well concentrated around the real instantaneous frequencies and the instantaneous frequencies are accurately estimated. Most importantly, the main spikes of real ECG shape functions are precisely recovered, even if the $\SNR$ is small.

\begin{figure}
  \begin{center}
    \begin{tabular}{c}
      \includegraphics[height=1.6in]{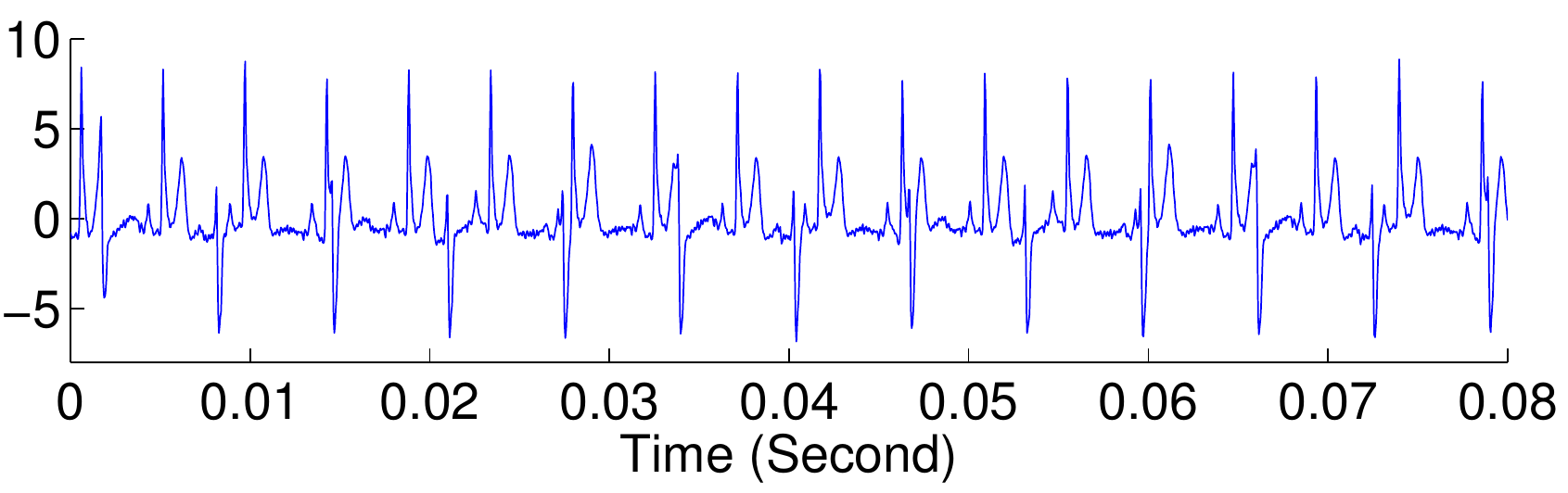}  
    \end{tabular}
    \begin{tabular}{cccc}
     \includegraphics[height=1.3in]{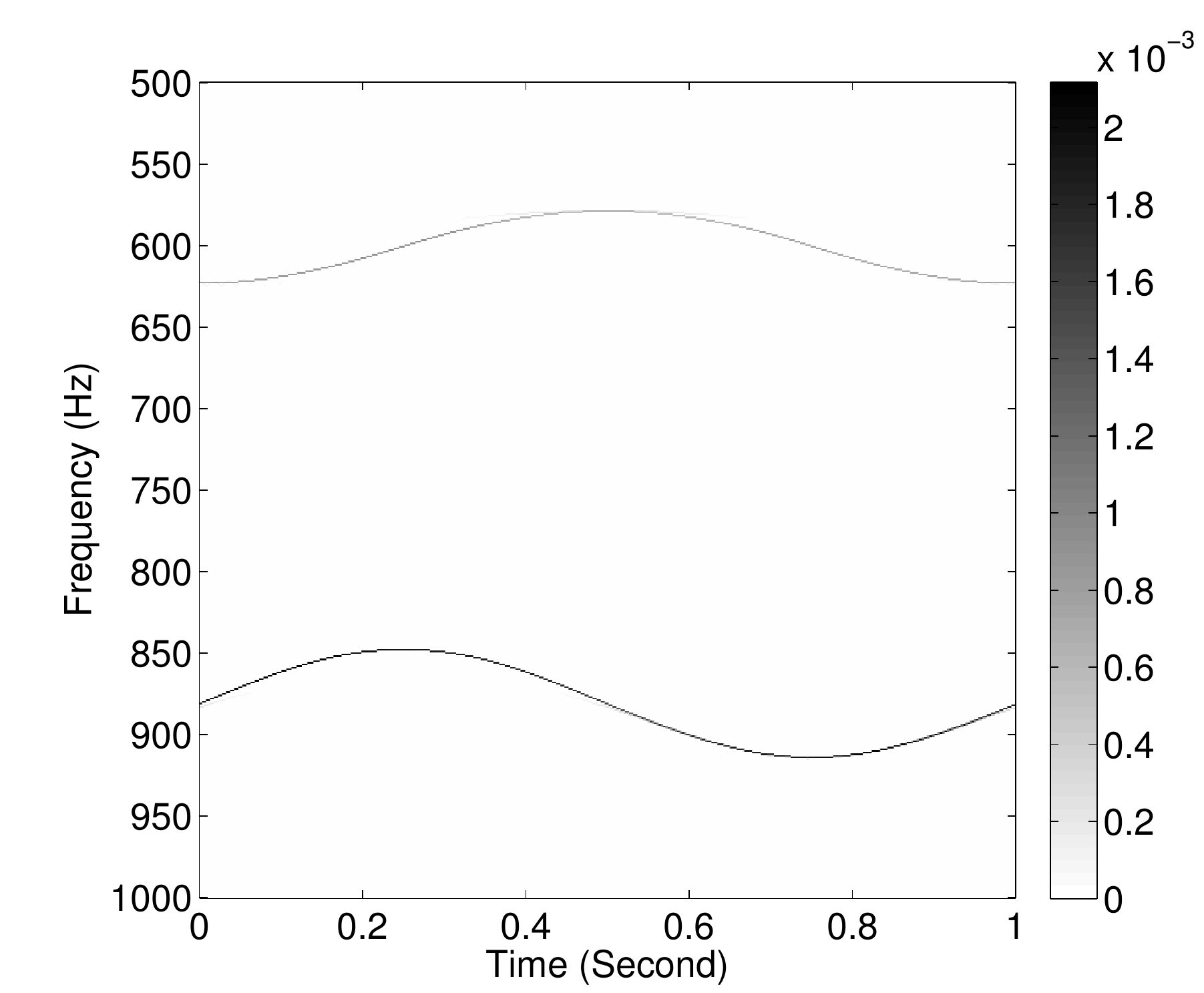}&           \includegraphics[height=1.3in]{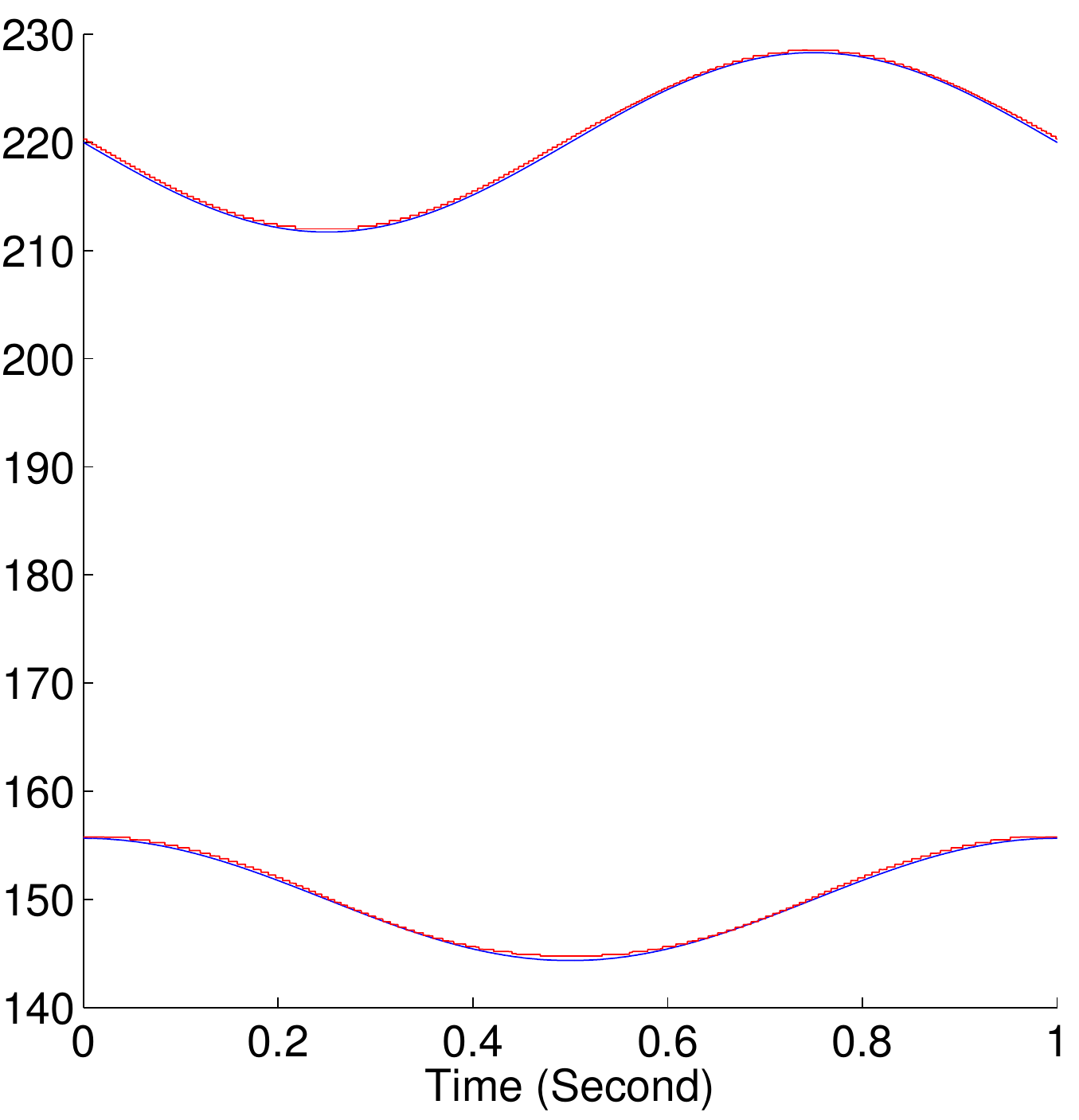}
&\includegraphics[height=1.3in]{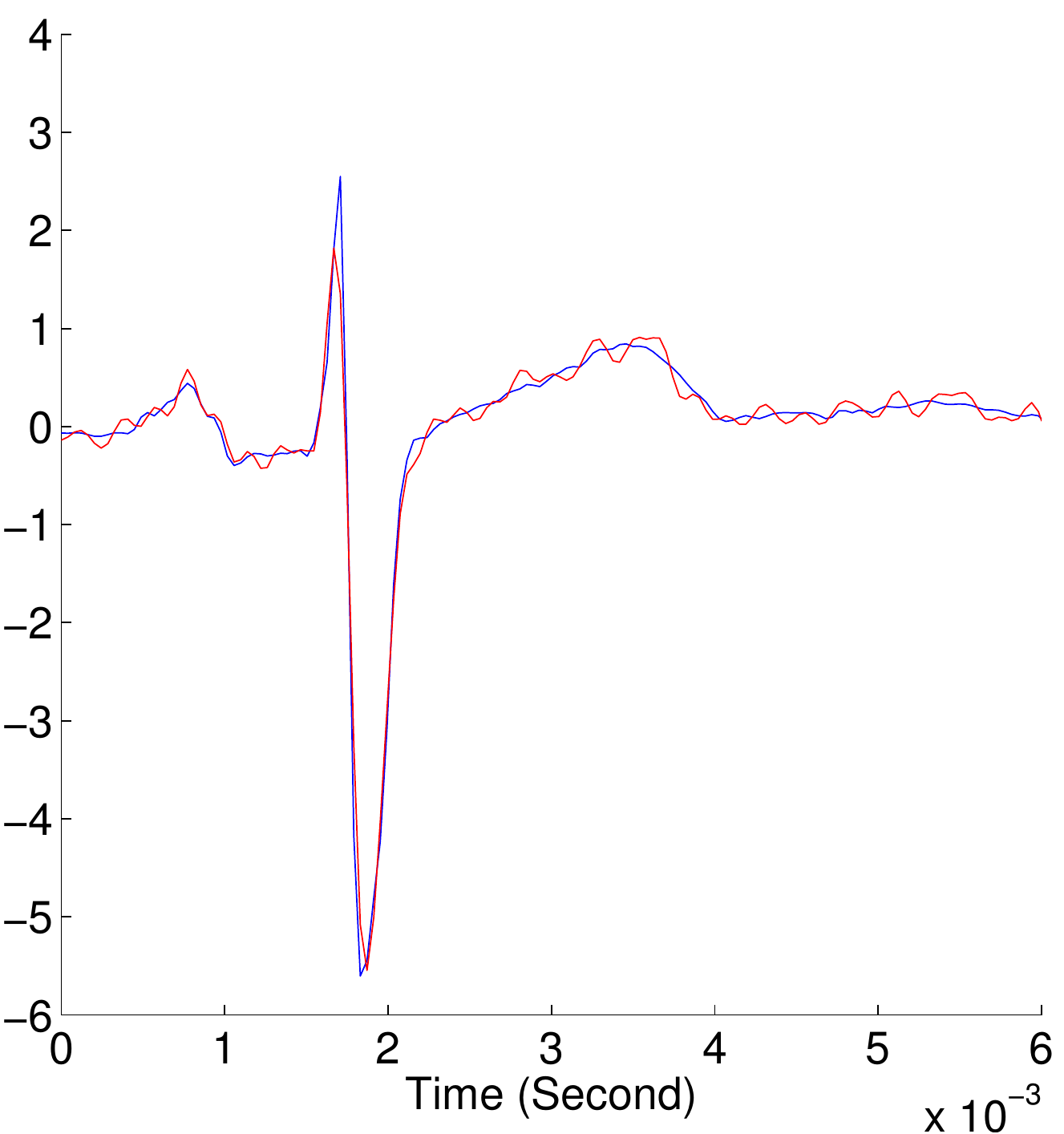} & \includegraphics[height=1.3in]{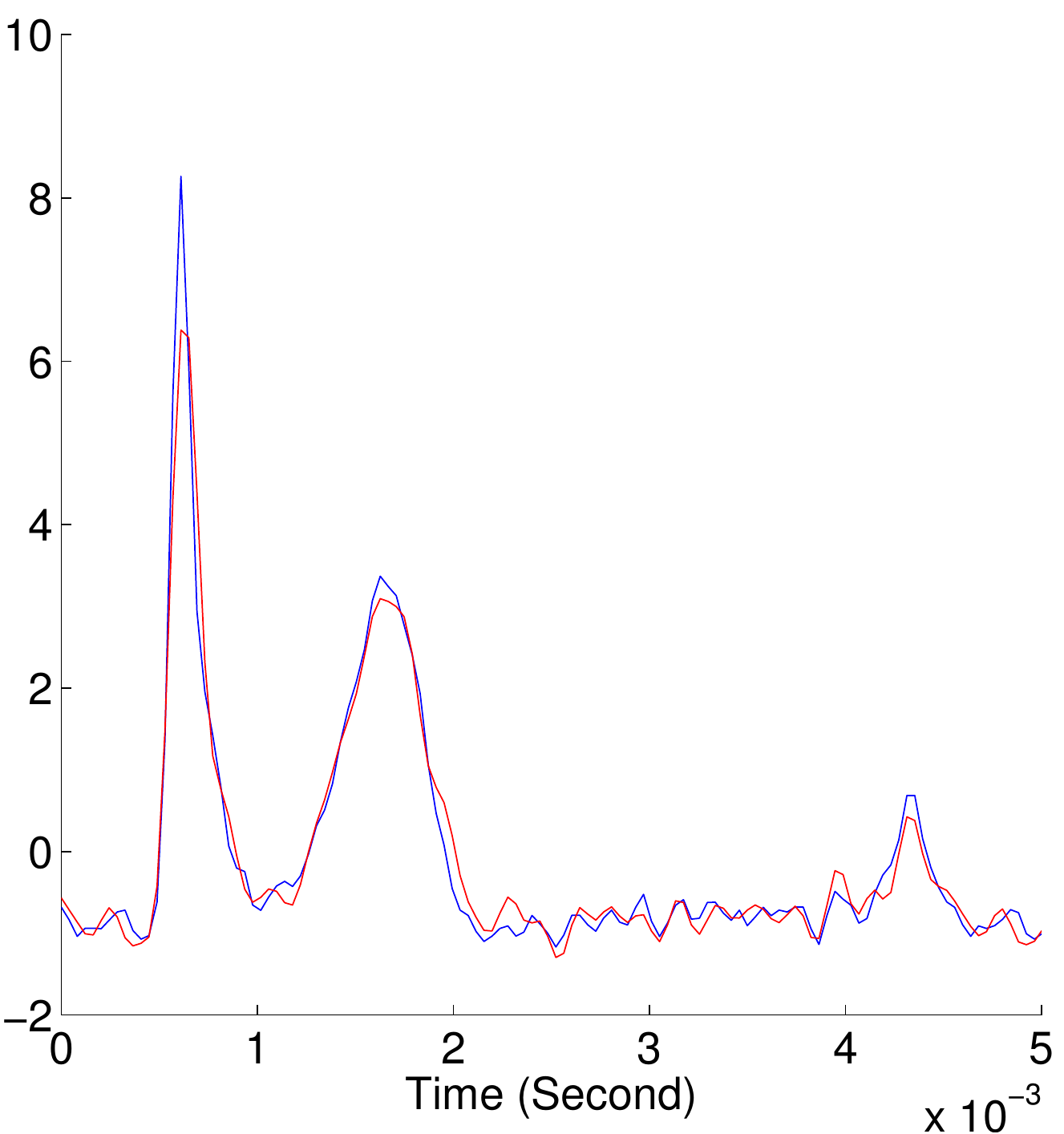}
    \end{tabular}
  \end{center}
  \caption{Top: A superposition of two synthetic ECG signals. Bottom left: The synchrosqueezed energy distribution. Bottom middle left: Real instantaneous frequencies (red) and instantaneous frequency estimates (blue). Bottom middle right: Real ECG shape function $s_5(t)$ (blue) and its estimate (red). Bottom right: Real ECG shape function $s_6(t)$ (blue) and its estimate (red). }
\label{fig:s5s6}
\end{figure}

\begin{figure}
  \begin{center}
    \begin{tabular}{c}
      \includegraphics[height=1.6in]{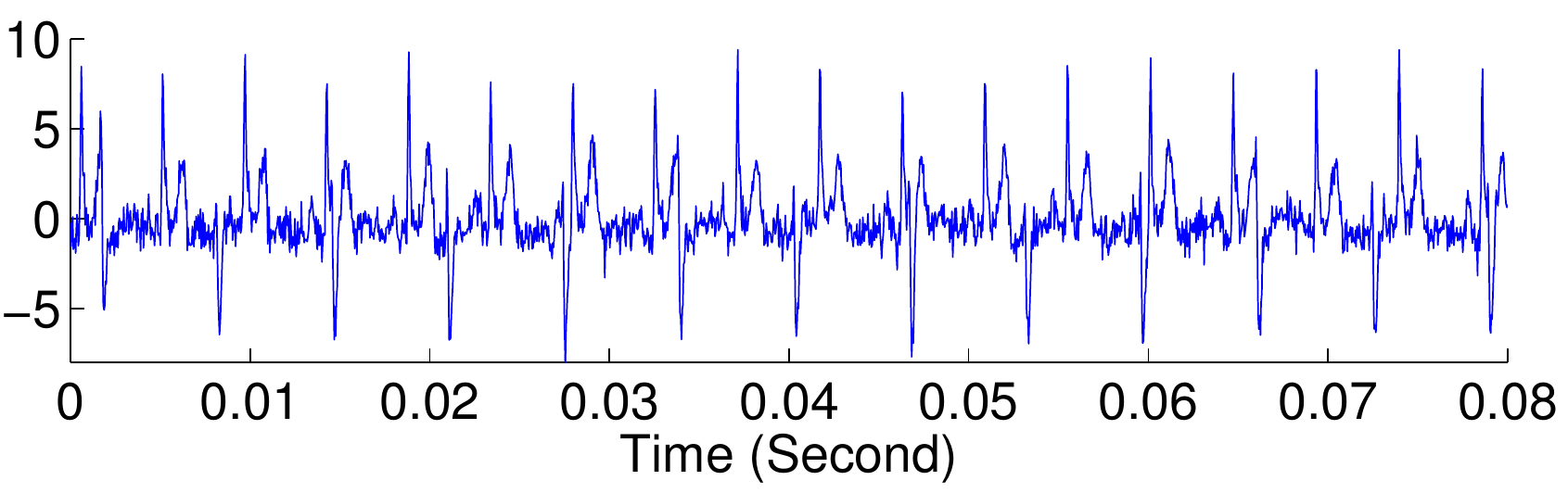}  
    \end{tabular}
    \begin{tabular}{cccc}
     \includegraphics[height=1.3in]{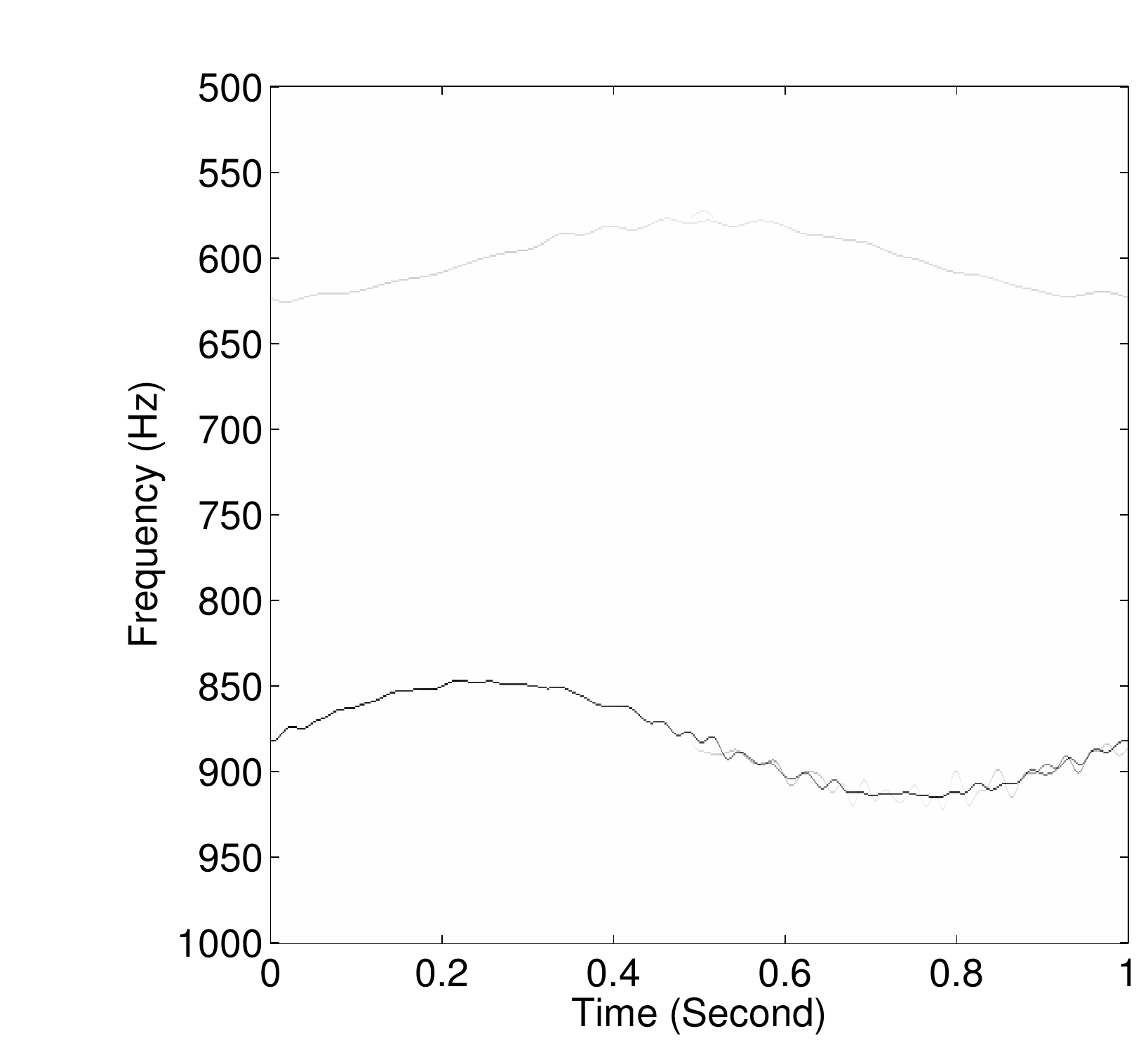}&           \includegraphics[height=1.3in]{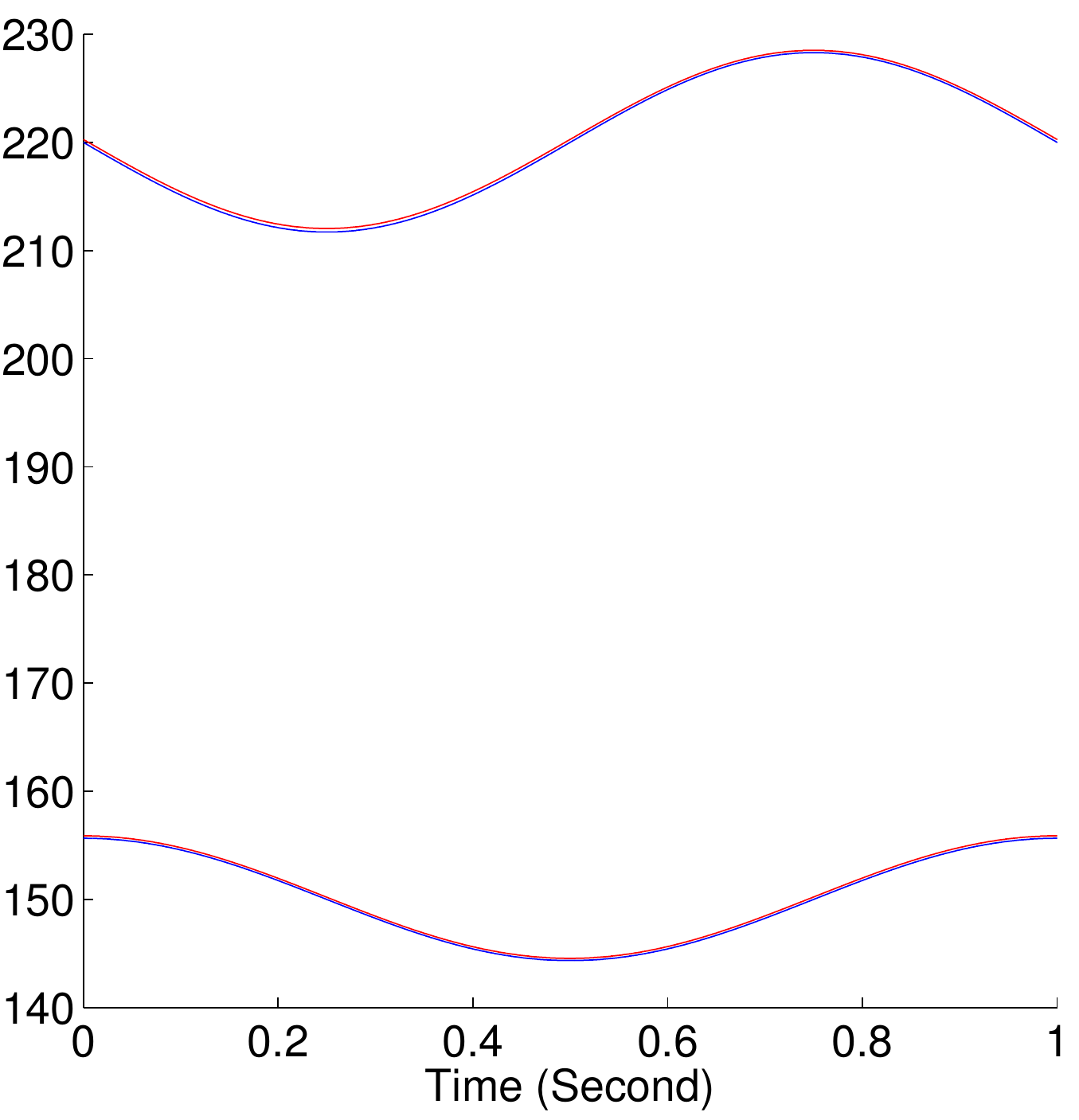}
&\includegraphics[height=1.3in]{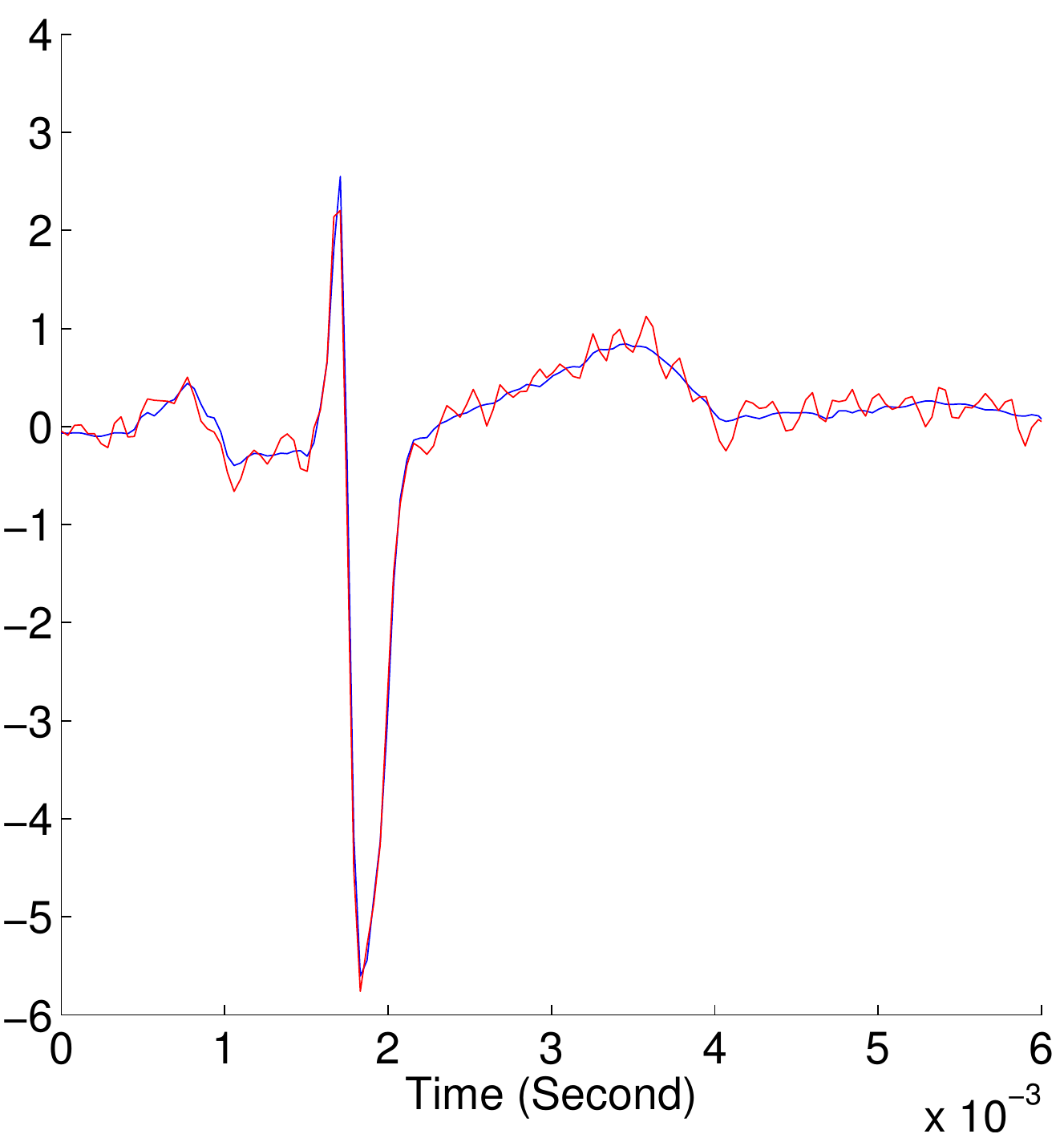} & \includegraphics[height=1.3in]{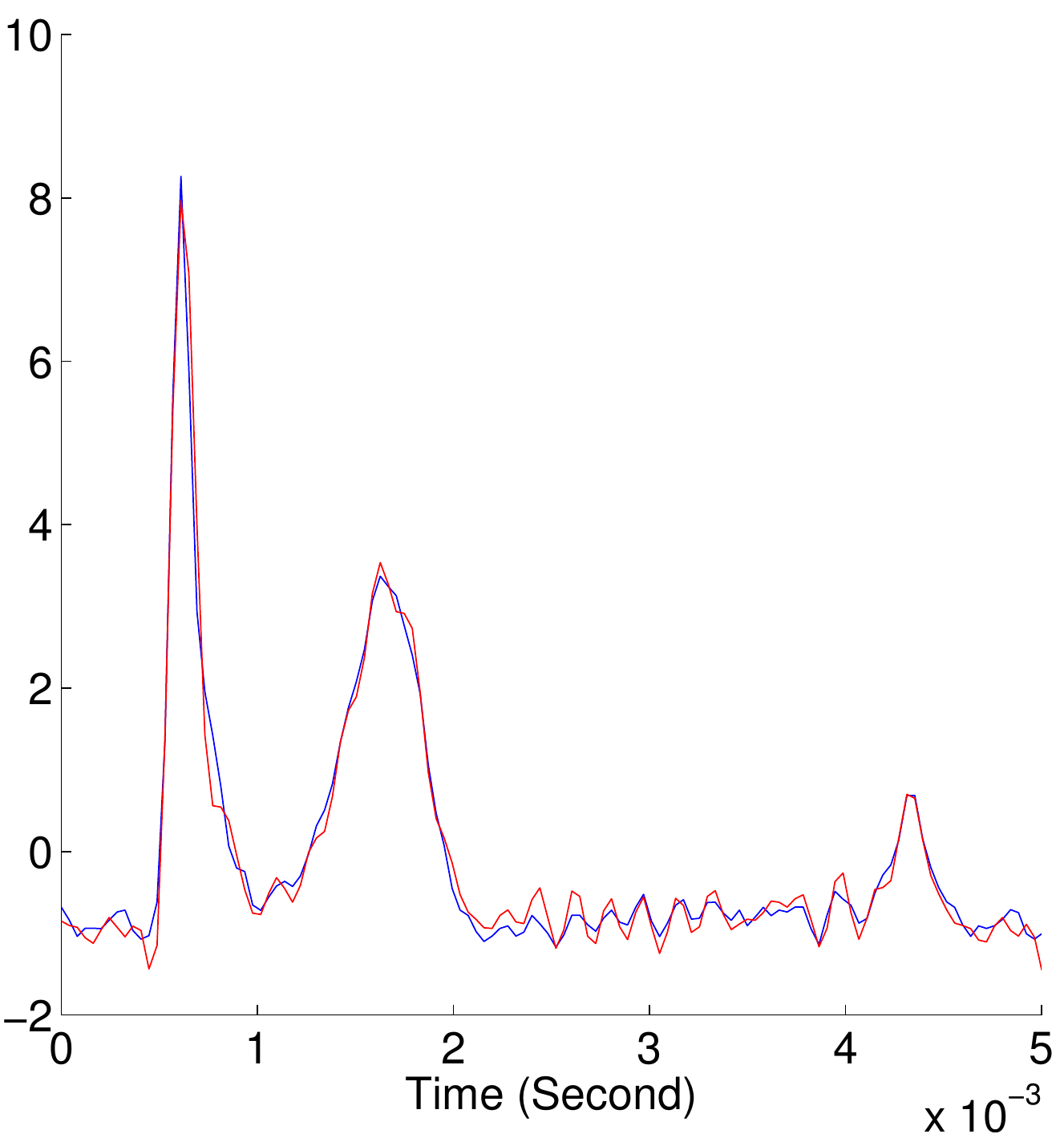}
    \end{tabular}
  \end{center}
  \caption{Top: A noisy superposition of two synthetic ECG signals with $\SNR=0$. Bottom left: The synchrosqueezed energy distribution. Bottom middle left: Real instantaneous frequencies (red) and instantaneous frequency estimates (blue). Bottom middle right: Real ECG shape function $s_5(t)$ (blue) and its estimate (red). Bottom right: Real ECG shape function $s_6(t)$ (blue) and its estimate (red). }
\label{fig:s5s6ns}
\end{figure}

\textbf{Example $6$:} Let us revisit the example shown in Figure \ref{fig:CO2wavelet} in the introduction. The original data $f_0(t)$ has a slowly growing trend linear in time. Suppose $f_r(t)$ is the linear regression of $f_0(t)$ and let $f(t)=f_0(t)-f_r(t)$. The synchrosqueezed energy distribution of $f(t)$ shown in Figure \ref{fig:CO2wp} left has three essential supports corresponding to three wave-like components. By weighting the locations of theses supports, we obtain the instantaneous frequency estimates of each component as shown in Figure \ref{fig:CO2wp}. According to the evolutive pattern of the intrinsic frequencies, there are only two general modes contained in the superposition. The curve classification step in the GMDWP method automatically groups the annual estimate and the semiannual estimate together.  Hence, the decomposition result of the GMDWP method contains a general mode which is the sum of the annual cycle and the semiannual cycle shown in Figure \ref{fig:CO2wpdec}. Because of the low frequency of the third term, it is reasonable to combine it with $f_r(t)$ to obtain a slowly varying growing trend shown in Figure \ref{fig:CO2wpdec}.

\begin{figure}
  \begin{center}
    \begin{tabular}{cccc}
\includegraphics[height=1.3in]{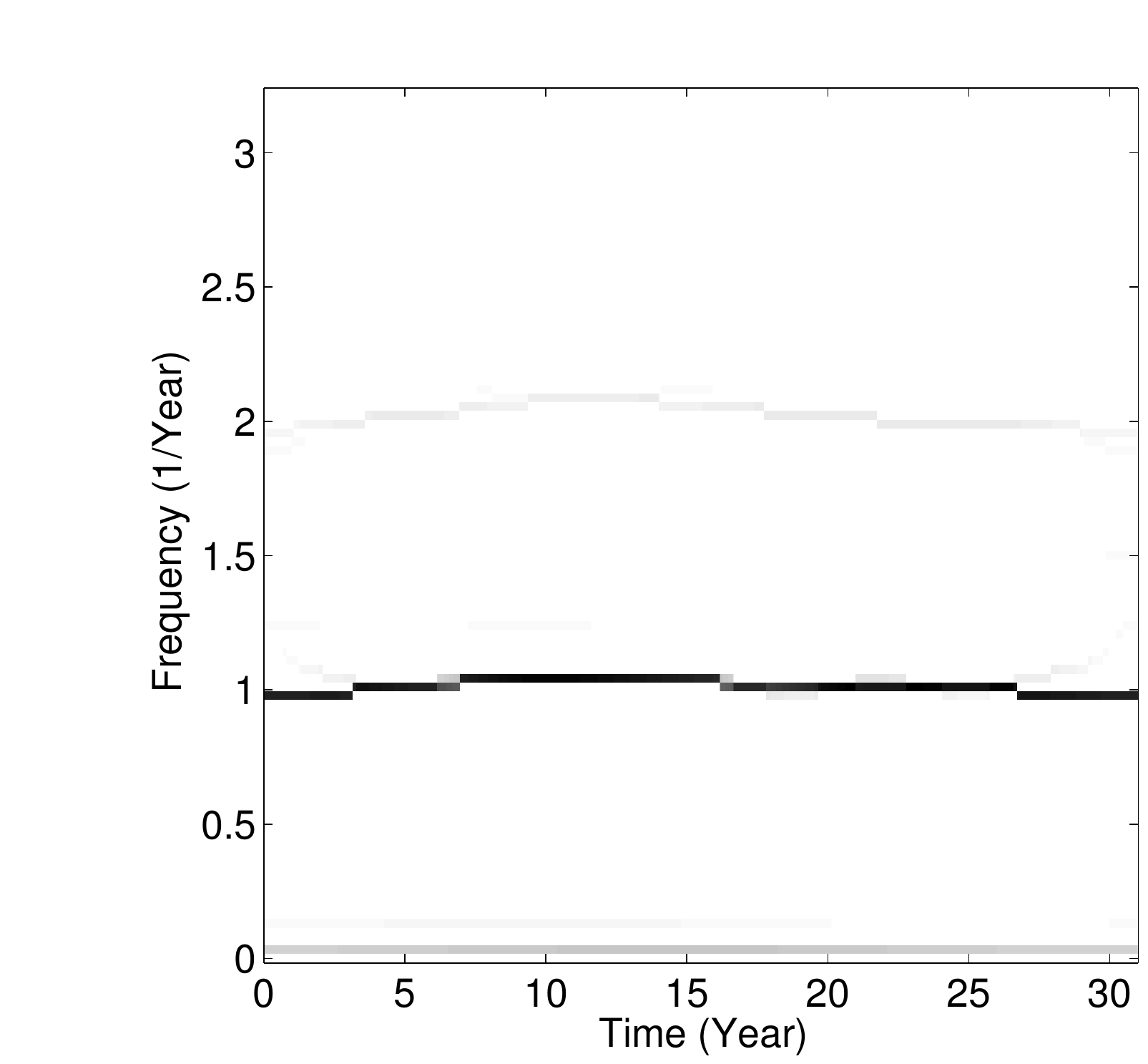}
& \includegraphics[height=1.3in]{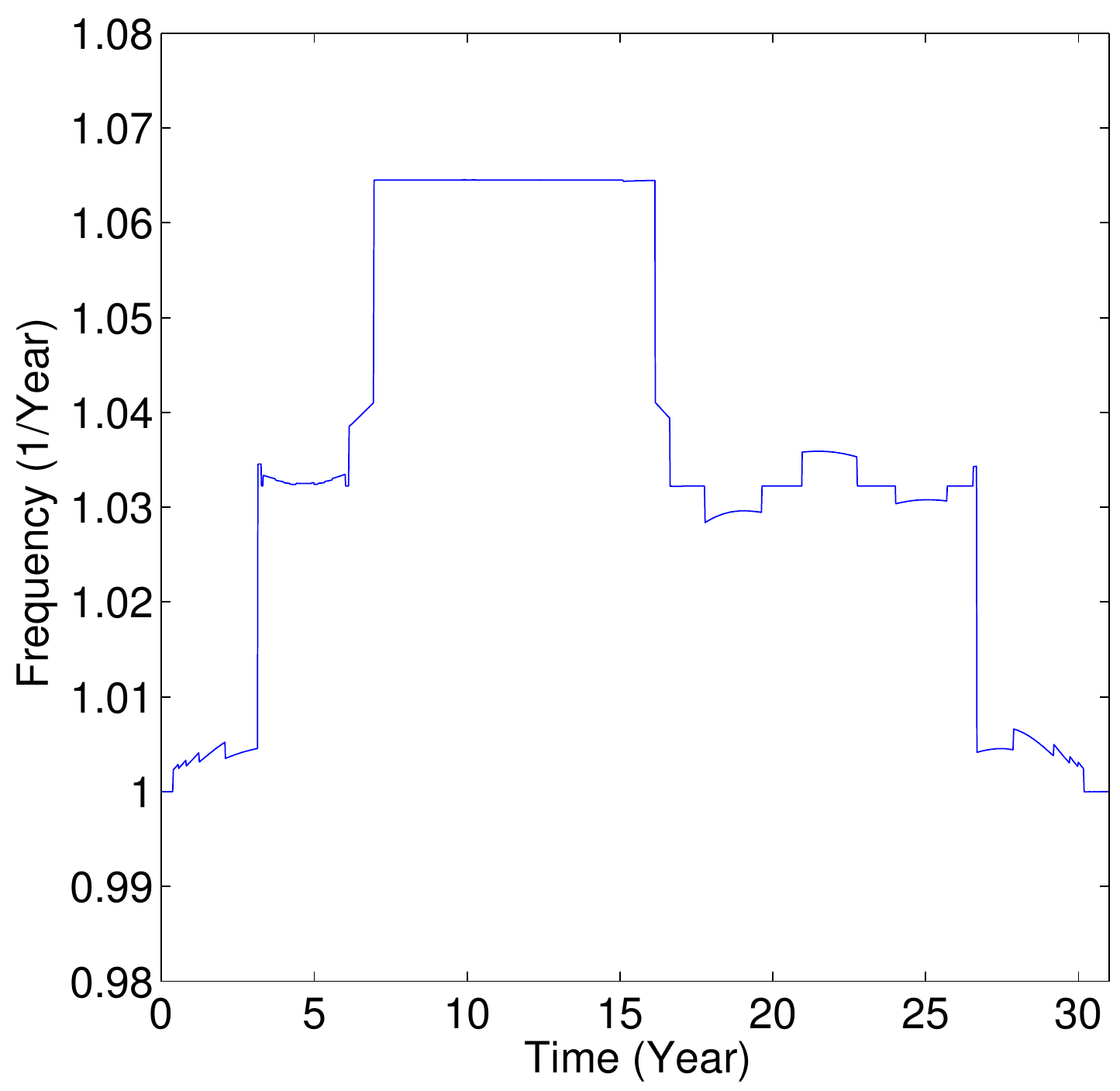}
&\includegraphics[height=1.3in]{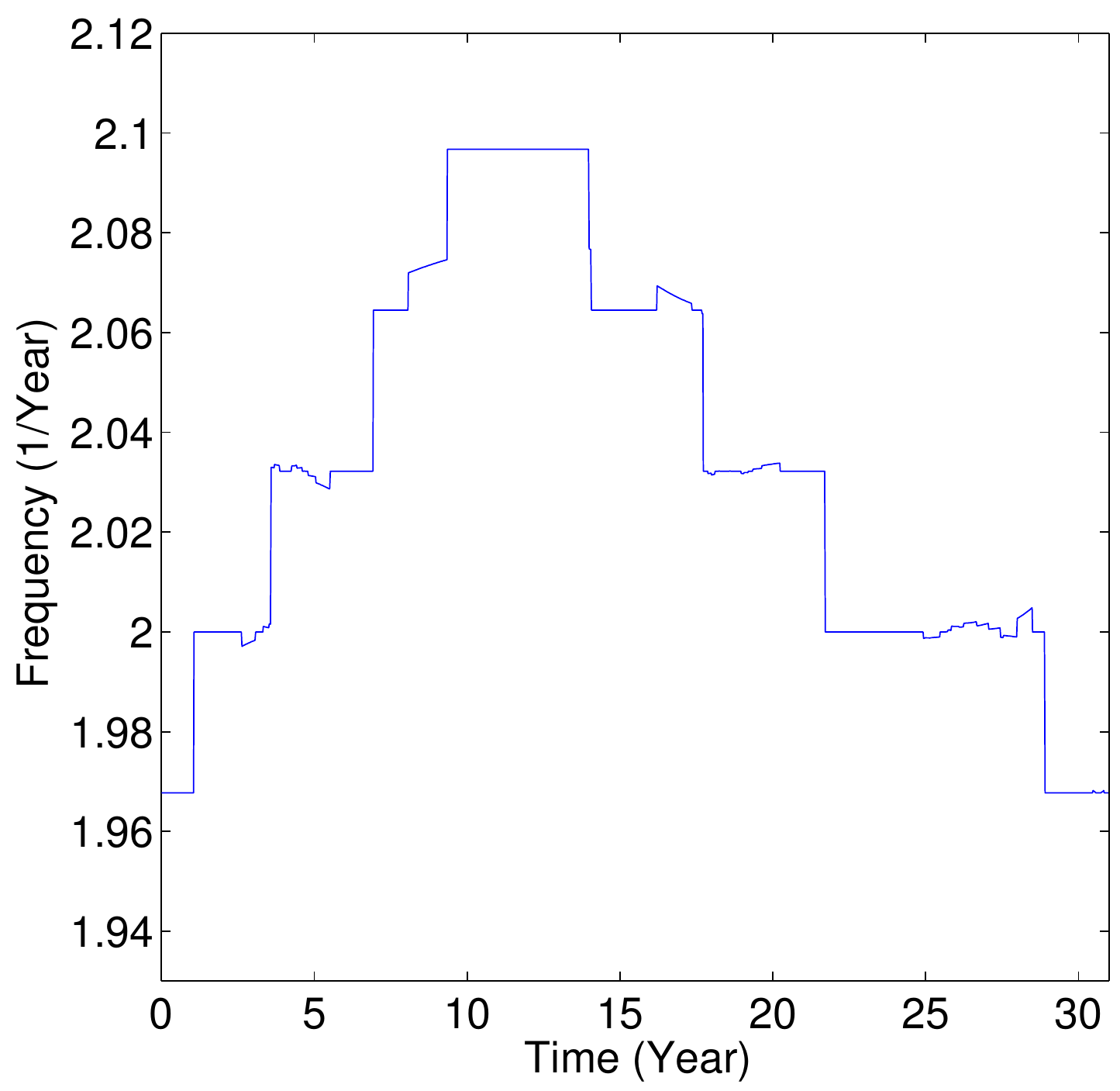}
&\includegraphics[height=1.3in]{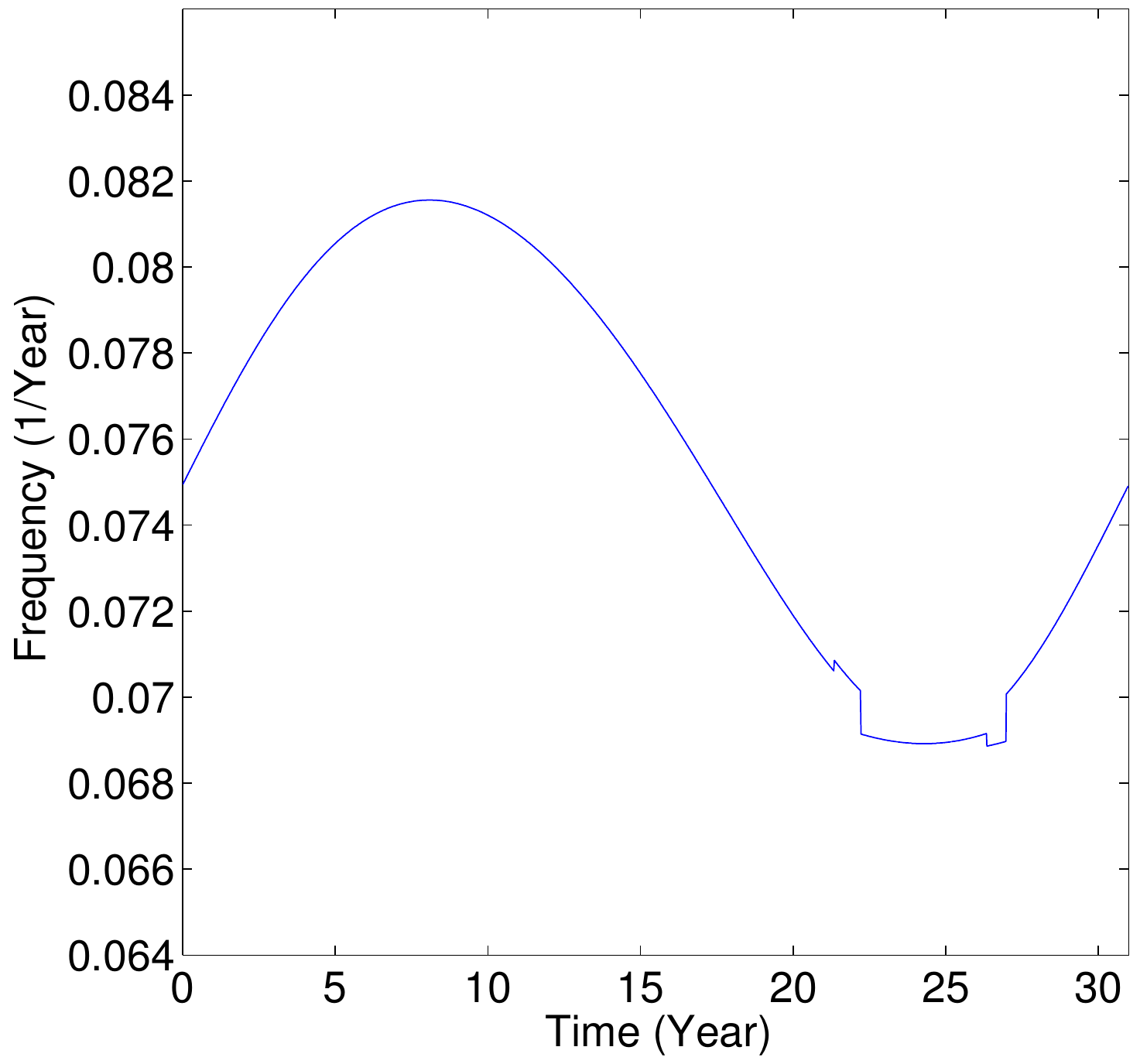}
\end{tabular}
  \end{center}
  \caption{Left: $T_f(a,b)$ of the $31$ years CO$_2$ concentration data. Middle left: The instantaneous frequency of the annual cycle. Middle right: The instantaneous frequency of the semiannual cycle. Right: the instantaneous frequency of the low frequency cycle. The curve classification algorithm groups the annual and semiannual cycle together.}
\label{fig:CO2wp}
\end{figure}

\begin{figure}
  \begin{center}
    \begin{tabular}{c}
      \includegraphics[height=2.8in]{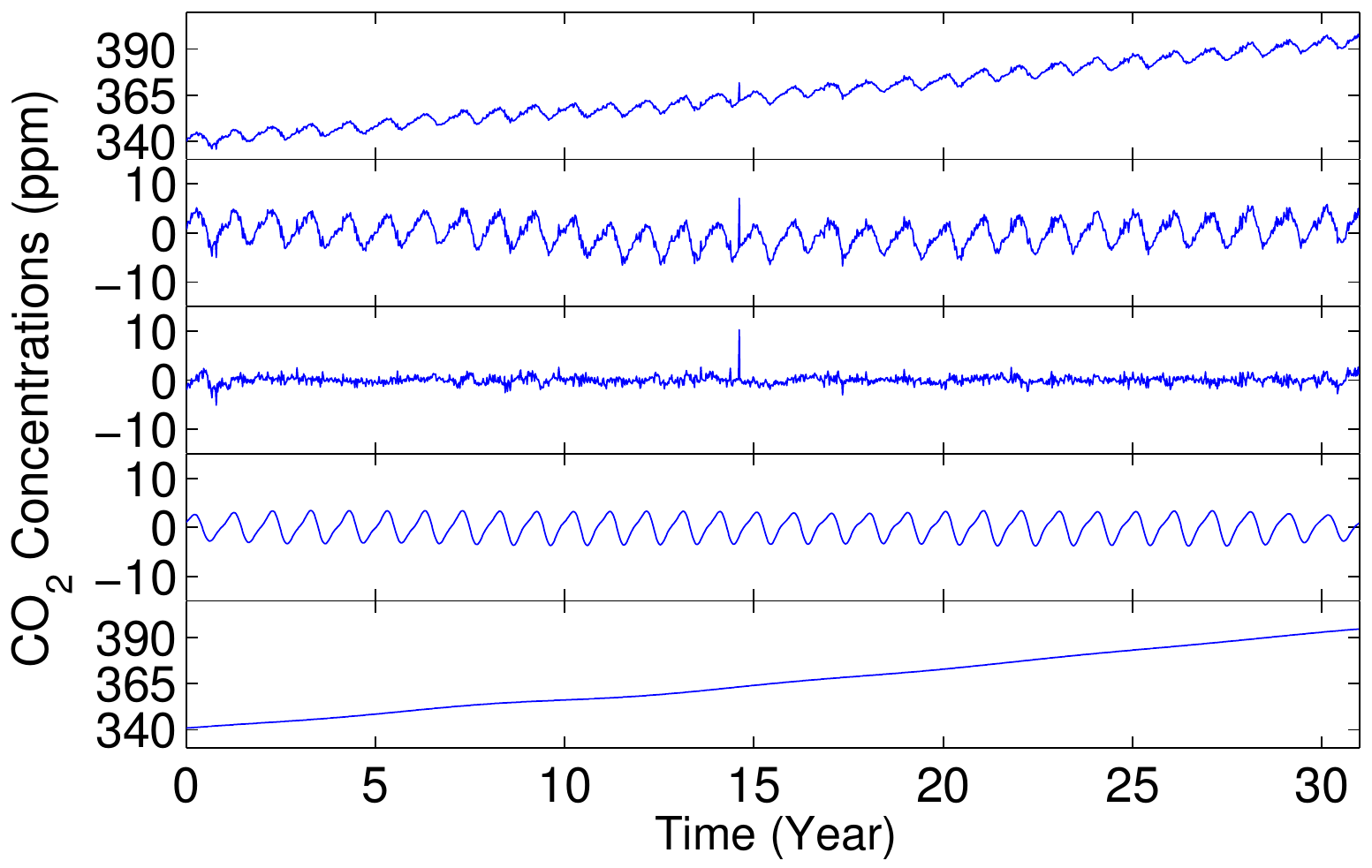}  
    \end{tabular}
  \end{center}
  \caption{Decomposition results of the $31$ years CO$_2$ concentration data provided by the GMDWP method. Top: The original data $f_0(t)$. Second row: The signal $f(t)=f_0(t)-f_r(t)$. Third row: The remaining noise term. Forth row: The annual general mode. Last row: The slowly growing mode, which is the sum of $f_r(t)$ and the low frequency component.}
\label{fig:CO2wpdec}
\end{figure}

\section{Conclusion}
\label{sec:conclusion}
This paper proposed the 1D synchrosqueezed wave packet transform (SSWPT) and the diffeomorphism based spectral analysis (DSA) method to solve the general mode decomposition problem under a weak well-separation condition and a well-different condition. These algorithms combine several ideas, namely (i) time-frequency transforms with better resolutions, (ii) a curve classification method, (iii) instantaneous information estimates, (iv) diffeomorphisms and (v) the (short time) Fourier transform. 

In particular, item (i) allows us to estimate the instantaneous information of high frequency Fourier expansion terms of general modes; items (ii) and (iii) classify the extracted Fourier expansion terms and provide the instantaneous information of general modes for the DSA method; item (iv) linearizes the phase functions of general modes so that item (v) is able to identify the spectra of general shape functions. As a consequence of the above steps, the general modes are reconstructed by adding up their Fourier expansion terms one-by-one.

There are many future directions for the general mode decomposition problem. The most important work is to estimate the instantaneous information of each general mode without any well-separation condition. Given the instantaneous information, the DSA method is able to decompose the general superposition accurately. Another work of importance is the rigorous noise analysis of these methods. Although numerical results have shown robustness against random Gaussian noise, theoretical analysis is still an open problem. The robustness properties of synchrosqueezed wavelet transforms have been analyzed in \cite{SSRobust} recently. Similar conclusions may be true for the methods proposed in this paper. It is also of interest to study other kinds of noise and to explore the effects of noise on the reconstruction. Finally, it would be appealing to weaken the well-different condition of phase functions in Theorem \ref{thm:main3} and to classify the class of well-different phase functions.


{\bf Acknowledgments.} H.Y. was partially supported by NSF grant
CDI-1027952. H.Y. thanks Lexing Ying for comments on the
manuscript.

\bibliographystyle{abbrv} \bibliography{ref}

\end{document}